\documentclass{article}

\usepackage{blindtext}
\usepackage{subfiles} 

\usepackage{packages/commonpackage}
\usepackage{packages/package_for_list_of_content}
\usepackage{packages/package_to_delete}
\usepackage{packages/macro}

\newcommand{\bG}{\mathbf{G}}

\numberwithin{equation}{section}
\theoremstyle{plain}
\newtheorem{theorem}{Theorem}[section]

\newtheorem{proposition}[theorem]{Proposition}
\newtheorem{corollary}[theorem]{Corollary}

\newtheorem{lemma}[theorem]{Lemma}
\newtheorem{definition}[theorem]{Definition}
\newtheorem{remark}[theorem]{Remark}
\newtheorem{example}[theorem]{Example}
\newtheorem{example*}{Example}
\newtheorem{claim}[theorem]{Claim}

\newtheorem{informaltheorem}[theorem]{Informal Theorem}
\newtheorem{motivatingquestion}[theorem]{Motivating Question}

\title{Curved Kakeya problems and \\the projective geometry of paths}
\date{}
\author{Shaoming Guo, Larry Guth, Arian Nadjimzadah,\\
Minxing Shen and Ruixiang Zhang\\[1em]
\normalsize With an appendix written by Terence Tao}

\begin{document}

\maketitle

\begin{abstract}
    We introduce a general framework for curved Kakeya problems in $\R^n$, encompassing those arising from H\"ormander-type oscillatory integrals. Every family of curves determines a spray geometry, which allows us to use the projective geometry of paths in the study of curved Kakeya problems. We focus on the two extremes of the ``best'' and ``worst'' possible behaviors of curved Kakeya sets. 

    We characterize when the incidence structure underlying Wolff's hairbrush argument persists. In particular,  
    we prove that the existence of many totally geodesic surfaces, as required by Wolff's hairbrush argument, is equivalent to projective flatness of the associated spray.  Within this projectively flat class, Bourgain’s condition provides a clean dichotomy: when it holds, the family is direction-equivalent to a Bochner–Riesz type family of lines and satisfies the Katz–Wolff condition, and thus the Wang--Zahl result is applicable; when it fails, every totally geodesic surface supports a two-dimensional Kakeya set.

    We also show that under an extra semi-algebraic assumption, a family of curves in $\R^3$ admits a curved Kakeya set of Hausdorff dimension $2$ if and only if it admits a curved Kakeya set contained in a surface. Equivalently if this compression is absent, every associated curved Kakeya set has dimension strictly greater than $2$.

\end{abstract}

\tableofcontents


\section{Introduction}

\subsection{Informal introduction}

We begin with an informal discussion of the main results, postponing precise definitions until later. A curved Kakeya problem starts with a $2(n-1)$-parameter family of curve segments $\ell_{\xi,w}$ in $\R^n$, where $\xi\in \R^{n-1}$ records the ``direction'' and $w \in \R^{n-1}$ records the ``position''. These curves satisfy a non-degeneracy condition so that locally each point and each tangent direction determine a unique curve. The corresponding curved Kakeya set contains a curve $\ell_{\xi,w(\xi)}$ for each direction $\xi$. The key model case is when the curves $\ell_{\xi,w}$ are (segments of) the lines $(w,0) + \R(\xi,1)$, and the corresponding curved Kakeya set is a classical Kakeya set. 

\medskip

\noindent \textbf{Overview of curved Kakeya sets:}

In $\R^2$, curved Kakeya sets are known to have full Hausdorff dimension 2. It was originally discovered by Bourgain \cite{Bou91}, in the more special case of curves arising from H\"ormander's oscillatory integral problem, that the situation in $\R^3$ is significantly more subtle, and finer geometric properties of the family $\{\ell_{\xi,w}\}$ plays a major role. There can even be curved Kakeya sets of Hausdorff dimension 2, and the known examples are contained in a surface. 

Curved Kakeya maximal function estimates tend to fail generically too, though usually not as severely as having dimension 2. The still--mysterious \emph{Bourgain's condition} describes the class that is not covered by the generic counterexample construction. It is conjectured that this condition is necessary and sufficient for optimal curved Kakeya maximal function estimates \cite{nadjimzadah2026bourgainsconditionstickykakeya, GWZ24}.

\medskip

\noindent \textbf{Two motivating questions:}

This paper arose out of the following two fundamental question about Kakeya sets of curves, roughly concerning their best possible and worst possible behaviors. 
\begin{motivatingquestion}[Best-case behavior]\label{question: best case}
    Can the techniques that led to the resolution of the Kakeya set conjecture in $\R^3$ be applied to Kakeya sets of curves? As a first step, what about Wolff's hairbrush argument in $\R^3$? For which families of curves can such arguments go through? 
\end{motivatingquestion}

\begin{motivatingquestion}[Worst-case behavior]\label{question: worst case}
    When can a family of curves in $\R^3$ admit a two-dimensional curved Kakeya set? Is compression of a full two-dimensional family of directions into a surface the only mechanism by which this can occur?
\end{motivatingquestion}

\noindent \textbf{The geometric input to Wolff's hairbrush argument:}

Let us explore these questions starting with Question \ref{question: best case}. Wolff's hairbrush argument shows that Kakeya sets in $\R^3$ have Hausdorff dimension at least $5/2$ \cite{Wol95}. In the argument, one essentially reduces to studying a hairbrush: a set of lines passing through a fixed stem. There are two important ingredients needed to make the argument work. 

First, the lines must organize into essentially disjoint planes around the stem---this does not require one to distinguish the direction parameter $\xi$. Another way to say this is that for each pair of intersecting curves, the set of curves intersecting the first two are contained in a common surface. For generic curves one would expect a volume instead of a surface. We call this the \emph{totally geodesic} condition and it is formalized in Definitions \ref{250510defi1_4}, \ref{250801definition1_16}. 

Second, there can't be too many lines inside each plane (this is known as the Wolff axioms)---the separation of the direction parameters $\xi$ guarantees that this holds. Nets Katz had another way of running the hairbrush argument which appeared in Wisewell's work \cite{Wis05}: for each pair of intersecting curves, the set of direction parameters corresponding to the curves intersecting the first two is a curve. We call this the \emph{Katz-Wolff} condition and it is formalized in Definitions \ref{250510defi1_3}, \ref{250510defi1_1}. For a generic family, one would expect these direction parameters to fill a two-dimensional region rather than a curve. 

We should also note here that the resolution of the Kakeya conjecture of Wang--Zahl makes crucial use of the fact that tubes organize into rectangular prisms, which is closely connected to the totally geodesic condition. 
Bourgain's condition unfortunately does not guarantee the Katz-Wolff condition or even the totally geodesic condition, in light of an example from \cite{nadjimzadah2026bourgainsconditionstickykakeya}. 

\medskip 

\noindent \textbf{Applying methods from projective spray geometry:}

To understand how the conditions above interact, what they reduce to, and to provide a foundation for a more systematic study of curved Kakeya problems, we establish a correspondence between curve families $\{\ell_{\xi,w}\}$ and projective spray spaces. More precisely, after forgetting the distinguished direction labels, a family $\{\ell_{\xi,v}\}$ forms a local system of paths and hence arises as the family of unparameterized geodesics of an associated spray; this is our coarsest level of geometry. The curved Kakeya problem contains additional data recording which curves carry each direction parameter $\xi$; this distinction will be essential below. Proposition \ref{prop: Kakeya to spray space} makes this correspondence precise. It allows us to use projective invariants of sprays, including the Douglas and Weyl curvatures, to detect when the underlying path geometry is locally projectively flat. Our first main theorem uses this connection to relate the conditions above, and suggests that without a serious new idea, the answer to the last question in Question \ref{question: best case} is only classical Kakeya lines. 

\begin{informaltheorem}[Formally Theorem \ref{thm: main}]
Let $\{\ell_{\xi,w}\}$ be a non-degenerate family of curves in $\R^3$.
Then
\[
\begin{aligned}
\text{\rm totally geodesic}
&\quad\Longleftrightarrow\quad
\substack{
\text{\rm after a local diffeomorphism of $\R^3$, }\\
\text{\rm the curves become straight lines};
}
\\[1.5em]
\text{\rm totally geodesic}
\;+\;
\text{\rm Bourgain's condition}
&\quad\Longleftrightarrow\quad
\text{\rm Katz--Wolff}
\\[-0.1em]
&\quad\Longleftrightarrow\quad
\substack{
\text{\rm direction-equivalent to}\\
\text{\rm a Bochner--Riesz-type family}\\[-0.1em]
\ell_{\xi,w}(t)
   =(\xi,\phi(\xi))-t(w,1);
}
\\[1.5em]
\text{\rm totally geodesic}
\;+\;
\neg\,\text{\rm Bourgain's condition}
&\quad\Longleftrightarrow\quad
\underbrace{
\text{\rm extra-worst compression}
}_{\text{\rm as in Bourgain's example \eqref{250801e1_8}}}
\\[-0.1em]
&\quad\overset{\rm def}{\Longleftrightarrow}\quad
\substack{
\text{\rm every totally geodesic surface supports}\\
\text{\rm a two-dimensional curved Kakeya set.}
}
\end{aligned}
\]
In particular, within the totally geodesic---equivalently, projectively
flat---class, Bourgain's condition distinguishes two opposite possibilities:
the Katz--Wolff/Bochner--Riesz regime and the extra-worst-compression regime. 
\end{informaltheorem}

We now turn to Question \ref{question: worst case}, which we resolve in the semi-algebraic case as our second main theorem. The proof uses tools from semi-algebraic geometry. 
\begin{informaltheorem}[Formally Theorem \ref{260615theorem1_30}]
    Let $\{\ell_{\xi,w}\}$ be a non-degenerate family of curves in $\R^3$, and furthermore the curves have a parameterization $(\xi,w,t) \mapsto \ell_{\xi,w}(t)$ which is semi-algebraic. Then:
    \[
\begin{array}{c}
\text{The family admits a curved Kakeya set of Hausdorff dimension $2$.}
\\[0.5em]
\Updownarrow
\\[0.5em]
\text{The family admits a curved Kakeya set contained in a two-dimensional surface.}
\end{array}
\]
    Moreover, if these conditions fail, then there exists $\kappa>0$, depending only on the family, such that every associated curved Kakeya set
$K$ satisfies
\[
    \dim_{\mathrm H} K\geq 2+\kappa.
\]
\end{informaltheorem}

In this paper, we also include an appendix by Terence Tao which makes progress on understanding Bourgain's condition. The goal is to provide a coordinate-invariant definition that uses natural objects in differential geometry, much in the spirit of our introduction of spray geometry. This definition gives a convenient way to generate all phase functions satisfying Bourgain's condition, including some interesting examples failing the Katz-Wolff condition.

\bigskip 

\noindent {\bf Notation.} We list some notation that will be used throughout the paper. 

\begin{enumerate}
\item For $\epsilon>0$ and $\mathbf{x} \in \mathbb{R}^n$, we let $\mathbb{B}_\epsilon^n(\mathbf{x})$ denote the ball of radius $\epsilon$ in $\mathbb{R}^n$ centered at $\mathbf{x}$. If $\epsilon=1$, we often abbreviate $\mathbb{B}_1^n(\mathbf{x})$ to $\mathbb{B}^n(\mathbf{x})$; if $\mathbf{x}=0$, then we often abbreviate $\mathbb{B}_\epsilon^n(\mathbf{x})$ to $\mathbb{B}_\epsilon^n$.

\item For a set $E$, we will use $\mathbbm{1}_E$ to denote its indicator function. 

\item For a set $E \subset \mathbb{R}^n$ and $\delta>0$, we use $\mathcal{N}_\delta(E)$ to denote the $\delta$-neighborhood of $E$. 

\item For a set $E\subset \R^n$, we use $\mathcal{L}^n(E)$ to denote its Lebesgue measure. 

\item We use $f\equiv 0$ to mean that the function $f$ vanishes everywhere. 

\item Throughout the paper, we always assume that the phase functions $\phi$ and the  maps $\Phi$ defining curved Kakeya sets are all analytic, unless otherwise stated. 
\end{enumerate}

\subsection{H\"ormander's problems and related Kakeya problems: A brief review}\label{250708subsection1_1}

In \cite{Hor73}, H\"ormander suggested to study the following  oscillatory integral operators
\begin{equation}\label{230717e1_1}
T^{(\phi)}_N f(x, t):=\int_{\R^{n-1}} e^{i N \phi(x, t; y)} a(x, t; y) f(y) \mathrm{d} y,
\end{equation}
which are often referred to as \underline{H\"ormander-type} oscillatory integral operators, or \underline{H\"ormander's operators} for simplicity. 
Here 
\begin{equation}
x\in \R^{n-1}, t\in \R, y\in \R^{n-1}
\end{equation}
 and $N\in \R$ is a large real number. To simplify notation, we often write $\bfx=(x, t)$. Moreover, $\phi(x, t; y)$ is a smooth function, and  $a(x, t; y)$ is a smooth function supported in a bounded open neighborhood of the origin. The bounds that we are interested in proving are of the form 
\begin{equation}\label{250708e1_2}
\|T^{(\phi)}_N f\|_{L^p(\R^n)}\lesim_{\phi, a, p, \epsilon} N^{-\frac{n}{p}+\epsilon} \|f\|_{L^p(\R^{n-1})},
\end{equation}
for every $\epsilon>0$ and every $N\ge 1$, and for a range of $p$ that is as large as possible. Here the implicit constant is allowed to depend on $\phi, a, p$ and $\epsilon.$\\

If we take 
\begin{equation}\label{250708e1_3}
\phi(x, t; y)= \inn{x}{y}+ t|y|^2,
\end{equation}
then we obtain the classical Fourier extension operator. \\

{\bf Fourier restriction conjecture.} Let $n\ge 3$. Let $\phi$ be given by \eqref{250708e1_3}. The estimate \eqref{250708e1_2} holds for every 
\begin{equation}\label{250708e1_4}
p\ge \frac{2n}{n-1}.
\end{equation}

In \cite{Hor73}, H\"ormander asked whether the estimate \eqref{250708e1_2} still holds for the same range \eqref{250708e1_4} when $\phi$ is a ``small" perturbation of \eqref{250708e1_3}. To make this precise, H\"ormander introduced the following non-degenerate condition.   

\begin{definition}[H\"ormander's condition for oscillatory integrals, \cite{Hor73}]\label{250726defi1_1}
We say that $\phi$ satisfies \underline{H\"ormander's condition} if  on the support of $a(x, t; y)$, it holds that 
\begin{enumerate}
\item[(H1)] $\rank \nabla_{x}\nabla_y \phi(x, t; y)=n-1;$ 
\item[(H2)] if we define 
\begin{equation}
G_0(\bfx; y):=\partial_{y_1} \nabla_{\bfx} \phi(\bfx; y)\wedge\dots\wedge\partial_{y_{n-1}} \nabla_{\bfx} \phi(\bfx; y),
\end{equation}
then 
\begin{equation}
\left.\operatorname{det} \nabla_{y}^2\left\langle\nabla_{\mathbf{x}} \phi(\mathbf{x} ; y), G_0\left(\mathbf{x} ; y_0\right)\right\rangle\right|_{y=y_0} \neq 0.
\end{equation}
\end{enumerate}
\end{definition}
H\"ormander \cite{Hor73} and Bourgain \cite{Bou91} observed that when studying H\"ormander's problem for phase functions $\phi$ satisfying Definition \ref{250726defi1_1}, we can always apply elementary changes of coordinates so that phase functions can be written in the form 
\begin{equation}\label{250726e1_7}
\phi(\mathbf{x} ; y)=x \cdot y+t\langle y, A y\rangle+O\left(|t||y|^3+|\mathbf{x}|^2|y|^2\right),
\end{equation}
where $A$ is an $(n-1)\times (n-1)$ non-degenerate matrix, and all $x, t, y$ take values near their origins. In this sense, we say that \eqref{250726e1_7} is a small perturbation of the phase function \eqref{250708e1_3} for the Fourier extension operator. The form \eqref{250726e1_7} is called a \underline{normal form} for $\phi$ at the origin. If the matrix $A$ is positive-definite, we say that the phase function $\phi$ is \underline{elliptic}. \\

In \cite{Hor73}, H\"ormander asked whether for every phase function $\phi$ satisfying H\"ormander's condition in Definition \ref{250726defi1_1}, the estimate 
\eqref{250708e1_2} still holds for the same range \eqref{250708e1_4}. \\

Let us very briefly discuss the history of H\"ormander's problem, and we refer to \cite{GHI19}, \cite{HI22}, \cite{GWZ24}, \cite{DGGZ24}, \cite{CGGHIW24} and \cite{nadjimzadah2025newcurvedkakeyaestimates} for more detailed discussions.

We restrict our discussion to the case $n=3$, which is already interesting enough. Stein \cite{Ste84} proved that \eqref{250708e1_2} holds for every phase function satisfying H\"ormander's condition in Definition \ref{250726defi1_1} and every $p\ge 4$ if we take $n=3$. 
In \cite{Bou91}, Bourgain gave a negative answer to H\"ormander's problem by showing that if we take $n=3$ and  
\begin{equation}\label{250801e1_8}
\phi(\bfx; y)= x\cdot y+ ty_1 y_2+ \frac12 t^2 y_1^2,
\end{equation}
where $y=(y_1, y_2)$, then the estimate \eqref{250708e1_2} cannot hold for any $p<4$. Indeed, regarding negative results, Bourgain  \cite{Bou91} proved much more. In order to state Bourgain's result, let us first introduce the notion of Bourgain's condition.

\begin{definition}[Bourgain's condition, \cite{Bou91}, \cite{GWZ24}]\label{230903definition1_4}
 Let $\phi$ be a phase function satisfying H\"ormander's conditions. We say that it satisfies Bourgain's condition at $(\bfx_0; y_0)$ if 
\begin{equation}\label{230719e1_9}
\left(\left(G_0 \cdot \nabla_{\mathbf{x}}\right)^2 \nabla_{y}^2 \phi\right)\left(\mathbf{x}_0 ; y_0\right) \text { is a multiple of }\left(\left(G_0 \cdot \nabla_{\mathbf{x}}\right) \nabla_{y}^2 \phi\right)\left(\mathbf{x}_0 ; y_0\right) \text {. }
\end{equation}
The constant here is allowed to depend on $\bfx_0$ and $y_0$. 
\end{definition}

Bourgain \cite{Bou91} proved that in the case $n=3$, if the phase function $\phi$ fails Bourgain's condition in Definition \ref{230903definition1_4} at least at one point, then the estimate \eqref{250708e1_2} already cannot hold for any 
\begin{equation}\label{250726e1_10}
p< 3+ \frac{1}{39}.
\end{equation}
Such a result was generalized to a wider range than \eqref{250726e1_10} and to every $n\ge 3$ in \cite{GWZ24}. Moreover,  in \cite{GWZ24} it is proven that if Bourgain's condition in Definition \ref{230903definition1_4} holds at every point, then the estimate \eqref{250708e1_2} holds for a wide range of $p$. We refer to \cite{GWZ24} for the precise statement. \\

So far we have finished our discussion on H\"ormander's problem, and next we will discuss Kakeya problems associated to H\"ormander's operators. \\

Given a phase function $\phi(\bfx; y)$ satisfying H\"ormander's condition in Definition \ref{250726defi1_1}, we pick $\epsilon_{\phi}>0$ to be a sufficiently small constant depending on $\phi$. The following several notions are taken from \cite{DGGZ24}, which were further taken from Wisewell \cite{Wis05} and Bourgain \cite{Bou91}.

\begin{definition}[Characteristic curves and curved tubes, \cite{Bou91}]
Let $\phi(\bfx; y)$ be a phase function satisfying H\"ormander's condition in Definition \ref{250726defi1_1}. 
For $y\in \B^{n-1}_{\epsilon_{\phi}}$, $\bfx\in \B^{n}_{\epsilon_{\phi}}$ and $0< \delta< \epsilon_{\phi}$, define
\begin{equation}
\begin{aligned}
& \Gamma^{(\phi)}_y(\bfx):=\left\{\bfx' \in \R^n\cap \B^n_{2\epsilon_{\phi}}: \nabla_y \phi(\bfx'; y)=\nabla_y \phi(\bfx; y)\right\}, \\
& T_y^{\delta, (\phi)}(\bfx):=\left\{\bfx' \in \mathbb{R}^n\cap \B^n_{2\epsilon_{\phi}}:\left|\nabla_y \phi(\bfx'; y)-\nabla_y \phi(\bfx; y)\right|<\delta\right\}.
\end{aligned}
\end{equation}
If $\bfx$ is of the form $(\omega, 0)$, that is, the last coordinate is $0$, then we often abbreviate $\Gamma^{(\phi)}_y(\bfx)$ to $\Gamma^{(\phi)}_y(\omega)$, and abbreviate $T_y^{\delta, (\phi)}(\bfx)$ to $T_y^{\delta, (\phi)}(\omega)$. If it is clear from the context which $\phi$ is involved, we often drop $(\phi)$ from the above notations and simply write $\Gamma_y(\omega)$ and $T_y^{\delta}(\omega)$. 
\end{definition}

\begin{definition}[$\phi$-Kakeya sets, \cite{Bou91}, \cite{Wis05}]
Let $\phi(\bfx; y)$ be a phase function satisfying H\"ormander's condition in Definition \ref{250726defi1_1}.  A set $E\subset \R^n$ with $\mathcal{L}^n(E)=0$ is called a $\phi$-Kakeya set if for all $y\in \B^{n-1}_{\epsilon_{\phi}}$ there exists an $\omega\in \B^{n-1}_{\epsilon_{\phi}}$ such that $\Gamma_y(\omega)\subset E$. 
\end{definition}

\begin{definition}[$\phi$-Kakeya maximal function, \cite{Bou91}] Let $\phi(\bfx; y)$ be a phase function satisfying H\"ormander's condition in Definition \ref{250726defi1_1}. For $y\in \B^{n-1}_{\epsilon_{\phi}}$ and $\delta>0$, define 
\begin{equation}
\mathcal{K}_\delta^{(\phi)} f(y):=\sup _{\omega \in \mathbb{B}_{\epsilon_\phi}^{n-1}} \frac{1}{\mathcal{L}^n\left(T_y^{\delta,(\phi)}(\omega)\right)} \int_{T_y^{\delta,(\phi)}(\omega)}|f| .
\end{equation}
We will write $\mathcal{K}_\delta$ instead of $\mathcal{K}_\delta^{(\phi)}$ if it is clear from the context which $\phi$ we are talking about. 
\end{definition}

If we take $\phi$ to be the phase function for the Fourier extension operator, that is, 
\begin{equation}\label{250726e1_13}
\phi(x, t; y)= \inn{x}{y}+ t|y|^2,
\end{equation}
then the $\phi$-Kakeya set becomes a standard Kakeya set (for line segments), and the $\phi$-Kakeya maximal function becomes a standard Kakeya maximal function (for line segments). \\

Regarding $\phi$-Kakeya sets and $\phi$-Kakeya maximal functions, we are mostly interested in estimating the Hausdorff and Minkowski dimensions of $\phi$-Kakeya sets, and proving bounds of the form 
\begin{equation}\label{250726e1_14}
\norm{
\mathcal{K}_{\delta} f
}_{L^p(
\B^{n-1}_{\epsilon_{\phi}}
)}
\lesim_{\epsilon, \phi, p}
\delta^{1-\frac{n}{p}-\epsilon} \|f\|_{L^p(\R^n)},
\end{equation}
for all $0<\delta< \epsilon_{\phi}$ and all $\epsilon>0$, and a range of $p$ that is as large as possible. Note that \eqref{250726e1_14} is trivial at $p=1$. \\

{\bf Kakeya conjecture and Kakeya maximal conjecture.} If we take $\phi$ to be \eqref{250726e1_13}, then the associated $\phi$-Kakeya sets have Hausdorff dimension $n$ and 
 the estimate \eqref{250726e1_14} holds for all $1\le p\le n$. \\

It is proven in \cite{Bou91} and \cite{GWZ24} that if $\phi$ fails Bourgain's condition in Definition \ref{230903definition1_4} at some point, then there exists $p_{\phi}< n$ such that \eqref{250726e1_14} fails for all $p_{\phi}< p\le n$. We refer to \cite{HRZ22}, \cite{GWZ24}, \cite{DGGZ24}, \cite{CGGHIW24} and \cite{nadjimzadah2025newcurvedkakeyaestimates} for more positive results.

\subsection{Finer characterizations for H\"ormander-type oscillatory integrals: New concepts}\label{250726subsection1_2}

To motivate our new conditions for H\"ormander-type phases, let us think about the following question. Previously we have seen that Bourgain's condition in Definition \ref{230903definition1_4} is a necessary condition for the $\phi$-Kakeya maximal estimate \eqref{250726e1_14} to hold for the widest possible range $1\le p\le n$. Is it possible that Bourgain's condition is also sufficient?

If we run the polynomial partitioning algorithm (for instance repeat the argument in \cite{HRZ22}) by combining the polynomial Wolff axiom proven in \cite{GWZ24} under Bourgain's condition, then in the case $n=3$ we obtain that \eqref{250726e1_14} holds for $p=7/3$ and as a consequence we obtain that every $\phi$-Kakeya has Hausdorff dimension at least $7/3$ under the assumption that $\phi$ satisfies Bourgain's condition.

However, Wolff \cite{Wol95} already proved that the Hausdorff dimension of the standard Kakeya set is at least $5/2$, by using his hairbrush argument. It is therefore very natural to ask whether one can run Wolff's hairbrush argument under Bourgain's condition. The following notion seems to be the minimal assumption one needs to run the hairbrush argument\footnote{Here we run the version of the hairbrush argument by Katz, see for instance \cite{Wis03}}, and we will see that it is strictly stronger than Bourgain's condition, see Theorem \ref{thm: main} below.

\begin{definition}[Katz-Wolff condition for H\"ormander phases]\label{250510defi1_3}
Let $\phi(\bfx; y)$ be a phase function satisfying H\"ormander's condition. 
For $y_1, y_2\in \B^{n-1}_{\epsilon_{\phi}}, y_1\neq y_2$, $\bfx_0=(x_0, t_0)\in \B^{n}_{\epsilon_{\phi}}$, and $t_1, t_2\neq t_0$, let $x_i\in \R^{n-1}$ be the unique point such that 
\begin{equation}
(x_i, t_i)\in \Gamma^{(\phi)}_{y_i}(\bfx_0), \ i=1, 2.
\end{equation}
Moreover, let $w(t_1, t_2)\in \B^{n-1}_{\epsilon_{\phi}}$ and $y(t_1, t_2)\in \B^{n-1}_{\epsilon_{\phi}}$ be the unique (if exist) solutions  satisfying that the characteristic curve 
\begin{equation}
\Gamma^{(\phi)}_{y(t_1, t_2)}(w(t_1, t_2))
\end{equation}
 passes through the two points $(x_1, t_1)$ and $(x_2, t_2)$.  We say that $\phi$ satisfies the Katz-Wolff condition if
 \begin{equation}
        \rank[\partial_{t_1} y(t_1, t_2), \partial_{t_2} y(t_1, t_2)]< 2, \ \ \forall t_1, t_2.
    \end{equation} 
\end{definition}

Definition \ref{250510defi1_3} requires that for two intersecting curves $\Gamma_{y_1}(\bfx_0)$ and $\Gamma_{y_2}(\bfx_0)$, the direction set for the characteristic curves intersecting both of these two curves, given by 
\begin{equation}
\{y(t_1, t_2): t_1, t_2\neq t_0\},
\end{equation}
is one-dimensional.

\begin{definition}[Totally geodesic conditions for H\"ormander phases]\label{250510defi1_4}
Under the same notation as in Definition \ref{250510defi1_3}, we say that $\phi$ satisfies the totally geodesic condition if for all choices of $\bfx_0$ and $y_1\neq y_2$, the union of curves 
\begin{equation}
\bigcup_{
t_1, t_2
} \Gamma^{(\phi)}_{y(t_1, t_2)}(w(t_1, t_2))
\end{equation}
forms a two-dimensional surface. 
\end{definition}

Definition \ref{250510defi1_4} requires that for two intersecting curves $\Gamma_{y_1}(\bfx_0)$ and $\Gamma_{y_2}(\bfx_0)$, all the characteristic curves intersecting these two curves form a two-dimensional surface.  \\

Wolff's original hairbrush argument \cite{Wol95} requires Definition \ref{250510defi1_4} and a non-concentration assumption. Nets Katz later gave an alternative argument which appeared in \cite{Wis05} that requires the condition given in Definition \ref{250510defi1_3}. It is trivial to check that both of these conditions hold for the standard Kakeya problem. However, we will see that for general $\phi$-Kakeya problems, whether Definitions \ref{250510defi1_3} and \ref{250510defi1_4} hold starts to be a very interesting problem, and we will see that Definition \ref{250510defi1_3} always implies Definition \ref{250510defi1_4}, but the reverse direction is not always true (see Theorem \ref{thm: main} below). \\

Let us remark on the formulation of Definition \ref{250510defi1_3}. The way in which this definition is formulated is not ``local" and may not be easy to check. This will be addressed as part of our main theorem (Theorem \ref{thm: main} below), where we will give several equivalent characterizations for Katz-Wolff conditions. In particular, we will give a characterization similar to how Bourgain's condition is formulated in Definition \ref{230903definition1_4}.

\subsection{General curved Kakeya problems and their Bourgain's conditions}\label{250929sub1_3}

In Subsection \ref{250708subsection1_1}, we stated $\phi$-Kakeya sets and $\phi$-Kakeya maximal estimates. To introduce these objects, one always needs to start with a phase function $\phi$. In the current paper, we will introduce a more general and natural way of stating curved Kakeya problems. \\

Let $n\ge 2$ and let 
\begin{equation}
\Phi:\R^{n-1}\times\R\times\R^{n-1}\rightarrow \R^n
\end{equation}
be a smooth map. We use $(\omega, t; \xi)$ to denote the variables of $\Phi$. We always write vectors in column forms, unless otherwise specified. 
	
	\begin{definition}[$\Phi$-Kakeya sets and $\Phi$-Kakeya maximal function]\label{260309defi1_8}
		We say that a Borel set $E\subset\R^n$ with $\mathcal{L}^n(E)=0$ is a $\Phi$-Kakeya set if there exists $\varepsilon_{\Phi}>0$ such that for every $\xi\in \B_{\varepsilon_{\Phi}}^{n-1}$, there exists $\omega\in \B_{\varepsilon_{\Phi}}^{n-1}$ satisfying 
		\begin{equation}
			\Phi(\omega, t; \xi)\in E\ \text{for every}\ |t|\le \varepsilon_{\Phi}.
		\end{equation}
We define $\Phi$-Kakeya maximal functions similarly to $\phi$-Kakeya maximal functions. 
	\end{definition}

If we take  
\begin{equation}\label{250709e1_15}
\Phi(\omega, t; \xi)= (\omega+t\xi, t),
\end{equation}
then we see the standard Kakeya problem. \\

Similarly to H\"ormander's condition in Definition \ref{250726defi1_1} for H\"ormander-type oscillatory integrals, we also introduce H\"ormander's condition for $\Phi$-Kakeya sets. 
	\begin{definition}[H\"ormander's condition for $\Phi$-Kakeya sets]\label{HorCond}
		Let $\Omega\subset \R^{n-1}\times \R\times \R^{n-1}$ be an open set. 
		Let $\Phi: \Omega\to \R^n$ be a smooth map. We say that $\Phi$ satisfies H\"ormander's condition at $(\omega, t; \xi)\in \Omega$ if both of the following two conditions are satisfied: 
		\begin{enumerate}
			\item[(1)] the following $n\times n$ matrix
			\begin{equation}\label{le1.1}
				\nabla_{(\omega, t)} \Phi(\omega, t; \xi) 
			\end{equation}
			is non-degenerate; 
			\item[(2)] the following $2n\times 2n$ matrix 
			\begin{equation}\label{le1.2}
				\begin{bmatrix}
					0 & \nabla_{(\omega, t)}\Phi & \nabla_{\xi}\Phi\\
					\partial_t \Phi & \nabla_{(\omega, t)}\partial_t \Phi & \nabla_{\xi}\partial_t\Phi
				\end{bmatrix}
			\end{equation}
			is non-degenerate at $(\omega, t; \xi)$. 
		\end{enumerate}
		We say that $\Phi$ satisfies H\"ormander's condition on $\Omega$ if it satisfies H\"ormander's condition at every point in $\Omega$. 
	\end{definition}

	 It is elementary to check that  $\Phi(\omega, t; \xi)= (\omega+t\xi, t)$ satisfies this H\"ormander's condition. \\

Let $\Phi$ satisfy H\"ormander's condition in Definition \ref{HorCond}. Since $\nabla_{(\omega, t)} \Phi(\omega, t; \xi)$ is non-degenerate, it is elementary to see that we can write 
	\begin{equation}\label{250727e1_25}
		\Phi(\omega, t; \xi)=\left(X(\omega, t; \xi),t\right)
	\end{equation}
	after a change of coordinates, where 
	\begin{equation}
	X: \R^{n-1}\times\R\times\R^{n-1}\rightarrow\R^{n-1}
	\end{equation}
	 is a smooth map. Then H\"ormander's condition in Definition \ref{HorCond} is equivalent to the following condition. \footnote{The equivalence is elementary to see, and we leave out the details. }

	\begin{definition}[H\"ormander's condition for $X$]\label{HorCondX}
		Let $\Phi(\omega, t; \xi)=\left(X(\omega, t; \xi), t\right)$. We say $X$ satisfies H\"ormander's condition at $(\omega, t; \xi)$ if both of the following two conditions are satisfied:
		\begin{enumerate}
			\item[(1)] the following $(n-1)\times (n-1)$ matrix
			\begin{equation}\label{le1.11}
				\nabla_{\omega} X(\omega, t; \xi) 
			\end{equation}
			is non-degenerate; 
			\item[(2)] the following $2(n-1)\times 2(n-1)$ matrix 
			\begin{equation}\label{le1.12}
				\begin{bmatrix}
					\nabla_\omega X & \nabla_{\xi}X\\
					\nabla_\omega\partial_t X & \nabla_{\xi}\partial_t X
				\end{bmatrix}
			\end{equation}
			is non-degenerate at $(\omega, t; \xi)$.
		\end{enumerate}
	\end{definition}

	 We now show that if a phase function $\phi$ satisfies H\"ormander's condition in Definition \ref{250726defi1_1}, then its $\phi$-Kakeya (maximal) problem can be written as a $\Phi$-Kakeya (maximal) problem for some map $\Phi$ satisfying H\"ormander's condition in Definition \ref{HorCond}. We do this exercise here to help us develop some intuition for Definition \ref{HorCond} and Definition \ref{HorCondX}. 
	 
	 Let $\phi(x, t; \xi)$ be a phase function satisfying H\"ormander's condition in Definition \ref{250726defi1_1}. Its characteristic curves are given by 
	 \begin{equation}
	\Phi(w, t; \xi)= (X(w, t; \xi), t)
	 \end{equation}
where $X(w, t; \xi)$ satisfies 
\begin{equation}\label{250726e1_30}
\nabla_{\xi} \phi(X(w, t; \xi), t; \xi)=w.
\end{equation}
We show that the map $X$ satisfies Definition \ref{HorCondX}.

To see \eqref{le1.11}, we differentiate both sides of \eqref{250726e1_30} in $w$ and obtain 
\begin{equation}
\nabla_{x} \nabla_{\xi} \phi(X(w, t; \xi), t; \xi) \nabla_{w} X(w, t; \xi)=I_{n-1}.
\end{equation}
That $\nabla_{w} X$ is non-degenerate follows from the condition (H1) in Definition \ref{250726defi1_1}. That \eqref{le1.12} holds can be proven similarly and we leave out the details.

 Moreover, it is very easy to find 
	 maps $\Phi$ such that $\Phi$-Kakeya sets cannot be realized as $\phi$-Kakeya sets for any phase function $\phi$, for instance
\begin{example}
	Let $\Phi(\omega, t; \xi)= (X(\omega, t; \xi), t)$ with 
	\begin{equation}
	X(\omega, t; \xi)=(\omega_1-t\xi_1-t^2\xi_1\xi_2-t^3\xi_2,\omega_2-t\xi_2-t^2\xi_1\xi_2-t^3\xi_1).
	\end{equation}
	\end{example}
The key missing symmetry here is that $\nabla^2_{\xi}\phi$ is a symmetric matrix, but no analogous conditions are needed for $\Phi$-Kakeya problems. Indeed, curved Kakeya problems that do not come from phase functions are ``generic" among all curved Kakeya problems. \\
	
	For every $X$ satisfying the H\"ormander's condition in Definition \ref{HorCondX}, we can transform it into a normal form (see Theorem \ref{250727theorem1_13} below).  Throughout the paper, we do not distinguish the normal form of $\Phi$ from the normal form of $X$. In other words, whenever we say the normal form of $\Phi$, we always mean that we first find  $X$ as above and then reduce $X$ to its normal form. 
	
	\begin{definition}[Normal forms for $\Phi$-Kakeya problems]\label{NormalForm}
		Let $\Phi(\omega, t; \xi)=\left(X(\omega, t; \xi),t\right)$ with $X$ satisfying the H\"ormander's condition in Definition \ref{HorCondX}. We say $\Phi$ is of a normal form, or $X$ is of a normal form, if
		\begin{equation}
			X(\omega, t; \xi)= \omega+ t(\xi+O(|\omega||\xi|))+ t^2 O(|\xi|)+t^3 O(|\xi|)+\cdots,
		\end{equation}
		where $O(|\xi|)$ and $O(|\omega||\xi|)$ are functions of $\omega,\xi$.
	\end{definition}

Analogously to the study of H\"ormander-type oscillatory integrals and related curved Kakeya problems, H\"ormander's condition in Definition \ref{HorCond} does not guarantee that $\Phi$-Kakeya sets have full Hausdorff dimensions. For this purpose, we introduce Bourgain's condition for $\Phi$-Kakeya sets.

%
	
	\begin{definition}[Bourgain's condition for normal forms of $\Phi$-Kakeya problems]\label{BourCond}
		Consider a normal form $\Phi(\omega, t; \xi)=\left(X(\omega, t; \xi),t\right)$. We say $\Phi$ (or $X$) satisfies Bourgain's condition at the origin, if
		\begin{equation}
		\nabla_{\xi}\partial^2_t X(\omega, t; \xi) \text{ is a multiple of } \nabla_{\xi}\partial_t X(\omega, t; \xi),
		\end{equation}
		at $(\omega, t; \xi)=(0,0;0)$. Similarly, we define Bourgain's condition at a general point $(\omega_0, t_0; \xi_0)$.
	\end{definition}

We conclude this subsection by stating a few results that are completely parallel to those for $\phi$-Kakeya (maximal) problems. 

\begin{theorem}\label{250727theorem1_13}
 Let $\Phi$ be a map satisfying H\"ormander's condition in Definition \ref{HorCond}. Let the map $X$ be given by \eqref{250727e1_25}. 
\begin{enumerate}
\item\label{thm1_13item1} If $\Phi$ fails Bourgain's condition in Definition \ref{BourCond} at some point, then there exists $p_{\Phi}< n$ such that the following $\Phi$-Kakeya maximal estimate fails for every $p_{\Phi}< p\le n$: 
\begin{equation}\label{250726e1_14zzz}
\norm{
\mathcal{K}^{(\Phi)}_{\delta} f
}_{L^p(
\B^{n-1}_{\epsilon_{\Phi}}
)}
\lesim_{\epsilon, \Phi, p}
\delta^{1-\frac{n}{p}-\epsilon} \|f\|_{L^p(\R^n)},
\end{equation}
for all $0<\delta< \epsilon_{\Phi}$ and all $\epsilon>0$.

\item\label{thm1_13item2} To prove $L^p$ bounds for the $\Phi$-Kakeya maximal function, it is equivalent to prove the same bounds for $\Phi$ of a normal form at the origin. In other words, one can apply changes of coordinates to turn $\Phi$ into a normal form at the origin.  

\item\label{thm1_13item3} The map $X$ satisfies Bourgain's condition in Definition \ref{BourCond} at a point $(w,t ; \xi)$ if and only if there exists $c=c(w, t; \xi)$ such that 
\begin{align}\label{25100e1_35ddd}
		 {\rm rank}\begin{bmatrix}
				\nabla_\omega X & \nabla_\xi X  & 0  \\
				\partial_t\nabla_\omega X & \partial_t\nabla_\xi X & \nabla_\omega X  \\
				\partial_t^2\nabla_\omega X & \partial_t^2\nabla_\xi X  
				& c \nabla_\omega X+ 2\ \partial_t\nabla_\omega X 
			\end{bmatrix} = 2(n-1),
		\end{align}
if and only if there exists $c=c(w, t; \xi)$ such that 
    \begin{align}\label{260615e1_37}
        \partial_t^2( (\nabla_\omega X)^{-1} \nabla_\xi X)(w,t;\xi) = c\  \partial_t( (\nabla_\omega X)^{-1} \nabla_\xi X)(w,t;\xi).
    \end{align}

\item\label{thm1_13item4} Let $\phi$ be a phase function satisfying H\"ormander's condition in Definition \ref{250726defi1_1}. Define $X$ via \eqref{250726e1_30}. 
If $\phi$ satisfies Bourgain's condition, then $X$ also satisfies Bourgain's condition.  

\item\label{thm1_13item5} 
If $\Phi$ satisfies Bourgain's condition in Definition \ref{BourCond} everywhere, then the $\Phi$-Kakeya problem satisfies polynomial Wolff axioms\footnote{
It is not our main focus to discuss polynomial Wolff axioms here, and we refer the definition to Theorem 1.2 in \cite{GWZ24}
}.
\end{enumerate}

\end{theorem}

The proof for Theorem \ref{250727theorem1_13} is routine and tedious. As the paper is already long, we decide to leave out the proofs.

Readers familiar with the formulation of Bourgain's condition for phase function in \cite{GWZ24} will find the form \eqref{260615e1_37} very familiar as well, and very easy to accept. The reason that we provide another equivalent formulation \eqref{25100e1_35ddd} is that it is extremely convenient when applied to the study of Schr\"odinger equations with potentials, see Appendix  \ref{260615append_c}.

Below let us try to make the second item in Theorem \ref{250727theorem1_13} more precise. We make explicit the transformations we may apply to a family of curves to get the same Kakeya problem. Write $\mathcal C$ for a $2(n-1)$-parameter family of curves, and $\ell_{\xi, w}$ for a curve in the family. Write the parameter space as $\Sigma \times W$, and the space in which the curves live as $M$. Consider a map 
\begin{equation}
F =(F_{\Sigma \times W},F_{M}),
\end{equation}
factoring into local diffeomorphism on $\Sigma \times W$ and $M$ respectively. The map $F$ induces a new family of curves $\mathcal C' = F_M(\mathcal C)$ pointwise: 
  A curve $\ell_{\xi,w} \in \mathcal C$ is sent to 
  \begin{equation}
      \ell'_{\xi', w'} = F_M(\ell_{F^{-1}_{\Sigma\times W}(\xi', w')}) \in  F_M(\mathcal{C}).
  \end{equation}

\begin{definition}[Direction-equivalence]
    A direction-preserving transformation is a  transformation
    \begin{equation}
    F(\xi, w, \bfx) = (F_{\Sigma \times W}(\xi,w), F_M(\bfx))
    \end{equation}
     which in addition satisfies 
    \begin{align}
        F_{\Sigma \times W}(\xi, w) = (F_\Sigma(\xi), F_W(\xi, w)).
    \end{align}
    We say that $\mathcal C$ is direction-equivalent to $\mathcal C'$ if there exists a direction-preserving transformation. 
\end{definition}

The key Kakeya-hypothesis preserved under direction-preserving transformations is, of course, direction separatedness. The Katz-Wolff condition is preserved under direction-equivalence. \\

We should remark on another more restrictive notion of equivalence at the level of phase functions. 
Write $\ell_{\xi, w}$ for the characteristic curve 
\begin{equation}
\{\bfx : \nabla_\xi \phi(\bfx,\xi)=w\}.
\end{equation}
We may apply diffeomorphisms $\bfx' = G_1(\bfx)$ and $\xi' = G_2(\xi)$ separately to a phase function $\phi$ and obtain
\begin{equation}
\phi'(\bfx', \xi') = \phi(G_1^{-1}(\bfx'), G_2^{-1}(\xi')).
\end{equation}
 A quick calculation relates the characteristic curves $\ell'_{\xi',w'}$ of $\phi'$ to those of $\phi$: 
\begin{align}
    \ell'_{\xi',w'} = G_1(\ell_{\xi, \nabla G_2(\xi) w'}). 
\end{align}
We see that at the level of curves, the corresponding map $F$ is given by: 
\begin{align}
    F(\xi, w,\bfx) &=(F_{\Sigma \times W}(\xi,w), F_M(\bfx)) \\
    &= ((G_2(\xi), \nabla G_2(\xi)^{-1} w), G_1(\bfx)).
\end{align}
This is a direction-preserving transformation with extra restrictions.

\subsection{Finer characterizations for 
\texorpdfstring{$\Phi$}{}-Kakeya sets
}

In Subsection \ref{250726subsection1_2}, we introduced finer characterizations for H\"ormander-type phases. In the current subsection, we introduce parallel finer characterizations for the general curved Kakeya problems ($\Phi$-Kakeya problems) in Subsection \ref{250929sub1_3}, see Definition \ref{250510defi1_1} and Definition \ref{250801definition1_16} below. Moreover, we will also introduce a few new notions, motivated by various examples, and we will see eventually how all these notions fit together in our main theorem (see Theorem \ref{thm: main}). 

\begin{definition}[Katz-Wolff condition for curved Kakeya problems]\label{250510defi1_1}
    Consider the $\Phi$-Kakeya problem where
    \begin{equation}
        \Phi(w, t; \xi)= (X(w, t; \xi), t),
    \end{equation}
    and without loss of generality we assume that $X$ is in a normal form. 
    Fix $w_0$ and $\xi'\neq \xi''$, and consider two characteristic curves 
    \begin{equation}
        \{(X(w_0, t; \xi'), t): |t|\le \epsilon_{\Phi}\}, \ \ \{(X(w_0, t; \xi''), t): |t|\le \epsilon_{\Phi}\}.
    \end{equation}
    For given $t_1, t_2$, let $w(t_1, t_2)$ and $\xi(t_1, t_2)$ be the unique (if exist) solutions satisfying that the characteristic curve 
    \begin{equation}
        \{(X(w(t_1, t_2), t; \xi(t_1, t_2)), t): |t|\le \epsilon_{\Phi}\}
    \end{equation}
    passes through the two points
    \begin{equation}
        (X(w_0, t_1; \xi'), t_1), \ \ (X(w_0, t_2; \xi''), t_2).
    \end{equation}
    We say that $\Phi$ satisfies the  Katz-Wolff condition if for all choices of $w_0, \xi'\neq \xi''$, the image of $\xi(t_1,t_2)$ is a curve (one-dimensional instead of two-dimensional). Equivalently, 
\begin{equation}
        \rank[\partial_{t_1} \xi(t_1, t_2), \partial_{t_2} \xi(t_1, t_2)]< 2, \ \ \forall t_1, t_2.
    \end{equation}
\end{definition}

\begin{definition}[Totally geodesic condition for curved Kakeya problems]\label{250801definition1_16}
Under the same notation as in Definition \ref{250510defi1_1}, 
    we say that $\Phi$ satisfies
    the totally geodesic condition if for all choices of  $w_0$ and $\xi'\neq \xi''$, the set 
    \begin{equation}\label{eq: curve surface union}
        \bigcup_{t_1, t_2} \{(X(w(t_1, t_2), t; \xi(t_1, t_2)), t): |t|\le \epsilon_{\Phi}\}
    \end{equation}
    forms a two-dimensional surface (instead of a 3-dimensional region). Moreover, for a phase function $\phi$ satisfying H\"ormander's condition in Definition \ref{250726defi1_1}, we say that it satisfies the totally geodesic condition if its induced $\Phi$ (given by \eqref{250726e1_30}) satisfies the totally geodesic condition. 
\end{definition}

Recall Bourgain's example \eqref{250801e1_8} in \cite{Bou91}. We write it as a $\Phi$-Kakeya problem, with 
\begin{equation}\label{250930e1_46}
    \Phi(w, t; \xi)= (w_1+ t\xi_2+ t^2\xi_1, w_2+ t\xi_1, t),
\end{equation}
where $w=(w_1, w_2)$ and $\xi=(\xi_1, \xi_2)$. 
We apply a change of coordinate and see that $\Phi$ is direction-equivalent to 
\begin{equation}\label{260526e1_53}
    \widetilde{\Phi}(w, t; \xi)= 
    (w_1- tw_2+ t\xi_2, w_2+ t\xi_1, t).
\end{equation}
From this, it is elementary to see that both $\Phi$ and $\widetilde{\Phi}$ satisfy the totally geodesic condition in Definition \ref{250801definition1_16}. Moreover, every totally geodesic surface gives rise to a Kakeya set (see Subsection \ref{250801subsection2_4} below). 

\begin{definition}[Extra-worst compression condition]\label{250801defi1_17zz}
Take $n=3$.
    We say that $\Phi$ (or a phase function $\phi$) satisfies the extra-worst compression condition if $\Phi$ (or the induced map $\Phi$ given by \eqref{250726e1_30}) satisfies the totally geodesic condition in Definition \ref{250801definition1_16}, and every totally geodesic surface gives rise to a $\Phi$-Kakeya set. 
\end{definition}

The maps $\Phi$ satisfying Definition \ref{250801defi1_17zz} are in some sense ``extra-worst possible" maps, and there are plenty of low dimensional Kakeya sets associated to these $\Phi$. In comparison, in the following definition, we introduce ``worst possible" maps, meaning that we can find at least one low dimensional Kakeya set.

\begin{definition}[Worst compression condition]\label{260615defi1_19}
Take $n=3$ and we follow the notation in Definition \ref{250510defi1_1} and Definition \ref{250801definition1_16}. We say that $\Phi$ (or a phase function $\phi$) satisfies the worst compression condition if 
we can find $w_0$ and $\xi'\neq \xi''$ satisfying that 
\eqref{eq: curve surface union} forms a two-dimensional surface and 
\begin{equation}
        \rank[\partial_{t_1} \xi(t_1, t_2), \partial_{t_2} \xi(t_1, t_2)]= 2, 
    \end{equation}
    for some $t_1\neq t_2$. 
\end{definition}

If a map $\Phi$ satisfies the worst compression condition in Definition \ref{260615defi1_19}, it is immediate to see that one can find a $\Phi$-Kakeya set that is two dimensional. Moreover, we will show below in one of our main theorems (see Theorem \ref{260615theorem1_30}) that 
this is the only way for a two-dimensional $\Phi$-Kakeya set to occur, under the extra assumption that $\Phi$ is semi-algebraic; we will also prove the above result more quantitatively, meaning that if $\Phi$ fails the worst compression condition, then we can find $\varepsilon_{\Phi}>0$ such that every $\Phi$-Kakeya set has dimension $\ge 2+\varepsilon_{\Phi}$. \\

Maps $\Phi$ satisfying the extra-worst compression condition are extremely rigid; for instance, we will show in our main theorem (see Theorem \ref{thm: main} below) that $\Phi$ satisfies the extra-worst condition if and only if it satisfies the totally geodesic condition in Definition \ref{250801definition1_16} and fails Bourgain's condition. However, maps satisfying the worst compression condition are less rare, and it is very easy to find maps satisfying the worst compression condition but not the extra-worst compression condition. \\

In contrast to the above two definitions, the maps $\Phi$ satisfying the following definition are the ``best possible", in the sense that their associated Kakeya sets are expected to have full dimensions.

\begin{definition}[Bochner-Riesz-type]
 We say that the map $\Phi$ (or a phase function $\phi$) is of Bochner-Riesz-type if it (or the induced map $\Phi$ given by \eqref{250726e1_30}) is direction-equivalent to the characteristic curves of the phase function 
 \begin{equation}
 \phi(x, t; \xi) = |(x, t) - (\xi, \varphi(\xi))|,
 \end{equation}
  where $\varphi : \R^{n-1} \to \R$ is a smooth function. The curves are given by 
  \begin{equation}
  \ell_{\xi,w}(t) = (\xi, \varphi(\xi)) - t (w,1).
  \end{equation}
\end{definition}

The connections among the newly introduced definitions in this subsection will become clear once we state our main theorem (see Theorem \ref{thm: main} below).

\subsection{The geometry of paths and other geometry background}

We are interested in putting a geometric structure on $M = \R^n$ whose geodesics are exactly the curves induced by $\Phi$ (which we can call $\Phi$-curves)--this will give us access to geometric invariants that we can use to study our Kakeya set.  
Riemannian geometry is unfortunately too restrictive (see \cite{DGGZ24}) to encode the family of $\Phi$-curves for most $\Phi$, so we need to work instead with a more general structure. The structure we use is a \emph{spray space}. In this section we include the minimal definitions needed to state and motivate our results. See Appendix \ref{250711appendix_a} for further details. Most of the materials in this subsection are standard, and can found for instance in the book \cite{Shen01}. \\

First we write $M$ for a manifold and $TM$ for its tangent bundle, with projection 
\begin{equation}
\pi_{TM} : TM \to M.
\end{equation}
If $V$ is a vector space, we say $C \subset V \setminus 0$ is \emph{conic} if for all $\lambda > 0$, $\lambda C = C$ (these can be identified with subsets of $(V\setminus 0)/ \R_{> 0}$, or with an inner product chosen, subsets of $S^{n-1}$). A set
\begin{equation}
\mathcal U \subset TM \setminus 0
\end{equation}
 is a fiberwise conic set if for each $x \in M$, 
\begin{equation}
\mathcal U_x := \mathcal U \cap T_xM \subset T_xM \setminus 0
\end{equation}
is conic. 

\begin{definition}[Spray space]\label{260526defi1_19}
    A spray space is a triple $(M, \mathcal U, G)$ of a manifold $M$, an open fiberwise conic set $\mathcal U \subset TM \setminus 0$, and $\mathbf G \in \Gamma(T\mathcal U)$ expressed in the local coordinates $(x^i,y^i)$ in $\mathcal U$ as 
    \begin{align}
        \mathbf{G} = y^i \frac{\partial}{\partial x^i} - 2 G^i(x,y) \frac{\partial}{\partial y^i},
    \end{align}
    where $G^i(x,y)$ are local functions on $\mathcal U$ satisfying
    \begin{equation}\label{260526e1_60zz}
    G^i(x,\lambda y) = \lambda^2 G^i(x,y)
    \end{equation}
for every     $\lambda > 0$. \footnote{This is a modification of the definition of a spray space from \cite{Shen01}. All the definitions and theorems about spray spaces that we use are applied locally, so there is no issue in working on a fiberwise conic section instead of the entire slit tangent bundle $TM \setminus 0$.}
\end{definition}

As the introduction is already long, we will postpone the discussion of a geometric interpretation of spray spaces to Section \ref{260526section3}. For readers not familiar with spray spaces, it may be a good idea to take a look at the first half of Section \ref{260526section3} first, and this may help understand the contents of the forthcoming materials, say Proposition \ref{prop: Kakeya to spray space}. \\

In what follows, the objects in bold are coordinate invariant. 
The \emph{geodesics} of $\bf G$ are the projections to $M$ of the integral curves of $\bf G$. That is, the geodesics $\gamma$ satisfy the equation 
\begin{equation}\label{260526e1_61a}
\ddot c^i + 2 G^i(c,\dot c)=0.
\end{equation}
We refer to Section \ref{260526section3} (in particular, equation \eqref{250710ea_5}) for 
how to derive the equation \eqref{260526e1_61a} and for its geometric interpretation in local coordinates.  It is natural to work with coordinate invariant objects, since Kakeya estimates are essentially invariant with respect to diffeomorphism (of course not depending on scale parameters like $\delta$). 
The proposition below says that families of $\Phi$-curves and the geodesics of a spray space are essentially equivalent (after remembering the identifications of parameters).

\begin{proposition}\label{prop: Kakeya to spray space}
Let $n\ge 2$ and let $\Phi$ be a smooth map satisfying H\"ormander's condition as in Definition \ref{HorCond}. Let $\varepsilon_{\Phi}>0$ be a sufficiently small constant depending on $\Phi$. Then there is a spray space $(M, \mathcal U, \bG)$ and a map 
\begin{equation}
(W,\Xi) : \mathcal U \to \mathcal \B^{n-1}_{\varepsilon_{\Phi}} \times \B^{n-1}_{\varepsilon_{\Phi}},
\end{equation}
where we write
\begin{equation}
(W_x,\Xi_x) :\mathcal U_x \to \B^{n-1}_{\varepsilon_{\Phi}} \times \B^{n-1}_{\varepsilon_{\Phi}}
\end{equation}
with the following properties: 
The geodesic with tangent $v_x \in \mathcal U_x$ is the curve with parameters
        \begin{equation}
        (W_x(v_x),\Xi_x(v_x)).
        \end{equation}
         Conversely, every curve with parameters in $\B^{n-1}_{\varepsilon_{\Phi}} \times \B^{n-1}_{\varepsilon_{\Phi}}$ is a geodesic of the spray space.
\end{proposition}

Proposition \ref{prop: Kakeya to spray space} is stated in a coordinate-free form. When we prove this proposition in Section \ref{260526section3}, we will first write it in local coordinates, and readers who are not used to the language of spray spaces can read the formulation there directly.

\begin{example}[Flat spray]
     A spray $(M,\mathcal U, \bG)$ is flat if there are local coordinates on $M$ where $\bG = y^i \frac{\partial}{\partial x^i}$. The geodesics are the family of lines. 
\end{example}

When studying $\Phi$-Kakeya sets, we may always localize to a positive fraction of the curves and a smaller fraction of space (of course not depending on a scale parameter like $\delta$). In the language of spray spaces, this gives us the freedom to shrink $M$ and $\mathcal U$ to smaller open sets around our point of interest. For example, after fixing $(x_0,y_0) \in \mathcal U$ and a local trivialization of $TM$, we can find $U \subset M$ containing $x_0$ and an open cone $C \subset \R^n$ containing $y_0$ such that $U \times C \subset \mathcal U$. It will sometimes be useful to restrict to this product set. 

We of course are only interested in studying $\Phi$-Kakeya sets as point sets (we do not care how the curves are parameterized). This is captured by the following definition. 
\begin{definition}[Pointwise projectively related]
    Two sprays $\bG$ and $\tilde \bG$ defined on $(M, \mathcal U)$ are pointwise projectively related if for any geodesic $c(t)$ of $\bG$, there is an orientation-preserving reparameterization $t= t(s)$ such that $c(s) := c(t(s))$ is a geodesic of $\tilde \bG$, and vice versa. Equivalently, there is a scalar function $P$ on $\mathcal U$ which is positively $1$-homogeneous in $y$ such that 
    \begin{align}
        \tilde \bG = \bG - 2P \mathbf{Y},
    \end{align}
    where in coordinates $\mathbf Y = y^i \frac{\partial}{\partial y^i}$. 
\end{definition}
There are two natural projective invariants for $\bG$. The first is the \emph{Douglas tensor} 
\begin{equation}
\mathbf D \in \Gamma(\mathcal U, \hom((\pi^*T M)^{\otimes 3}, \pi^* TM)),
\end{equation}
 which detects if $\bG$ is \emph{affine} up to projective transformation (this is a non-Riemannian quantity). The other is the \emph{Weyl tensor}
 \begin{equation}
 \mathbf W \in \Gamma(\mathcal U, \mathrm{End}(\pi^* TM)),
 \end{equation}
  which detects curvature up to projective transformation. The tensor $\mathbf W$ is the projective trace-free part of the Riemann tensor 
  \begin{equation}
  \mathbf R \in \Gamma(\mathcal U, \mathrm{End}(\pi^* TM)).
  \end{equation}
   We will give formulas for all these objects in Appendix \ref{250711appendix_a} and their properties. These invariants let us detect whether a spray $\bG$ is (locally) projectively flat for $\dim M \geq 3$ (that is, there are coordinates on $M$ where $\bG$ is projectively related to the flat spray $y^i \frac{\partial}{\partial x^i}$).
\begin{definition}[Projectively flat]\label{def: locally projectively flat}
    A spray space $(M, \mathcal U, \bG)$ is (locally) projectively flat if
    any $(x_0,y_0) \in \mathcal U$ has a neighborhood $\mathcal U' \subset \mathcal U$ on which $\bG$ is projectively related to a flat spray. 
\end{definition}

Let us formulate Definition \ref{def: locally projectively flat} in the setting of $\Phi$-Kakeya sets. A $\Phi$-Kakeya problem is said to be \underline{projectively flat} if 
one can find a local diffeomorphism $F$ acting on $\R^n$ such that 
the image of 
\begin{equation}
\{\Phi(w, t; \xi): |t|\le \varepsilon_{\Phi}\}
\end{equation}
under the map $F$ is a line segment, for every $w, \xi$.

\begin{theorem}[Minor modification of \protect{\cite[Theorem 13.5.1]{Shen01}}]\label{thm: douglas weyl vanishing}
    Let $(M, \mathcal U, \bG)$ be a spray space with $\dim M \geq 3$. Suppose that  
    \begin{equation}
    \mathbf{D}\equiv 0, \ \  \mathbf{W} \equiv 0.
    \end{equation}
     Then $\bG$ is locally projectively flat.
\end{theorem}
To see that this holds, one uses $\mathbf D = 0$ to pass to an affine spray defined globally, which agrees with $\bG$ on an open $\mathcal U' \subset \mathcal U$. Then apply \cite[Theorem 13.5.1]{Shen01}.\\

It was pointed out in \cite{gao2025curvedkakeyasetsnikodym} that 
the Kakeya problem on Riemannian manifolds of constant sectional curvature can be turned into the Kakeya problem on the Euclidean space. If we use the language of spray spaces, then  the above result in \cite{gao2025curvedkakeyasetsnikodym} can be summarized as: The Kakeya problem on Riemannian manifolds of constant sectional curvature is projectively flat. This is a special case of Theorem \ref{thm: douglas weyl vanishing}, as Douglas curvature already vanishes on a Riemannian manifold, and constant curvature is the same as vanishing Weyl curvature in the setting of Riemannian geometry. \\

Let us introduce the notion of a totally geodesic surface for a spray space. 
\begin{definition}\label{def: totally geodesic surface}
    Let $(M, \mathcal U, \bG)$ be a spray space. We say that a smooth surface $S \subset M$ is totally geodesic if for any $x \in S$ and geodesic $\gamma$ satisfying $\gamma(0) = x$ and
    \begin{equation}
    \dot \gamma(0) \in T_x S \cap \mathcal U_x,
    \end{equation}
     we have $\gamma(t) \in S$ for $t$ in a neighborhood of $0$. 
\end{definition}

We may now introduce a version of Definition \ref{250801definition1_16} adapted to spray spaces. 
\begin{definition}[Totally Geodesic condition: geometric version]\label{def: T.G. geometric}
    We say that $(M, \mathcal U,\mathbf G)$ has many totally geodesic surfaces if for each $x \in M$ and 2-plane $\Pi$ through the origin with $\Pi \cap \mathcal U_x \neq \emptyset$, there is a totally geodesic surface $S$ satisfying $T_x S = \Pi$.
\end{definition}

So far for a curved Kakeya problem, we have two totally geodesic conditions: Definition \ref{250801definition1_16} and Definition \ref{def: T.G. geometric}. Definition \ref{250801definition1_16} is stated for curved Kakeya problems directly. For Definition \ref{def: T.G. geometric}, we first apply Theorem \ref{prop: Kakeya to spray space} to our curved Kakeya problem to obtain a spray space; once we have a spray space, we can check Definition \ref{def: T.G. geometric}. 
Let us give a sketch that Definition \ref{250801definition1_16} and Definition \ref{def: T.G. geometric} are equivalent. 
\begin{proof}[Sketch of equivalence of Definition \ref{250801definition1_16} and Definition \ref{def: T.G. geometric}]
    Suppose that $X$ is as in Definition \ref{250801definition1_16}. Consider the surface $S(y_1,y_2,\mathbf x_0)$ formed by the two intersecting curves $\Gamma_{y_1}(\mathbf x_0)$ and $\Gamma_{y_2}(\mathbf x_0)$. Identify some interior point $x_{\mathrm{int}}(y_1,y_2,\mathbf x_0)$ and compute the tangent plane $P(y_1, y_2, \mathbf x_0)$ at that point. By the implicit function theorem, for any $x$ in a sufficiently small neighborhood and tangent plane $P$ intersecting a small enough open cone, there are $y_1,y_2,\mathbf x_0$ such that
    \begin{equation}
    x = x_{\mathrm{int}}(y_1,y_2, \mathbf x_0), \ \ P = P(y_1,y_2, \mathbf x_0).
    \end{equation}
     Then $S = S(y_1,y_2, \mathbf x_0)$ is a totally geodesic surface near  $x$ in the sense of Definition \eqref{def: T.G. geometric}. 
    Indeed, there is a 1-parameter family of curves in $S$ passing through any given point. So for any direction (in a small enough open cone) tangent to $S$, there is a curve in $S$ in that direction, and that curve is unique. 

    For the other direction, assume that $X$ is an Definition \ref{def: T.G. geometric}. Fix two intersecting curves $\Gamma_{y_1}(x_0)$ and $\Gamma_{y_2}(x_0)$ (with $y_1$ and $y_2$ close enough), and consider the plane $\Pi$ which contains their tangent vectors at $x_0$. By Definition \ref{def: T.G. geometric}, there is a totally geodesic surface $S$ with $T_{x_0} S = \Pi$. These curves lie in $S$ locally.
    Now fix points $x_1 \in \Gamma_{y_1}(x_0)$ and $x_2 \in \Gamma_{y_2}(x_0)$, both different from $x_0$. Consider the 1-parameter family of (unit) tangent vectors in $T_{x_1} S$. Since $S$ is totally geodesic, these curves all lie in $S$. By the implicit function theorem, one of these curves intersects $x_2$ (for $x_1, x_2$ appropriately chosen in open subsets of the curves $\Gamma_{y_1}(x_0)$ and $\Gamma_{y_2}(x_0)$). Passing over all $x_1, x_2$ in fixed open subsets of the curves $\Gamma_{y_1}(x_0)$ and $\Gamma_{y_2}(x_0)$, we get Definition \ref{250801definition1_16}. 
\end{proof}

\subsection{Statement of main theorems}

Before discussing the formal theorem statements in the next section, we give the picture at the level of curved Kakeya problems. 
We abbreviate the definitions and make some small comments for the picture:

\begin{itemize}
    \item \textbf{K.W.}: Katz-Wolff condition.
    \item \textbf{T.G.}: Totally geodesic condition. 
    \item \textbf{Flat}: (Direction-equivalent to) projectively flat. 
    \item \textbf{B. Cond.}: Bourgain condition. 
    \item \textbf{B.R.-type:} Bochner-Riesz-type. 
    \item \textbf{extra-worst-comp.} Extra-worst compression condition in Definition \ref{250801defi1_17zz}. 
\end{itemize}

\begin{figure}
    \centering
    \includegraphics[width=0.75\linewidth]{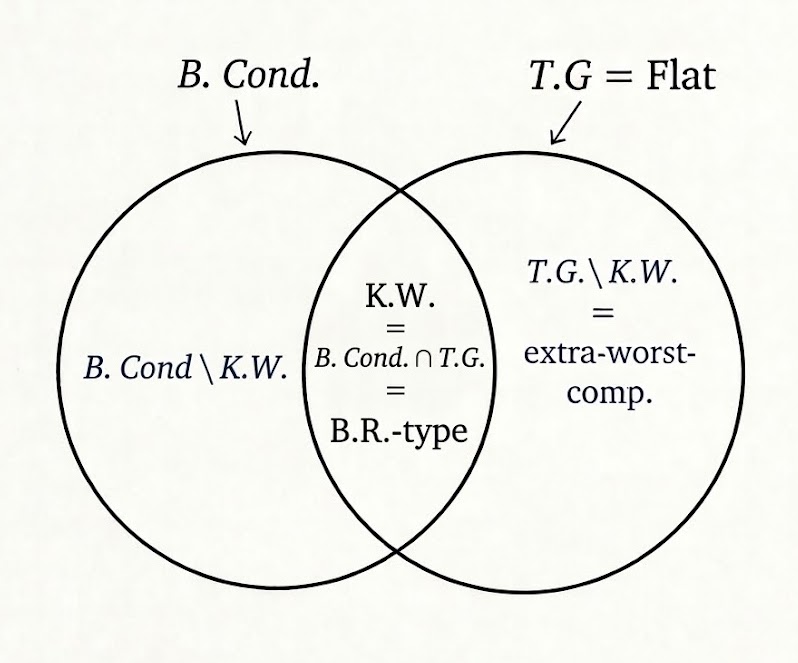}
    \caption{Main Venn Diagram of Conditions}
    \label{fig: venn diagram}
\end{figure}

Below we state the theorems that justify Figure \ref{fig: venn diagram}.

\begin{theorem}[Main Venn-diagram Theorem]\label{thm: main}
Let $\Phi$ be a map satisfying H\"ormander's condition in Definition \ref{HorCond} and take $n\geq 3$. 
    \begin{enumerate}
        \item\label{251003thm_item1} Take $n=3$. That $\Phi$ satisfies the totally geodesic condition (either Definition \ref{def: T.G. geometric} or Definition \ref{250801definition1_16}) is equivalent to that the induced  spray space (by $\Phi$ as in Proposition \ref{prop: Kakeya to spray space}) is projectively flat.\label{thm: main part 1}
        \item\label{251003thm_item2} If $\Phi$ satisfies the Katz-Wolff condition, then it also satisfies the totally geodesic condition.
        \label{thm: main part 2}
        \item\label{251003thm_item3} If $\Phi$ satisfies the Katz-Wolff condition, then it also satisfies Bourgain's condition. \label{thm: main part 3}
        \item\label{251003thm_item4} If $\Phi$ satisfies both the totally geodesic condition and Bourgain's condition, then it is of the Bochner-Riesz-type\footnote{We also refer to the beginning of Section \ref{251003secc6} for a more quantitative version of this statement.}, which further implies that it satisfies the Katz-Wolff condition.   \label{thm: main part 4}
        \item\label{251003thm_item5} Take $n=3$. If $\Phi$ satisfies the totally geodesic but fails Bourgain's condition, then it satisfies the extra-worst compression condition.\footnote{For dimensions $n\ge 4$, the worst compression will not just occur inside $(n-1)$-dimensional sub-manifolds, but sub-manifolds of even smaller dimensions, see examples for instance in \cite{GHI19}. Therefore when $n\ge 4$, one will need new definitions that replace Definition \ref{250801defi1_17zz}. We will leave this to future works.}  \label{thm: main part 5}
    \end{enumerate}
Moreover, if we let $\phi$ be a phase function satisfying H\"ormander's condition in Definition \ref{250726defi1_1}, then all the above statements hold true with $\phi$ in place of $\Phi$. 
\end{theorem}

Examples in Subsection \ref{260526subsec2_3} show that the Katz-Wolff condition is strictly stronger than Bourgain's condition. Moreover, recall the discussion around \eqref{250930e1_46} and \eqref{260526e1_53} that Bourgain's example \eqref{250801e1_8} satisfies the extra-worst compression condition. \\

As a consequence of 
Theorem \ref{thm: main}, we immediately obtain the following several results. 

The original definition for the totally geodesic condition in Definition \ref{250801definition1_16} is not a local condition, and may not be easy to check. As a consequence of item 1 in Theorem \ref{thm: main}, we see that to check the totally geodesic condition, it is equivalent to write down the spray space (Proposition \ref{prop: Kakeya to spray space}), and then to check that $W=D=0$.

The Katz-Wolff condition in Definition \ref{250510defi1_1} is not a local condition either. As a consequence of item 2 and item 3 in Theorem \ref{thm: main}, to check the Katz-Wolff condition, it is equivalent to check Bourgain's condition, and then to check totally geodesic condition. \\

Theorem \ref{thm: main} discusses, in particular, a ``best" possible behavior for curved Kakeya problems. These curved Kakeya problems satisfy the Katz-Wolff condition, which is equivalent to Bourgain's condition together with the totally geodesic condition, and is further equivalent to the Bochner-Riesz type condition. One can apply the result of Wang and Zahl \cite{wang2025volume} and obtain that for $n=3$, $\Phi$-Kakeya sets must have dimension 3 whenever $\Phi$ satisfies any of the above conditions. \\

The next theorem characterizes all $\Phi$ that admit a two-dimensional Kakeya set, under the extra assumptions that $n=3$ and $\Phi$ is semi-algebraic. 

\begin{theorem}\label{260615theorem1_30}
Assume $n=3$.
\begin{enumerate}
\item Assume the map $\Phi$ is semi-algebraic. Then $\Phi$ satisfies the worst compression condition in Definition \ref{260615defi1_19} if and only if 
there exists a $\Phi$-Kakeya set that is two dimensional. Whenever $\Phi$ fails the worst compression condition, we can find $\kappa_{\Phi}>0$ such that every $\Phi$-Kakeya set has dimension $\ge 2+ \kappa_{\Phi}.$

\item Assume that the phase function $\phi$ is semi-algebraic. Then $\phi$ fails the worst compression condition if and only if the following statement holds: For every $2< p< \infty$, there exists $\beta_{\phi}(p)< 1/p$ such that 
\begin{equation}\label{260616e1_76}
\norm{
\mathcal{K}^{(\phi)}_{\delta} f
}_{L^p(
\B^{n-1}_{\epsilon_{\phi}}
)}
\lesim_{\phi, p}
\delta^{-\beta_{\phi}(p)} \|f\|_{L^p(\R^n)},
\end{equation}
for every $\delta\in (0, \epsilon_{\phi})$. In particular, if $\phi$ fails the worst compression condition, then there exists $\kappa_{\phi}>0$ such that every $\phi$-Kakeya set has dimension $\ge 2+ \kappa_{\phi}$.
\end{enumerate}
\end{theorem}

Note that the estimate \eqref{260616e1_76} with $p\ge 2$ and $\beta_{\phi}(p)=1/p$ holds for every phase function $\phi$, see for instance equation (6.15) in \cite{Bou91} and  Proposition 1.4 in \cite{CGGHIW24}. Moreover, \cite{Bou91} and  \cite{CGGHIW24} gave several sufficient conditions for \eqref{260616e1_76} to be true for every $p>2$ and some $\beta_{\phi}(p)< 1/p$; once such an estimate is proven (for some $p>2$), a standard argument shows that every $\phi$-Kakeya set has dimension $\ge 2+ \kappa_{\phi}$ for some $\kappa_{\phi}>0$.

We use the example 
\begin{equation}\label{260616e1_77}
\phi(x, t; y)=
\langle x, y\rangle+
\frac{t}{2}\left(y_1^2+y_2^2\right)+t^2\left(y_1 y_2+\frac{1}{6} y_1^4\right)+\frac{t^3}{2} y_1^2
\end{equation}
to show that the conditions in \cite{Bou91} and  \cite{CGGHIW24} are not necessary. 
Let $\mathcal{Z}_K(\phi)$ denote the set of all  $y\in \B^2$ for which there exist scalars $\mu_{i, j}, 1\le i, j\le 2$ such that 
\begin{equation}
\operatorname{det} \partial_{y y}^2 \phi(0, t ; y)=\sum_{1 \leq i, j \leq 2} \mu_{i, j} \partial_{y_i y_j}^2 \phi(0, t ; y) \quad \text { for all } t \in(-1,1) .
\end{equation}
Hypothesis I)  in \cite[Definition 2.2]{CGGHIW24} requires that 
\begin{equation}\label{260616e1_79}
\mathcal{L}^2(\mathcal{Z}_K(\phi))=0.
\end{equation}
Under this hypothesis, \cite{CGGHIW24} proved that \eqref{260616e1_76} holds for every $p>2$ and some $\beta_{\phi}(p)< 1/p$. However, it is elementary to check that the phase function \eqref{260616e1_77} fails \eqref{260616e1_79} and at the same time fails the worst compression condition as well, showing that \eqref{260616e1_79} is not a necessary condition. \\

Once the necessary and sufficient condition is established in Theorem \ref{260615theorem1_30}, one can follow the standard  argument in \cite[Section 5]{Bou91} (see also \cite[Proposition 2.10]{CGGHIW24}) and immediately obtain the following corollary.

\begin{corollary}\label{260617corollary1_31}
    Assume that $n=3$ and the phase function $\phi$ is elliptic and semi-algebraic. Then $\phi$ fails the worst compression condition if and only if there exists $\kappa_{\phi}>0$ such that the estimate \eqref{250708e1_2} holds, that is, 
    \begin{equation}
        \|T^{(\phi)}_N f\|_{L^p(\R^n)}\lesim_{\phi, a, p, \epsilon} N^{-\frac{n}{p}+\epsilon} \|f\|_{L^p(\R^{n-1})}
    \end{equation}
    holds, 
    for all 
$p\ge 10/3-\kappa_{\phi}$. 
\end{corollary}

\subsection{A coordinate invariant formulation of Bourgain's condition}

Let $M \subset \R^n$ and $\Sigma \subset \R^{n-1}$ be open sets, and $\phi : M \times \Sigma \to \R$ a phase function satisfying Bourgain's condition. Since Bourgain's condition is preserved under changes of coordinates on $M$ and $\Sigma$ separately, there ought to be a coordinate-free description of Bourgain's condition. The Appendix \ref{Tao's note} by Tao provides one such answer, which we summarize below. \footnote{Note that a different description of Bourgain's condition based on the geometry of the characteristic curves of $\phi$ was developed in \cite{nadjimzadah2026bourgainsconditionstickykakeya}.} 

By H\"ormander's first condition the map $d_\xi \phi : M \times \Sigma \to T^* \Sigma$ is a submersion, so $M \times \Sigma$ is locally a fiber bundle over $T^* \Sigma$. The projections of the fibers of $d_\xi \phi$ to $M$ are the characteristic curves of $\phi$. A choice of local trivialization $M \times \Sigma \approx T^* \Sigma \times \R$, $(x,\xi) \mapsto (\xi,p,c)$, amounts to parameterizing each characteristic curve as $c \mapsto x(\xi,p,c)$, and a different trivialization reparameterizes the curves. In Appendix \ref{Tao's note}, a special trivialization is chosen. A Legendre-type transformation is used to pass from $T^* \Sigma$ to $T \Sigma$, which helps motivate the objects in the main characterization below.

\begin{theorem}\label{B.C.Tao}
    If $\phi$ is a phase function satisfying Bourgain's condition, then there exist a connection $\nabla^*$ on $T\Sigma$ and two smooth functions $F^*, J^*\in C^\infty(T\Sigma)$ such that for all $X,Y\in \Gamma(T\Sigma)$, we have
    \begin{enumerate}
        \item[(1)] Curvature constraint: 
            \begin{align}\label{260604le1-56}
                [\nabla^*_X,\nabla^*_Y] - \nabla^*_{[X,Y]} + (V_X J^*)V_Y - (V_Y J^*)V_X = 0;
            \end{align} 
        \item[(2)] Compatibility constraint:
            \begin{align}\label{260604le1-57}
                [\nabla^*_X, V_Y] - [\nabla^*_Y, V_X]= V_{[X,Y]};
            \end{align}
        \item[(3)] $J^*$ Compatibility constraint: 
            \begin{align}\label{260604le1-58}
                \nabla^*_X V_Y J^* = \nabla^*_Y V_X J^*;
            \end{align}
        \item[(4)] Symmetry constraint:
            \begin{align}\label{260604le1-59}
                \nabla^*_X V_Y F^* = \nabla^*_Y V_X F^*,
            \end{align}
    \end{enumerate}
    where $V$ is a vertical lift map $V:\Gamma(T\Sigma)\to \Gamma_v(T(T\Sigma))$ and  $V_XV_YF^*$ is non-degenerate. 

    \smallskip

    Conversely, if we have $\nabla^*$, $F^*$ and $J^*$, with $V_XV_YF^*$ non-degenerate, satisfying \eqref{260604le1-56}, \eqref{260604le1-57}, \eqref{260604le1-58} and \eqref{260604le1-59}, then we can construct a phase function satisfying Bourgain's condition.
\end{theorem}

Theorem \ref{B.C.Tao} gives a convenient way to generate some interesting examples of phase functions satisfying Bourgain's condition, some of which are recorded in Subsection \ref{260526subsec2_3}. Theorem \ref{B.C.Tao} is a step toward systematically classifying the phase functions satisfying Bourgain's condition, but we do not see how to reduce the constraints much further at this stage.

The following reformulation of Bourgain's condition in \cite{nadjimzadah2026bourgainsconditionstickykakeya} is implicit in the proof of Theorem \ref{B.C.Tao}. In particular, the proof of Theorem \ref{B.C.Tao} essentially takes advantage of the symmetry of $\partial_{\xi^i} \partial_{\xi^j} \partial_{\xi^k} \phi$ after taking a further derivative of \eqref{eq: B.C.Arian}. 
\begin{proposition}[\cite{nadjimzadah2026bourgainsconditionstickykakeya}]\label{B.C.Arian}
    Let $\phi$ be a H\"ormander-type phase function. Then it satisfies Bourgain's condition if and only if there exist a smooth scalar-valued function $c: M\times\Sigma\to\R$ and two smooth symmetric matrix-valued $A,B: W\times \Sigma\to \rm{Sym}_{n-1}(\R)$, with $B$ non-degenerate, such that for all $(\bfx,\xi)\in M\times\Sigma$,
    \begin{align}\label{eq: B.C.Arian}
        \nabla_\xi^2\phi(\bfx,\xi)= A(\nabla_\xi\phi(\bfx,\xi),\xi) + c(\bfx,\xi)B(\nabla_\xi\phi(\bfx,\xi),\xi).
    \end{align}
\end{proposition}
We refer the readers to Appendix \ref{Tao's note} for the proof of Theorem \ref{B.C.Tao} and further details.

\subsection{Future directions}

As the current paper proposes several new concepts, we feel it is natural to propose several future directions to explore as well. \\

Let $\Phi$ be a map satisfying H\"ormander's condition as in Definition \ref{HorCond}. We define the Kakeya index of $\Phi$ to be 
\begin{equation}\label{260714e1_81}
    \inf\{\dim(E): E\text{ is a } \text{$\Phi$-Kakeya set}\}.
\end{equation}
Under the assumption that $\Phi$ is analytic and semi-algebraic, can we conclude  that the Kakeya index of $\Phi$ is an integer? In particular, when we are in dimension 3, is it the case that the Kakeya index of $\Phi$ is either $2$ or $3$? If we take into account Theorem \ref{260615theorem1_30}, then an equivalent formulation of the previous question is: If $\Phi$ fails the worst compression condition in Definition \ref{260615defi1_19}, can we conclude that $\Phi$-Kakeya sets always have dimension 3? 

Roughly speaking, the above question  asks whether or not the dimension of a $\Phi$-Kakeya set must be an integer. However, it is very easy to produce ``boring" curved Kakeya sets that have fraction dimensions. For instance, we work in dimension 3, and take Bourgain's example \eqref{250801e1_8}. This example admits a two dimensional curved Kakeya set, supported on a hyper-surface. We then make multiple copies of this curved Kakeya set by translating the above hyper-surface, and can produce a curved Kakeya set of an arbitrary fraction dimension between 2 and 3. This also explains the reason of taking inf in \eqref{260714e1_81}. \\





More generally, we propose a systematic study of $\Phi$-Kakeya maps, according to, for instance, the best possible maximal estimates they satisfy. This work is part of this program, and in particular sets up path geometry as a language to develop a finer classification of $\Phi$-Kakeya maps. The work of Bourgain \cite{Bou91} first established that curved Kakeya problems can behave in many different ways. Sogge \cite{Sog99} and Wisewell \cite{Wis03,Wis05} then pioneered their finer study. There has been more recent progress on classifying curved Kakeya problems by a variety of authors \cite{HI22,GWZ24,DGGZ24,CGGHIW24,nadjimzadah2025newcurvedkakeyaestimates,nadjimzadah2026bourgainsconditionstickykakeya}, but there is much work left. 


\section{Examples}

\subsection{
H\"ormander's condition for curved Kakeya problems
}

	\begin{example}[Parabolic Kakeya sets]\label{260626example2_1}
		A Borel set $E\subset \R^2$ with $\mathcal{L}^2(E)=0$ is called a parabolic Kakeya set if for every $u\in [1, 2]$, there exists $x\in [-1, 1]$ such that 
		\begin{equation}\label{le1.3}
			(x, 0)+ (t, u t^2)\in E, \ \forall t\in [0, 1].
		\end{equation}
	\end{example}
	
	 If we set $\Phi(x,t;u)= (x+t,(1+u)t^2)$, then parabolic Kakeya sets essentially become $\Phi$-Kakeya sets. One can check that 
	\begin{equation}
		{\rm det} \nabla_{(x, t)} \Phi(x, t; u)= 2(1+u)t,
	\end{equation}
	and 
	\begin{equation}
		{\rm det}	\begin{bmatrix}
			0 & \nabla_{(x, t)}\Phi & \nabla_{u}\Phi\\
			\partial_t \Phi & \nabla_{(x, t)}\partial_t \Phi & \nabla_{u}\partial_t\Phi
		\end{bmatrix}
	= 2(1+u)t^2.
	\end{equation}
From these, we see that $\Phi$ does satisfy H\"ormander's condition in Definition \ref{HorCond} if $t\neq 0$. The degeneracy at $t=0$ can be expected because for fixed $x$ and $u$, when we change $t$, the curve $(x+t,(1+u)t^2)$ always travels in the horizontal direction, and not any other different direction (compared with the traditional Kakeya problem).

If one is interested in studying a parabolic Kakeya maximal function, 
\begin{equation}\label{260309e2_4}
    \sup_{|x|\le 1} \int_0^1 
    |f(x+ t, u t^2)| dt,
\end{equation}
then we can apply a standard dyadic decomposition to $t$ with respect to $t=0$, and for each dyadic piece, after appropriate scaling we will see that it satisfies H\"ormander's condition in Definition \ref{HorCond}.

	\begin{remark}[Wolff's parabolic Kakeya sets]\label{Wolff_Parabolic}
		Wolff \cite{Wol97} studied another kind of parabolic Kakeya sets which superficially looks very similar to the one in Example \ref{260626example2_1} but is fundamentally different. 
		
		We say that a Borel set $E\subset \R^2$ with $\mathcal{L}^2(E)=0$ is called Wolff's parabolic Kakeya set if for every $u\in [1, 2]$ there exists $(x, y)\in \R^2$ such that 
		\begin{equation}\label{le1.4}
			(x, y)+ (\theta, u\theta^2)\in E, \ \forall \theta\in [0, 1].
		\end{equation}
		It requires Wolff's deep work \cite{Wol97} to show that Wolff's parabolic Kakeya sets must have Hausdorff dimension 2. In contrast, it is elementary to show that parabolic Kakeya sets in Example \ref{260626example2_1} have Hausdorff dimension 2. 
	\end{remark}

	\begin{example}
		[Nikodym sets]A Borel set $E\subset \R^d$ with $\mathcal{L}^d(E)=0$ is called a Nikodym set if for every $\xi\in \B^{d-1}_1$, there exists $x\in \B^{d-1}_1$ such that 
		\begin{equation}
			{\bf L}({(x, 0); (\xi, 1)})\subset E,
		\end{equation}
		where ${\bf L}({(x, 0); (\xi, 1)})$ denotes the line segment connecting $(x, 0)$ and $(\xi, 1)$.
	\end{example}
	
	 There are several definitions of Nikodym sets in the literature; the one we use here comes from Wisewell \cite{Wis05}. One can find easily that Nikodym sets can ``essentially" be realized as $\Phi$-Kakeya sets for some $\Phi$. For instance, one can take 
	\[
	\Phi(x, t; \xi)= 
	(x+ t(\xi-x), t).
	\] 
	One can check that
	\begin{equation}\label{260309eaa}
	{\rm det} \nabla_{(x, t)} \Phi(x, t; \xi)= (1-t)^{d-1},
	\end{equation}
	and 
	$$
	{\rm det}	\begin{bmatrix}
		0 & \nabla_{(x, t)}\Phi & \nabla_{\xi}\Phi\\
		\partial_t \Phi & \nabla_{(x, t)}\partial_t \Phi & \nabla_{\xi}\partial_t\Phi
	\end{bmatrix}
	= -1,
	$$	
	and therefore $\Phi$ satisfies H\"ormander's condition in Definition \ref{HorCond} when $t\neq 1$. 
	
	Let us also remark on what ``essentially" refers to. Because of the singularity in \eqref{260309eaa}, one would need to do a dyadic decomposition in $t$ with respect to $t=1$, similarly to \eqref{260309e2_4}. Each dyadic piece, after appropriate scaling, can be realized as a $\Phi$-Kakeya set. \\

This example is interesting  because we see that through the language of general curved Kakeya problems introduced in Subsection \ref{250929sub1_3}, we can ``unify" the study of the classical Kakeya problem (for line segments) and the classical Nikodym problem (for line segments). It is well-known in the literature that these two problems are equivalent after pseudo-conformal transformations, which exchange the roles of directions and positions. Using the language of general curved Kakeya problems, we can avoid such transformations. \\

\subsection{Examples satisfying Bourgain's condition but not Katz-Wolff condition}\label{260526subsec2_3}

By Theorem \ref{B.C.Tao}, we can construct a phase function $\phi$ satisfying Bourgain's condition from $\nabla^*$, $F^*$ and $J^*$. We assume $n=3$ and use $(\xi,v)$ as the local coordinates of $T\Sigma$. To show how to construct $\phi$, we start from 
\begin{align*}
     \nabla^*_{\partial_{\xi_1}} = \partial_{\xi_1},\  \nabla^*_{\partial_{\xi_2}} = \partial_{\xi_2}+ v^2_2\partial_{v_2},\  
     F^*(v;\xi)= |v|^2/2,\  J^*=0.
\end{align*}
It is a direct calculation to check that $\nabla^*$, $F^*$ and $J^*$ satisfy the equations in Theorem \ref{B.C.Tao}. By \eqref{260604leD-21}, the bundle diffeomorphism $\psi: T\Sigma\to T^*\Sigma$ is given by $\psi(\xi,v)= (\xi,v)$. So we can identify $T\Sigma$ with $T^*\Sigma$. By \eqref{260604leD-19} $\nabla^*$ is pushed forward to
\begin{align}
    \nabla_X=  X^i(\xi)\partial_{\xi_i}+ X^i(\xi)(A_{ij}(\xi,p)+ cB(\xi,p))\partial^{p_j},
\end{align}
where 
\begin{align}
    A=\begin{bmatrix}
        0 & 0\\
        0 & v_2^2
    \end{bmatrix},\quad
    B= I_{2}.
\end{align}
By Proposition \ref{B.C.Arian}, it further implies 
\begin{align}
     \partial_{\xi_i}\partial_{\xi_j}\phi(\bfx,\xi)= A_{ij}(\xi,\nabla_\xi\phi(\bfx,\xi))+ c(\bfx)B_{ij}(\xi,\nabla_\xi\phi(\bfx,\xi)).
\end{align}
The reason why $c$ is independent of $\xi$ is that we do not have $\partial_c$ in $\nabla_X$. If we take $c(\bfx)=t^2$, with $\bfx=(x,t)$, then we can construct $\phi$ as the \rm{tan}-example in \cite{nadjimzadah2026bourgainsconditionstickykakeya}, where it is shown that the family of characteristic curves of \rm{tan}-example is not direction-equivalent to a family of lines.

\begin{example}[The tan-example in \cite{nadjimzadah2026bourgainsconditionstickykakeya}]
    Fix $n\ge 3$ and write $\xi=(\xi',\xi_{n-1})$, $\bfx=(x',x_{n-1},t)$. The tan-example is given by 
    \begin{align}
        \phi(\bfx,\xi)=x'\cdot\xi'+\frac12 t^2|\xi'|^2+ \log(\sec (t\xi_{n-1}+ x_{n-1})).
    \end{align}
    It satisfies the Bourgain's condition in Proposition \ref{B.C.Arian}, in the sense that
    \begin{align}
        \nabla_\xi^2\phi(\bfx,\xi)= A(\xi,\nabla_\xi\phi(\bfx,\xi))+ t^2 B(\xi,\nabla_\xi\phi(\bfx,\xi)),
    \end{align}
    where 
    \begin{align}
    A=\begin{bmatrix}
        0 & 0\\
        0 & v_{n-1}^2
    \end{bmatrix},\quad
    B= I_{n-1}.
\end{align}
The induced curves are given by $(X(w,t;\xi),t)$, with
\begin{align}\label{260604le2-33}
    X(w,t;\xi)= \left(w_1-t^2\xi_1,\arctan\left(\frac{w_2}{t}\right)-t\xi_2\right).
\end{align}
\end{example}

Based on the tan-example, we can construct a family of phase functions $\phi_\epsilon$ satisfying Bourgain's condition, none of which satisfies Katz-Wolff condition. But if one takes the limit $\epsilon\to 0$, then $\phi_\epsilon$ converges to a Katz-Wolff phase.

\begin{example}
    For $\epsilon>0$, we consider a family of phase functions
    \begin{align}
        \phi_\epsilon(\bfx,\xi)=x'\cdot\xi'+\frac12 t^2|\xi'|^2+ \frac{1}{\epsilon^2}\log(\sec (\epsilon t\xi_{n-1}+ \epsilon x_{n-1})).
    \end{align}
    For each $\epsilon>0$, $\phi_\epsilon$ satisfies Bourgain's condition, since
    \begin{align}
        \nabla_\xi^2\phi_\epsilon(\bfx,\xi)= A_\epsilon(\xi,\nabla_\xi\phi_{\epsilon}(\bfx,\xi))+ t^2 B(\xi,\nabla_\xi\phi_{\epsilon}(\bfx,\xi)),
    \end{align}
    where 
    \begin{align}
    A_\epsilon=\begin{bmatrix}
        0 & 0\\
        0 & \epsilon^2 v_{n-1}^2
    \end{bmatrix},\quad
    B= I_{n-1}.
\end{align}
The induced curves are given by $(X_\epsilon(w,t;\xi),t)$, with
\begin{align}
    X_\epsilon(w,t;\xi)= \left(w_1-t^2\xi_1,\frac{1}{\epsilon}\arctan\left(\frac{\epsilon w_2}{t}\right)-t\xi_2\right).
\end{align}
Similarly to \eqref{260604le2-33}, $(X_\epsilon(w,t;\xi),t)$ is not direction-equivalent to a family of lines.\\

Now we take $\epsilon\to 0$. Then
\begin{align}
    X_\epsilon(w,t;\xi)\to X_0(w,t;\xi)= \left(w_1-t^2\xi_1,\frac{ w_2}{t}-t\xi_2\right)
\end{align}
For $t\sim 1$, we take a diffeomorphism $h(x_1,x_2,t)= (x_1,x_2 t,t^2)$. Then the family of curves $(X(w,t;\xi),t)$ is direction-equivalent to 
\begin{align}
    h(X(w,t;\xi),t)= (w_1-t^2\xi_1,w_2-t^2\xi_2,t^2),
\end{align}
a family of lines with directions $(\xi_1,\xi_2,1)$.
\end{example}

In the example above, we approximated a phase function satisfying the Katz-Wolff condition with a sequence of phase functions failing the Katz-Wolff condition within Bourgain's condition. It is therefore tempting to conjecture that the Katz‑Wolff condition generically fails within Bourgain's condition, i.e., the set of phase functions that violate the Katz‑Wolff condition is open and dense in the space of phase functions satisfying Bourgain's condition. We expect that a further development along the lines of Theorem \ref{B.C.Tao} will be useful in obtaining this structural information about Bourgain's condition.

\subsection{Bourgain's Kakeya-compression example is projectively flat and has extra-worst compression}\label{250801subsection2_4}

As mentioned above Definition \ref{250801defi1_17zz} that Bourgain's example \eqref{250930e1_46} is projectively flat (see \eqref{260526e1_53}) and satisfies the extra-worst compression condition. Here we explain more details.

\begin{example}
We take Bourgain's example as in \eqref{250930e1_46}: 
\begin{equation}\label{250930e2_27}
\Phi(w, t; \xi):=
(w_1+ t\xi_2+ t^2\xi_1, w_2+ t\xi_1, t).
\end{equation}
\end{example}

Bourgain's example admits a two-dimensional Kakeya set. Moreover, it is not difficult to show that in $\R^3$, the dimension of a curved Kakeya set is always $\ge 2$. In this sense we say that Bourgain's example allows certain ``worst" behavior. It is perhaps a bit surprising that Bourgain's example is projectively flat, which looks very much like certain ``very nice" property. Indeed, we will show now that Bourgain's example is worse than we thought, in the sense that every totally geodesic surface gives rise to a $\Phi$-Kakeya set. This explains why in Definition \ref{250801defi1_17zz} we introduced the notion ``extra-worst" compression condition. \\

To prove that Bourgain's example \eqref{250930e2_27} has extra-worst compression, it suffices to show that every plane in $\R^3$ gives rise to a $\widetilde{\Phi}$-Kakeya set, where $\widetilde{\Phi}$ is defined in \eqref{260526e1_53}. We can always write a plane $P$ in this form
\begin{equation}
    ay_1+ by_2+ cy_3 =d.
\end{equation}
If $P$ contains $\widetilde{\Phi}(w,t;\xi)$ for some $w$ and $\xi$, then $a\neq 0$ and 
\begin{equation}
    a (w_1 - tw_2 + t\xi_2)+ b (w_2+ t\xi_1)+ c t= d.
\end{equation}
Extracting the coefficients of $t$ implies
\begin{equation}\label{20251125e-2.46}
    \begin{cases}
    aw_1 + bw_2 =d,  \\
    a(-w_2+ \xi_2)+ b\xi_1+ c=0.
\end{cases}
\end{equation}
We solve this linear system and write $w$ as a function of $\xi$
\begin{equation}\label{20251125e-2.47}
    \begin{cases}
    w_1 = -\frac{b^2}{a^2}\xi_1- \frac{b}{a}\xi_2+ \frac{ad-bc}{a^2}, \\
    w_2 = \frac{b}{a}\xi_1+ \xi_2+ \frac{c}{a}.
\end{cases}
\end{equation}
As a result, we see that for any given $\xi=(\xi_1,\xi_2)$ we can find $w$ given by \eqref{20251125e-2.47}, such that
\begin{equation}
    \widetilde{\Phi}(w,t;\xi)\in P, \quad \forall t.
\end{equation}
This proves that $P$ is a $\widetilde{\Phi}$-Kakeya set, and hence Bourgain's example \eqref{250930e2_27} has extra-worst compression.

\subsection{Wisewell's parabolic Nikodym set is projectively equivalent to the standard Nikodym set}

In Wisewell's PhD thesis \cite{Wis03} (see also Wisewell \cite{Wis05}), one of the curved Kakeya/Nikodym problems she studied was the following parabolic Nikodym problem. 
\begin{example}\label{250930ex2_12}
Let $E\subset \R^3$ be a Borel set with $\mathcal{L}^3(E)=0$. If $E$ satisfies that for every $(w_1, w_2)$, there exists $(\xi_1, \xi_2)$ such that 
\begin{equation}\label{250930e2_29z}
(w_1+ t\xi_2+ t^2\xi_1, w_2+ t\xi_1, t)\in E, 
\end{equation}
for every $t\in [0, 1]$, then we say that $E$ is a Wisewell's parabolic Nikodym set. 
\end{example}

Note that curves in Wisewell's parabolic Nikodym set are exactly the same as in Bourgain's example \eqref{250930e2_27}. The only difference is that the roles of $(w_1, w_2)$ and $(\xi_1, \xi_2)$ are exchanged: In Bourgain's example, we study curved Kakeya sets, which means that we pick $(w_1, w_2)$ depending on $(\xi_1, \xi_2)$; in Wisewell's setting, we pick $(\xi_1, \xi_2)$ depending on $(w_1, w_2)$.

Recall from equation \ref{250930e2_27} that Bourgain's example admits two dimensional curved Kakeya sets. One question Wisewell asked in \cite{Wis03} was: How about the dimension of the Nikodym sets in Example \ref{250930ex2_12}? The crucial observation that Wisewell made was that Wolff's hairbrush argument\footnote{She indeed applied Katz's version of the hairbrush argument, see \cite{Wis03}.} in \cite{Wol95} can be applied in her curved setting, and she concluded that the Nikodym sets in Example \ref{250930ex2_12} have Hausdorff dimension $\ge 5/2$. In other words, for the same family of curves, its associated Kakeya problem may allow certain bad behavior, but its associated Nikodym problem may have certain ``best possible" properties. This philosophy played crucial roles in the work \cite{CGGHIW24} and its applications to the Pierce-Yung problem in \cite{BGH25}. \\

That Wisewell \cite{Wis03} observed that Wolff's hairbrush argument can be applied in her setting, if translated into our language, means that her Nikodym problem satisfies the Katz-Wolff condition. Now if we think of this fact from the perspectives of the projective geometry of paths, we will very quickly recover Wisewell's $5/2$-result: Recall from Theorem \ref{thm: main} that the Katz-Wolff condition implies the projective flat condition. In other words, we are able to find a change of coordinates to turn the curves in Example \ref{250930ex2_12} into lines. After realizing this, it is very easy to find the correct change of coordinates, and we just subtract the first entry in \eqref{250930e2_29z} by the product of the second and third entries. This turns the family of curves in \eqref{250930e2_29z} into
\begin{equation}\label{250930e2_29zz}
(w_1+ t\xi_2-t w_2, w_2+ t\xi_1, t). 
\end{equation}
Next, we apply the change of variables 
\begin{equation}
\xi_2\mapsto \xi_2+ w_2, \ \ \xi_1\mapsto \xi_1,
\end{equation}
which is allowed for Nikodym problems (but not for Kakeya problems), and we obtain the family of curves 
\begin{equation}
(w_1+ t\xi_2, w_2+ t\xi_1, t). 
\end{equation}
This is precisely the classical Nikodym problem. \footnote{It is also very helpful to check directly that Example \ref{250930ex2_12} satisfies both Bourgain's condition and the totally geodesic condition, and therefore also satisfies the Katz-Wolff condition. This was more or less how we realized that Wisewell's Nikodym problem was equivalent to the classical Nikodym problem. }

\subsection{Hilbert's fourth problem}

How Hilbert formulated his fourth problem was quite vague. The formulation we adopt here comes from projective geometry (see for instance \cite[Section 12.2]{Shen01}):  Characterize all metrics under which all the geodesics locally coincide with Euclidean lines. \\

There are many important examples that are solutions to Hilbert's fourth problem, including Hilbert's metric, Funk's metric, among others. Here let us pick Funk's metric to explain the connections to our current discussions.

Let $\Omega\subset \R^n$ be a bounded strongly convex domain with smooth boundary $\partial \Omega$. Take two points $p, q\in \Omega$ with $p\neq q$. We define the distance function (not symmetric)
\begin{equation}
d(p, q):= \ln \frac{|z-p|}{|z-q|},
\end{equation}
where $z$ is the intersection point of the ray 
\begin{equation}
p+ t(q-p), \ \ t\ge 0,
\end{equation}
with the boundary $\partial \Omega$. \\

This distance function comes from the Funk metric for the domain $\Omega$. Moreover, all the geodesics under this metric are Euclidean lines.  We refer to \cite[Section 4.3]{Shen01} for the proof and for more discussions on Funk metrics. \\

We consider the Carleson-Sj\"olin operator (see for instance \cite[Section 1.3]{DGGZ24}) for the Funk metric on the domain $\Omega$, given by 
\begin{equation}\label{251117e2_35}
\int_{\Omega} e^{i N d(p, q)}f(q) a(p, q) dq,
\end{equation}
where $N\in \R$, $a(\cdot, \cdot)$ is a smooth function supported in $\Omega\times \Omega$ and supported away from the diagonal $p=q$, and $dq$ is given by the Euclidean Lebesgue measure. We would like to view \eqref{251117e2_35} as a H\"ormander-type operator. To do so, we follow \cite{DGGZ24} and consider its reduced Carleson-Sj\"olin operator. More precisely, let $\Omega'\subset \Omega$ be a smooth hyper-surface and let $\Omega_0\subset \Omega$ be a region with 
\begin{equation}
d(\Omega', \Omega_0)>0.
\end{equation}
Consider 
\begin{equation}\label{251117e2_36}
\int_{\Omega'} e^{i N d(p, q')}f(q') a(p, q') dq',
\end{equation}
where $a(p, q')$ is a smooth function supported on $\Omega_0\times \Omega'$, and $dq'$ is given by (for instance) the Euclidean surface measure. \\

It is not difficult to show (for instance one can repeat the proof in \cite[Section 5.2]{DGGZ24}) that $d(p, q')$ satisfies H\"ormander's condition in Definition \ref{250726defi1_1}, that is, \eqref{251117e2_36} is a H\"ormander-type operator. Moreover, one can repeat the proof of Lemma 3.2 in \cite{DGGZ24} and see that all characteristic curves for the phase function $d(p, q')$ are straight. In other words, the curved Kakeya problem associated to the phase function $d(p, q')$ is exactly the standard Kakeya problem (for straight lines).

\section{Proof of Proposition \ref{prop: Kakeya to spray space}}\label{260526section3}

Let us begin by giving a more intuitive explanation for spray spaces.

 \begin{definition}[Conic systems of paths, Shen \cite{Shen01}]\label{250801defi1_17} Let $M$ be a manifold ($M=\R^n$ for us), and let $C_x\subset T_x M$ be an open cone of the tangent space $T_x M$ at $x\in M$.  
 Let $\mathcal{G}$ be a collection of $C^{\infty}$ parametrized curves $\sigma: (a, b)\to M$ with the following properties:
 \begin{enumerate}
 \item[(i)] (Existence) For every vector $y\in C_x\subset T_x M$ and every $t_{\circ}\in (a, b)$, there is a curve $\gamma: (a, b)\to M$ in $\mathcal{G}$ with 
 \begin{equation}
 \dot {\gamma}(t_{\circ})=y.
 \end{equation}
 \item[(ii)] (Uniqueness) For two arbitrary curves $\gamma$ and $\sigma$, if at some $t_0\in (a, b)$ and $t_1\in (a, b)$, it holds that 
 \begin{equation}
 \dot{\gamma}(t_0)=\dot{\sigma}(t_1),
 \end{equation}
 then 
 \begin{equation}
 \gamma(t_0+ t)= \sigma(t_1+ t),
 \end{equation}
 for all 
 \begin{equation}
 t\in (a-t_0, b-t_0)\cap (a-t_1, b-t_1).
 \end{equation}
 \item[(iii)] (Invariance) For every curve $\gamma: (a, b)\to M$ in $\mathcal{G}$ and every $t_{\circ}\in \R, \lambda>0$, the curve 
 \begin{equation}
 \widetilde{\gamma}(t):= 
 \gamma(\lambda t+ t_{\circ}),
 \end{equation}
 with 
 \begin{equation}
 \frac{a-t_{\circ}}{\lambda}< t< \frac{b-t_{\circ}}{\lambda}
 \end{equation}
 is still in $\mathcal{G}$. 
 \end{enumerate}

 \end{definition}

 Definition \ref{250801defi1_17} is a slight modification of the definition of the systems of paths (path spaces) in 
   \cite[Section 4.1]{Shen01}, in the sense that in Definition \ref{250801defi1_17} we only take a cone $C_x$ instead of taking the entire $T_x M$. \\

There is a one-to-one correspondence between spray spaces and conic systems of paths.

We first start with a spray space in Definition \ref{260526defi1_19}, and construct a conic system of paths. Let $U\subset \R^n$ be a small open region, and let $C\subset \R^n$ be an open cone. Let
\begin{equation}
G^i: U\times C\to \R
\end{equation}
be the smooth function given in Definition \ref{260526defi1_19}. Write 
\begin{equation}
x=(x^1, x^2, \dots, x^n).
\end{equation}
Consider the system of ODEs given by
\begin{equation}\label{250710ea_5}
\frac{d^2 x^i}{dt}+ 2 G^i(x, \frac{dx}{dt})=0, \ \ i=1, 2, \dots, n.
\end{equation}
Note that if 
\begin{equation}
x(t)= (x^1(t), \dots, x^n(t))
\end{equation}
is a solution to \eqref{250710ea_5}, then by the homogeneity assumption  \eqref{260526e1_60zz}, we know that 
\begin{equation}
\widetilde{x}(t):=x(\lambda t)
\end{equation}
is also a solution to  \eqref{250710ea_5} for every $\lambda>0$. This, combined with some elementary calculations, shows that the solution curves to \eqref{250710ea_5} with an appropriately chosen region of initial data satisfy the requirements of Definition \ref{250801defi1_17}.  \\

Next, we start with a conic system of paths, and try to construct a spray space. This is due to Douglas \cite{Dou27}.  Let us briefly sketch the proof here. Our strategy is to use the paths in Definition \ref{250801defi1_17}, write down the ODEs they satisfy and find the functions $G^i$ satisfying the homogeneity condition 
\begin{equation}\label{260527e3_12}
    G^i(x,\lambda y) = \lambda^2 G^i(x,y), \ \forall \lambda>0.
\end{equation}
Once these $G^i$ are found, we can define spray spaces. More precisely, let 
\begin{equation}
\bfG:=(
y^1, \dots, y^n, -2G^1(x, y), \dots, -2G^n(x, y)
)
\end{equation}
with 
\begin{equation}
y=(y^1, \dots, y^n),
\end{equation}
be a vector field on $U\times C$ where $U\subset \R^n$ is a small open region and $C\subset \R^n$ is an open cone. We often write 
\begin{equation}\label{250710ea_11}
\bfG= 
y^i \frac{\partial}{\partial x^i}- 2 G^i(x, y) \frac{\partial}{\partial y^i},
\end{equation}
where we use Einstein's summation convention. Note that $x(t)$ is a solution to \eqref{250710ea_5} if and only if its lift 
\begin{equation}
\widehat{x}(t):= \Big(x(t), \frac{dx}{dt}(t)\Big)
\end{equation}
is an integral curve of $\bfG$ in $U\times C$. This finishes defining the spray space, and explains the geodesic equation in \eqref{260526e1_61a} as well. \\

So far we have finished the discussion on the geometry of spray spaces. At the end of this section, we will prove Proposition \ref{prop: Kakeya to spray space}. The proof is actually part of the steps in \eqref{260527e3_12}--\eqref{250710ea_11}, and we take this opportunity to write down more details. Let us remark that the proof below is also due to Douglas \cite{Dou27}, and it is indeed exactly how Douglas \cite{Dou27} proved that every system of paths induces a spray space.

  We start with normal form for the curved Kakeya problem. Consider the map 
\begin{equation}\label{260527e3_17}
(X(w, t; \xi), t),
\end{equation}
for some map 
\begin{equation}
X: \R^{n-1}\times \R\times \R^{n-1}\to \R^{n-1}
\end{equation}
of the form 
\begin{equation}\label{250227e6_3}
X(w, t; \xi)= 
w+ t(\xi+ O(|w||\xi|))+ t^2 O(|\xi|)+ t^3 O(|\xi|)+\dots.
\end{equation}
Here both $O(|w||\xi|)$ and $O(|\xi|)$ depend on $w, \xi$. It is elementary to see that the collection of curves given by \eqref{260527e3_17} forms a conic system of paths. Our goal is to convert the above curves into the form 
\begin{equation}\label{250227e6_4}
\partial_t^2 \mathbf{x}= G(\mathbf{x}, \mathbf{p}), \ \mathbf{p}:= \partial_t \mathbf{x},
\end{equation}
where 
\begin{equation}
\bfx: \R^{n-1}\times \R\times \R^{n-1}\to \R^n,
\end{equation}
with appropriate initial data, and the function $G$  is homogeneous of degree two in $\bfp$. \\

Let $\alpha, \beta\in \R$ be two parameters that will be chosen later. Consider the new parametrization 
\begin{equation}\label{250227e6_12}
\mathbf{x}= (X(w, \alpha t+\beta; \xi), \alpha t+\beta).
\end{equation}
We differentiate \eqref{250227e6_12} twice, and obtain 
\begin{equation}\label{250227e6_13}
\begin{split}
& \mathbf{x}= (X(w, \alpha t+\beta; \xi), \alpha t+\beta),\\
& \mathbf{p}= \alpha\pnorm{(\partial_t X)(
w, \alpha t+\beta; \xi
), 1
}, \\
& \partial_t^2 \mathbf{x}= \alpha^2 \pnorm{(\partial^2_t X)(
w, \alpha t+\beta; \xi
), 0
}. 
\end{split}
\end{equation}
Consider the first two equations in \eqref{250227e6_13}. We would like to express $\alpha, \alpha t+\beta, w, \xi$ in terms of $\bfx$ and $\bfp$, and note that the number of ``variables" matches exactly. More precisely, we obtain 
\begin{equation}
\begin{split}
& x_n= \alpha t+\beta, \ p_n= \alpha;\\
& (x_1, \dots, x_{n-1})= X(w, x_n; \xi), \ (p_1, \dots, p_{n-1})= p_n (\partial_t X)(
w, x_n; \xi
).
\end{split}
\end{equation}
This is always solvable because of H\"ormander's condition for $X$ (from the normal form of $X$ we can also see easily that this has a unique solution). Once we express $\alpha, \alpha t+\beta, w, \xi$ using $\bfx$ and $\bfp$, we substitute into the last equation in \eqref{250227e6_13}, and obtain the function $G$ as desired in \eqref{250227e6_4}. That $G$ is homogeneous of degree two in $\bfp$ is immediate. This finishes the proof of the proposition. 

\section{Proof of Theorem \ref{thm: main}: The case of phase functions}

In this section we will prove Theorem \ref{thm: main} for the case where the curved Kakeya problem arise from a phase function. Item \ref{251003thm_item1} and Item \ref{251003thm_item2} do not involve phase functions directly, and we will prove them later for $\Phi$-Kakeya problems directly (see Section \ref{251003sec9} and Section \ref{251003sec8}). Moreover, for Item \ref{251003thm_item4}, if a phase function $\phi$ satisfies both the totally geodesic condition and Bourgain's condition, then we know (see Theorem \ref{250727theorem1_13}) that the induced $\Phi$-Kakeya problem also satisfies both the totally geodesic condition and Bourgain's condition. As a consequence, Item \ref{251003thm_item4} for the phase function $\phi$ follows from Item \ref{251003thm_item4} for the induced map $\Phi$, which will be proven in Section \ref{251003secc6} later. 

  In the rest of this section, we will prove Item \ref{251003thm_item3} and  Item \ref{251003thm_item5}.

\subsection{Katz-Wolff implying Bourgain}\label{25-Sec4.1}

Let $\phi$ be a phase function satisfying H\"ormander's condition. We assume that $\phi$ satisfies the Katz-Wolff condition, and our goal is to prove that it also satisfies Bourgain's condition. \\

Recall the Katz-Wolff condition. For all choices of $\bfx_0=(x_0, t_0)\in \B^{n}_{\epsilon_{\phi}}$ and $y_1\neq y_2$, and for all 
\begin{equation}
t_1, t_2\neq t_0,
\end{equation}
let $x_i\in \R^{n-1}$ be the unique point such that 
\begin{equation}
(x_i, t_i)\in \Gamma_{y_i}(\bfx_0), \ \ i=1, 2.
\end{equation}
Let $y(t_1, t_2)$ and $w(t_1, t_2)$ be the unique (if exist) solutions satisfying that the characteristic curve 
\begin{equation}
\Gamma_{y(t_1, t_2)}(w(t_1, t_2))
\end{equation}
passes through the two points $(x_1, t_1)$ and $(x_2, t_2)$. Then that $\phi$ satisfies the Katz-Wolff condition means that   
 \begin{equation}
        \rank[\partial_{t_1} y(t_1, t_2), \partial_{t_2} y(t_1, t_2)]<2, \ \ \forall t_1, t_2.
    \end{equation} 
In other words, the set 
\begin{equation}\label{250510e2_5}
\{
y(t_1, t_2): t_1, t_2\neq t_0
\}
\end{equation}
forms a curve. \\

We without loss of generality assume that $\phi$ is of a normal form at the origin: 
    \begin{align}\label{250512e_normal}
        \phi(x,t;y) = \langle x,y\rangle + (t/2)\langle Ay,y\rangle + O(|t| |y|^3 + |(x,t)|^2 |y|^2). 
    \end{align}
Moreover, in the above definition of the Katz-Wolff condition,  we take $y_1=0$, $y_2=z$ for some $z\neq 0$,
 and put $\bfx_0= (x_0,t_0)= (0,0)$, the origin in $\R^n$.
Note that under the above assumptions, we write the characteristic curves as $$
    \Gamma_{y_1}(\bfx_0)= \Gamma_{0}(0,0),\quad \Gamma_{y_2}(\bfx_0)= \Gamma_{z}(0,0).
$$ 
Let 
\begin{equation}
    y(z; t_1, t_2)
\end{equation}
be such that the characteristic curve $\Gamma_{y(z; t_1, t_2)}((0, t_1))$ intersects both the curve $\Gamma_{0}(0,0)$ at the height $t_1$, which is the vertical line, and the curve $\Gamma_{z}(0,0)$ at the height $t_2$. By the assumption that $\phi$ satisfies the Katz-Wolff condition, we know that the set
\begin{equation}
\{y(z; t_1, t_2): t_1, t_2\neq 0\}
\end{equation}
forms a curve, for all choices of $z\neq 0$.  In other words, 
    \begin{align}\label{eq: expr 1}
        \rank[\partial_{t_1} y(z; t_1, t_2), \partial_{t_2} y(z; t_1, t_2)] < 2,
    \end{align}
    for every $z\neq 0$, and every  $t_1, t_2\neq 0$. For convenience, we always write
    \begin{equation}
        \hat{y}=y(z; t_1, t_2).
    \end{equation}

Recall the definition of characteristic curves
\begin{equation}
    \Gamma_y(\bfx):=\left\{\bfx' \in \R^n\cap \B^n_{2\epsilon_{\phi}}: \nabla_y \phi(\bfx'; y)=\nabla_y \phi(\bfx; y)\right\}.
\end{equation}
For any $\tau_1,\tau_2$ and $y$, we define $X=X(y,\tau_1,\tau_2)$ by the following equation
\begin{equation}\label{eq: impl}
     \nabla_y \phi(X,\tau_2; y) = \nabla_y \phi(0,\tau_1; y).
\end{equation}
Note that $\{(X(y,\tau_1,\tau_2),\tau_2),\ |\tau_2|\le \epsilon_\phi\}$, is a parameterization of characteristic curve $\Gamma_y(0,\tau_1)$.
\begin{claim}\label{250512claim2_1}
    Under the above notation, if we fix $\tau_2\neq 0$, and pick $y$ and $\tau_1$ sufficiently close to $0$, then we have that 
    \begin{equation}
    \nabla_y X(y,\tau_1,\tau_2)
    \end{equation}
    is non-singular. 
    \end{claim}
\begin{proof}[Proof of Claim \ref{250512claim2_1}]
We differentiate both sides of \eqref{eq: impl} in the $y$ variable, and obtain 
\begin{equation}\label{251208e4-14}
\nabla_x\nabla_y \phi(X, \tau_2; y)
\nabla_y X
+ \nabla^2_y \phi(X, \tau_2; y)= \nabla^2_y \phi(0, \tau_1; y).
\end{equation}
It suffices to show that 
\begin{equation}\label{251003e10_19}
\nabla^2_y \phi(X, \tau_2; y)- \nabla^2_y \phi(0, \tau_1; y)
\end{equation}
is non-singular whenever $\tau_2\neq 0$ is fixed and $y, \tau_1$ are sufficiently small. When $y, \tau_1$ are sufficiently small, $X(y,\tau_1,\tau_2)$ will be as well. That \eqref{251003e10_19} is non-singular is immediate once we observe that $\phi$ is of a normal form at the origin. This finishes the proof of the claim. 
\end{proof}

By the same token, we can write
\begin{align}
    \Gamma_{\hat{y}}(0,t_1)=\{ (X(\hat{y},t_1,t),t):\ |t|\le \epsilon_{\phi}\},\\
    \Gamma_{z}(0,0)=\{ (X(z,0,t),t),\ |t|\le \epsilon_{\phi}\}.  \quad
\end{align}
Since $\Gamma_{\hat{y}}(0,t_1)$ intersects $\Gamma_{z}(0,0)$ at the height $t_2$, we see
\begin{equation}\label{eq: y hat impl}
    X(\hat{y},t_1,t_2)= X(z,0,t_2).
\end{equation}
 Differentiate \eqref{eq: y hat impl} to find $\partial_{t_1} \hat y$, $\partial_{t_2} \hat y$: 
    \begin{align}\label{250512e2_1516}
        \partial_{t_1} \hat y &= -\nabla_y X(\hat{y},t_1,t_2)^{-1} \partial_{\tau_1} X(\hat{y},t_1,t_2) \\
        \partial_{t_2} \hat y &= \nabla_y X(\hat{y},t_1,t_2)^{-1} [\partial_{\tau_2} X(z,0,t_2) - \partial_{\tau_2} X(\hat{y},t_1,t_2)].
    \end{align}
Note that the matrix $\nabla_y X(\hat{y},t_1,t_2)$ is non-singular by Claim \ref{250512claim2_1}, if we fix $t_2\neq 0$ and take $z,\ t_1$ sufficiently close to 0. Factoring it out and by \eqref{eq: expr 1}, we see\footnote{Claim \ref{250512claim2_1} requires that $y$ and $\theta'$ are sufficiently small, and later in \eqref{251003e10_25zz} we will see that this is indeed what we have. } that if we denote  
    \begin{align}
        U &:=\partial_{\tau_1} X(\hat y, t_1, t_2) \\
        V &:=\partial_{\tau_2} X(z, 0, t_2) - \partial_{\tau_2} X(\hat y, t_1, t_2),
    \end{align}
    then we must have 
    \begin{equation}\label{251003e10_22zz}
            \rank[U, V] < 2.
    \end{equation}
 To simplify notation, we will write \eqref{251003e10_22zz} as 
 \begin{equation}
 \det(U, V)=0,
 \end{equation}
 where 
 $\det(U, V)$ can be understood as the column vector whose entries are the determinants of all possible $2\times 2$ sub-matrices of the $(n-1)\times 2$ matrix $[U, V]$. \\

 Recall the product rule 
 \begin{equation}
 \det(U, V)' = \det(U, V') + \det(U', V).
 \end{equation}
We differentiate $\det(U, V)$ in $t_1$ and obtain 
\begin{equation}\label{250512e2_22z}
    0= \det(U, \partial_{t_1} V)+  \det(\partial_{t_1}U, V).
\end{equation}
We freeze $t_2$ and take the limit $t_1\to 0$. When taking this limit, we see that 
\begin{equation}\label{251003e10_25zz}
\lim_{t_1\to 0} y(z; t_1, t_2)= z,
\end{equation}
and therefore 
\begin{equation}
\lim_{t_1\to 0}V  = 0.
\end{equation}
When taking the limit $t_1\to 0$ on both sides of \eqref{250512e2_22z}, we obtain 
\begin{equation}\label{251208e4-29}
0=\lim_{t_1\to 0} \det(U, \partial_{t_1} V).
\end{equation}
 To compute the right hand side, we need to differentiate $V$: 
\begin{align}
    \partial_{t_1} V
    &=-\nabla_y \partial_{\tau_2} X(\hat{y}, t_1, t_2) \partial_{t_1} \hat{y}- \partial_{\tau_2}\partial_{\tau_1} X(\hat{y}, t_1, t_2)\notag\\
     &= \nabla_y \partial_{\tau_2} X(\hat{y}, t_1, t_2)  (\nabla_y X(\hat{y}, t_1, t_2))^{-1} \partial_{\tau_1} X(\hat{y}, t_1, t_2) - \partial_{\tau_2}\partial_{\tau_1} X(\hat{y}, t_1, t_2),
\end{align}
where we applied \eqref{250512e2_1516}. 
So it follows from \eqref{251208e4-29} that
\begin{align}\label{eq: expr 2}
    0= 
    \det\big(\partial_{\tau_1} X(z, 0, t_2),\ \nabla_y \partial_{\tau_2} X(z, 0, t_2)  
     (\nabla_y X(z, 0, t_2))^{-1} \partial_{\tau_1} X(z, 0, t_2)\notag\\ - \partial_{\tau_1} \partial_{\tau_2} X(z, 0, t_2)\big).
\end{align}
This says that there exists a scalar function $\lambda(z,t_2)$ such that
\begin{align}\label{251208e4-32}
    \nabla_y \partial_{\tau_2} X(z, 0, t_2)  
     (\nabla_y X(z, 0, t_2))^{-1}& \partial_{\tau_1} X(z, 0, t_2) - \partial_{\tau_1} \partial_{\tau_2} X(z, 0, t_2)\notag\\
     &= \lambda(z,t_2)\partial_{\tau_1} X(z, 0, t_2).
\end{align}

We implicitly differentiate \eqref{eq: impl} multiple times and use the fact that $\phi$ is of normal form. Then it is not hard to determine the relevant derivatives of $X$,
\begin{align*}
    \partial_{\tau_1} X(z,0,t_2)= A z+ O(|z|^2),\\
    \partial_{\tau_2} \partial_{\tau_1} X(z,0,t_2)= O(|z|^2),\\
    \nabla_y X(z,0,t_2)= t_2A+ O(|t_2|^2)+ O(|z|), \\
    \partial_{\tau_2}\nabla_y X(z,0,t_2)= A+ O(|t_2|)+ O(|z|).
\end{align*}
Fix a unit vector $z_0$. Let $r$ be a scalar and take $z=rz_0$. Then \eqref{251208e4-32} becomes
\begin{align}\label{251208e4-33}
    \nabla_y \partial_{\tau_2} X(rz_0, 0, t_2)  
     (\nabla_y X(rz_0, 0, t_2))^{-1}& \partial_{\tau_1} X(rz_0, 0, t_2) - \partial_{\tau_1} \partial_{\tau_2} X(rz_0, 0, t_2)\notag\\
     &= \lambda(rz_0,t_2)\partial_{\tau_1} X(rz_0, 0, t_2).
\end{align}
To make things cleaner, we denote that
\begin{align}
    & W(rz_0,t_2):=\text{ the left hand side of \eqref{251208e4-33}},   \\
    & Q(t_2):= \nabla_y \partial_{\tau_2} X(0,0,t_2)  
     (\nabla_y X(0,0,t_2))^{-1}.
\end{align}
Plug all these derivatives of $X$ into \eqref{251208e4-33}, then we can see that the limit of $\lambda(rz_0,t_2)$ exists, as $r\to 0$,
\begin{align}
    \lim_{r\to 0} |\lambda(rz_0,t_2)| &=  \lim_{r\to 0} \frac{|W(rz_0,t_2)|}{|\partial_{\tau_1} X(rz_0,0,t_2)|} \notag \\
     &= \lim_{r\to 0} \frac{|Q(t_2)rAz_0|}{|rAz_0|}= \frac{|Q(t_2)Az_0|}{|Az_0|}.
\end{align}
Denote this limit by $\lambda_0(t_2)$. We also need to see whether $|\partial_r\big(\lambda(rz_0,t_2)\big)|$ is bounded near 0. This is indeed the case, since
\begin{align}
    \lim_{r\to 0} \partial_r |\lambda(rz_0,t_2)| &= \lim_{r\to 0} \partial_r\frac{|W|}{|\partial_{\tau_1} X|} \notag \\
     &= \lim_{r\to 0} \frac{\partial_r|W||\partial_{\tau_1} X|-|W|\partial_r |\partial_{\tau_1} X|}{|\partial_{\tau_1} X|^2} \\
     &\leq \lim_{r\to 0} \frac{O(r^2)}{r^2|Az_0|^2},
\end{align}
where the last inequality follows from that 
\begin{align}
    \partial_r|W||\partial_{\tau_1} X|= |Q(t_2)Az_0||Az_0||r|+O(r^2), \\
    |W|\partial_r |\partial_{\tau_1} X|= |Q(t_2)Az_0||Az_0||r|+O(r^2).
\end{align}

If we take the limit $r\to 0$, then both sides of \eqref{251208e4-33} vanish. To deal with this, we differentiate \eqref{251208e4-33} in $r$ and next set $r=0$,
\begin{align}
    \nabla_y \partial_{\tau_2} X(0, 0, t_2)  
     (\nabla_y X(0, 0, t_2))^{-1} Az_0 = \lambda_0(t_2)Az_0.
\end{align}
Since $A$ is non-degenerate and $z_0$ is arbitrary, we see that
\begin{align}\label{251208e4-36}
    \nabla_y \partial_{\tau_2} X(0, 0, t_2)  
       = \lambda_0(t_2)\nabla_y X(0, 0, t_2).
\end{align}
If we set $y=0$, $\tau_1=0$ and $\tau_2=t_2$ in \eqref{251208e4-14}, then we will see
\begin{equation}
    \nabla_y X(0, 0, t_2)= -\nabla_y^2 \phi(0,t_2;0).
\end{equation}
Let $M(t) := \partial_y^2 \phi(0,t,0)$ and $t_2=t$. Then \eqref{251208e4-36} says
\begin{equation}
    \partial_t M(t)  
       = \lambda_0(t)M(t).
\end{equation}
Write 
\begin{equation}
M(t) = tA + t^2B + \cdots.
\end{equation}
 Finally take one more derivative in $t$ and set $t=0$ to get $B$ is a multiple of $A$. This finishes verifying Bourgain's condition.

\subsection{Totally geodesic and failing Bourgain's condition means extra-worst compression}

The goal of this subsection to show that if a phase function $\phi$ satisfies the totally geodesic condition, and fails Bourgain's condition in Definition \ref{230903definition1_4}, then $\phi$ satisfies the extra-worst compression condition. We will reduce the problem to the case of $\Phi$-Kakeya problems, which is Item \ref{251003thm_item5} for $\Phi$-Kakeya problems, and will be proven later in Section \ref{251003sec7}. \\

Let us be more precise. 
The characteristic curves of $\phi$ are given by 
\begin{equation}
\Phi(w, t; \xi)= (X(w, t; \xi), t)
\end{equation}
where $X(w, t; \xi)$ satisfies 
\begin{equation}\label{251004e4_55}
\nabla_{\xi} \phi(X(w, t; \xi), t; \xi)=w.
\end{equation}
Lemma \ref{251004lemma4_2} below says that if $\phi$ fails Bourgain's condition in Definition \ref{230903definition1_4}, then $\Phi$ also fails Bourgain's condition in Definition \ref{BourCond}, and therefore the problem is reduced to proving that if $\Phi$ satisfies the totally geodesic condition and fails Bourgain's condition in Definition \ref{BourCond}, then $\Phi$ satisfies the extra-worst compression condition, which is Item \ref{251003thm_item5} for $\Phi$-Kakeya problems, and will be proven later in Section \ref{251003sec7}.

\begin{lemma}\label{251004lemma4_2}
Let $\phi(x, t; \xi)$ be a phase function satisfying H\"ormander's condition. 
Let us assume that $X$ satisfies Bourgain's condition in Definition \ref{BourCond}. Then $\phi$ also satisfies Bourgain's condition, which is given by Definition \ref{230903definition1_4}. 
\end{lemma}

\begin{proof}[Proof of Lemma \ref{251004lemma4_2}]
We without loss of generality assume that $\phi$ is of a normal form at the origin, and our goal is to show that $\phi$ satisfies Bourgain's condition at the origin. By definition, that $X$ satisfies Bourgain's condition in Definition \ref{BourCond} means that 
\begin{equation}\label{251004e4_56}
\nabla_{\xi} \partial_t^2 X
\end{equation}
is a constant multiple of 
\begin{equation}\label{251004e4_57}
\nabla_{\xi} \partial_t X.
\end{equation}
Let us compute  \eqref{251004e4_56} and \eqref{251004e4_57}. By taking derivatives in $\xi$ on both sides of \eqref{251004e4_55}, we obtain 
\begin{equation}\label{251004e4_58}
\nabla^2_{\xi} \phi(X, t; \xi)+
\nabla_x 
 \nabla_{\xi} \phi(X, t; \xi) \nabla_{\xi} X= 0.
\end{equation}
We differentiate \eqref{251004e4_58} in $t$: 
\begin{multline}\label{251004e4_59}
\partial_t \nabla^2_{\xi} \phi(X, t; \xi)
+
\nabla_x \nabla^2_{\xi} \phi(X, t; \xi) \partial_t X
+
\partial_t \nabla_x 
 \nabla_{\xi} \phi(X, t; \xi) \nabla_{\xi} X
 +\\
 \nabla^2_x 
 \nabla_{\xi} \phi(X, t; \xi) \nabla_{\xi} X \partial_t X+
 \nabla_x 
 \nabla_{\xi} \phi(X, t; \xi) \partial_t \nabla_{\xi} X
 = 0.
\end{multline}
At the origin, we obtain 
\begin{equation}\label{251004e4_60}
\partial_t\nabla_{\xi}X= -\partial_t \nabla^2_{\xi} \phi.
\end{equation}
We differentiate \eqref{251004e4_59} in $t$, evaluate at the origin, and only collect the non-zero terms:
\begin{equation}
\partial^2_t \nabla^2_{\xi} \phi(X, t; \xi)+  \nabla_x 
 \nabla_{\xi} \phi(X, t; \xi) \partial^2_t \nabla_{\xi} X=0.
\end{equation}
This, together with \eqref{251004e4_60}, implies that 
\begin{equation}
\partial^2_t \nabla^2_{\xi} \phi
\end{equation}
is a constant multiple of 
\begin{equation}
\partial_t \nabla^2_{\xi} \phi
\end{equation}
at the origin, thus finishes the proof that $\phi$ satisfies Bourgain's condition at the origin. 
\end{proof}

\section{Proof of Theorem \ref{thm: main}, Part \ref{thm: main part 3}: Katz-Wolff  implies Bourgain}\label{251003secc5}

Let $\Phi$ be a map satisfying H\"ormander's condition. Without loss of generality, we write 
\begin{equation}
\Phi(w, t; \xi)= (X(w, t; \xi), t).
\end{equation}
In this section we will show that if $\Phi$ satisfies the Katz-Wolff condition, then it also satisfies Bourgain's condition. \\

We without loss of generality assume that $X$ is of a normal form at the origin, that is, 
\begin{equation}
X(w, t; \xi)=w+ t\pnorm{
\xi+ O(|w||\xi|)
}+ t^2 O(|\xi|)+ t^3 O(|\xi|)+\cdots.
\end{equation}
In the definition of the Katz-Wolff condition in Definition \ref{250510defi1_1}, we take two curves $\Gamma_1$ and $\Gamma_2$, both passing through the origin, one with frequency $\xi^\prime=0\in \R^{n-1}$ and the other with frequency $\xi^{\prime\prime}=\eta$ for some $\eta\neq 0$, i.e.,
\begin{align}
    \Gamma_1=\{(0,t):|t|\le\epsilon_\Phi\},\quad \Gamma_2=\{(X(0,t;\eta),t):|t|\le\epsilon_\Phi\}.
\end{align}
Katz-Wolff condition says that there exist $\hat{w}=w(\eta;t_1,t_2)$ and $\hat{\xi}=\xi(\eta;t_1,t_2)$ such that the curve
\begin{align}
    \Gamma=\{(X(\hat{w},t;\hat{\xi}),t):|t|\le\epsilon_\Phi\}
\end{align}
intersects $\Gamma_1$ at height $t_1$ and $\Gamma_2$ at height $t_2$, respectively, and we know that the set 
\begin{equation}
    \{ \xi(\eta; t_1, t_2): t_1, t_2\neq 0\}
\end{equation}
forms a curve, for all choices of $\eta\neq 0$. In other words, we know that 
\begin{equation}\label{251209e5-6}
\rank[
\partial_{t_1} \xi(\eta; t_1, t_2), \partial_{t_2} \xi(\eta; t_1, t_2)
]< 2
\end{equation}
for every $\eta\neq 0$, and every $t_1, t_2\neq 0$. To simplify future notation, we will write \eqref{251209e5-6} as 
\begin{equation}\label{251209e5-7}
\det\big(
\partial_{t_1} \xi(\eta; t_1, t_2), \partial_{t_2} \xi(\eta; t_1, t_2)
\big)=0,
\end{equation}
where the above determinant can be understood as the column vector whose entries are all the determinants of the $2\times 2$ sub-matrices of the $(n-1)\times 2$ matrix $[
\partial_{t_1} \xi(\eta; t_1, t_2), \partial_{t_2} \xi(\eta; t_1, t_2)
]$. We begin with two equations,
\begin{align}
    X(\hat{w},t_1;\hat{\xi})=0,\quad\quad\quad\ \  \label{251209e5-8} \\
    X(\hat{w},t_2;\hat{\xi})= X(0,t_2;\eta), \label{251209e5-9}
\end{align}
and repeat an argument similar to that in Section \ref{25-Sec4.1}.\\

We calculate $\partial_{t_1} \xi(\eta; t_1, t_2)$ and $\partial_{t_2} \xi(\eta; t_1, t_2)$. If we take $\partial_{t_1}$ on \eqref{251209e5-8} and \eqref{251209e5-9}, respectively, then we obtain
\begin{align}
    \nabla_w X(\hat{w},t_1;\hat{\xi})\partial_{t_1}\hat{w}+ \partial_t X(\hat{w},t_1;\hat{\xi})+ \nabla_\xi X(\hat{w},t_1;\hat{\xi})\partial_{t_1}\hat{\xi}=0, \label{251209e5-10}\\
    \nabla_w X(\hat{w},t_2;\hat{\xi})\partial_{t_1}\hat{w}+ \nabla_\xi X(\hat{w},t_2;\hat{\xi})\partial_{t_1}\hat{\xi}=0.\quad\quad\qquad\qquad\ \ \label{251209e5-11}
\end{align}
Solving \eqref{251209e5-10} and \eqref{251209e5-11} gives
\begin{align}\label{251209e5-12}
    \bigg(\nabla_\xi X(\hat{w},t_1;\hat{\xi})&- \nabla_w X(\hat{w},t_1;\hat{\xi}) \nabla_w X(\hat{w},t_2;\hat{\xi})^{-1} \nabla_\xi X(\hat{w},t_2;\hat{\xi})\bigg)\partial_{t_1}\hat{\xi} \notag\\
    &=-\partial_t X(\hat{w},t_1;\hat{\xi}).
\end{align}
If we take $\partial_{t_2}$ on \eqref{251209e5-8} and \eqref{251209e5-9}, respectively, then we obtain
\begin{align}
    \nabla_w X(\hat{w},t_1;\hat{\xi})\partial_{t_2}\hat{w}+ \nabla_\xi X(\hat{w},t_1;\hat{\xi})\partial_{t_2}\hat{\xi}=0, \quad\quad\qquad\qquad\ \ \label{251209e5-13}\\
    \nabla_w X(\hat{w},t_2;\hat{\xi})\partial_{t_2}\hat{w}+ \partial_t X(\hat{w},t_2;\hat{\xi})+ \nabla_\xi X(\hat{w},t_2;\hat{\xi})\partial_{t_2}\hat{\xi}=0. \label{251209e5-14}
\end{align}
Solving \eqref{251209e5-13} and \eqref{251209e5-14} gives
\begin{align}
    \bigg(\nabla_\xi X(\hat{w},t_1;\hat{\xi})&- \nabla_w X(\hat{w},t_1;\hat{\xi}) \nabla_w X(\hat{w},t_2;\hat{\xi})^{-1} \nabla_\xi X(\hat{w},t_2;\hat{\xi})\bigg)\partial_{t_2}\hat{\xi} \notag\\
    &= \nabla_w X(\hat{w},t_1;\hat{\xi}) \nabla_w X(\hat{w},t_2;\hat{\xi})^{-1}\partial_t X(\hat{w},t_2;\hat{\xi}).
\end{align}
By the following claim, \eqref{251209e5-7} actually says
\begin{equation}\label{251209e5-16}
    \det\big( \partial_t X(\hat{w},t_1;\hat{\xi}), \nabla_w X(\hat{w},t_1;\hat{\xi}) \nabla_w X(\hat{w},t_2;\hat{\xi})^{-1}\partial_t X(\hat{w},t_2;\hat{\xi}) \big)=0,
\end{equation}

\begin{claim}\label{260309claim5_1}
    Under the above notation, if we fix $t_2\neq 0$, and pick $\eta$ and $t_1$ sufficiently close to 0, then we have that
    \begin{align}\label{251209e5-17}
        \nabla_\xi X(\hat{w},t_1;\hat{\xi})- \nabla_w X(\hat{w},t_1;\hat{\xi}) \nabla_w X(\hat{w},t_2;\hat{\xi})^{-1} \nabla_\xi X(\hat{w},t_2;\hat{\xi})
    \end{align}
    is non-singular.
\end{claim}

\begin{proof}[Proof of Claim \ref{260309claim5_1}]
    By the normal form of $X$, we see that
    \begin{align}
        \nabla_w X(w,t;\xi)= I+O(|t||\xi|)+O(|t|^2),\ \\
        \nabla_\xi X(w,t;\xi)= tI+ O(|t||w|)+O(|t|^2),
    \end{align}
    which further implies 
    \begin{align}\label{251209e5-20}
        \nabla_w X(w,t;\xi)^{-1}\nabla_\xi X(w,t;\xi)= tI+O(|t||w|)+O(|t|^2).
    \end{align}
    To see \eqref{251209e5-17} is non-singular, it suffices to show
    \begin{equation}\label{251209e5-21}
        \nabla_w X(\hat{w},t_1;\hat{\xi})^{-1}\nabla_\xi X(\hat{w},t_1;\hat{\xi})-  \nabla_w X(\hat{w},t_2;\hat{\xi})^{-1} \nabla_\xi X(\hat{w},t_2;\hat{\xi})
    \end{equation}
    is non-singular. Note that $\hat{w}\to 0$ and $\hat{\xi}\to \eta$, as $t_1\to 0$. If we set $t_1$ and $\eta$ sufficiently close to 0, then  $\hat{w}$ and $\hat{\xi}$ are sufficiently close to 0. It hence follows from \eqref{251209e5-20} that \eqref{251209e5-21} behaves like $(t_1-t_2)I$, which is non-singular since $t_1\neq t_2$.    
    
\end{proof}

Let's continue with \eqref{251209e5-16}, which implies that there exists a scalar function $\lambda(\eta;t_1,t_2)$ such that
\begin{equation}\label{251209e5-22}
    \nabla_w X(\hat{w},t_2;\hat{\xi})^{-1}\partial_t X(\hat{w},t_2;\hat{\xi})= \lambda(\eta;t_1,t_2) \nabla_w X(\hat{w},t_1;\hat{\xi})^{-1} \partial_t X(\hat{w},t_1;\hat{\xi}).
\end{equation}
First we take $t_1\to 0$ on both sides of \eqref{251209e5-22}. Then by the normal form of $X$, we have that
\begin{align}
    \nabla_w X(\hat{w},t_1;\hat{\xi})\to I,\quad \partial_t X(\hat{w},t_1;\hat{\xi})\to \eta,\quad\quad \label{251209e5-23} \\
    \nabla_w X(\hat{w},t_2;\hat{\xi})^{-1}\partial_t X(\hat{w},t_2;\hat{\xi})\to \eta+ O(|t_2||\eta|). \label{251209e5-24}
\end{align}
These also ensure that we can take the limit of $\lambda(\eta;t_1,t_2)$ as $t_1\to 0$.
So we have
\begin{align}\label{251209e5-25}
    \nabla_w X(0,t_2;\eta)^{-1}\partial_t X(0,t_2;\eta)= \lambda(\eta;0,t_2)\eta.
\end{align}
Next we write $\eta_0$ be a unit vector and let $0< r\le 1$ be a scalar. Taking $\eta=r\eta_0$, we see that \eqref{251209e5-25} becomes
\begin{align}\label{251209e5-26}
    \nabla_w X(0,t_2;r\eta_0)^{-1}\partial_t X(0,t_2;r\eta_0)= \lambda(r\eta_0;0,t_2)r\eta_0.
\end{align}
Before we take $\partial_r$ on both sides of \eqref{251209e5-26} and take $r\to 0$, we need to check the existence of $\lim_{r\to 0}\lambda$ and $\lim_{r\to 0}\partial_r \lambda$. This again follows from the normal form of X,
\begin{equation}
    \nabla_w X(0,t_2;r\eta_0)^{-1}\partial_t X(0,t_2;r\eta_0)= r(I+O(|t_2||\eta_0|))(\eta_0+O(|t_2||\eta_0|)),
\end{equation}
which says that $\lambda(r\eta_0;0,t_2)\eta_0$ is a smooth function of $r$. So, it is time to take $\partial_r$ on both sides of \eqref{251209e5-26} and take $r\to 0$. Then we obtain
\begin{align}
    \partial_t \nabla_\xi X(0,t_2;0)\eta_0= \lambda(0;0,t_2)\eta_0.
\end{align}
Set $t=t_2$. Since $\eta_0$ can be chosen arbitrarily, we have
\begin{align}
    \partial_t \nabla_\xi X(0,t;0)= \lambda(0;0,t)I.
\end{align}
Finally, we take $\partial_{t}$ on both sides and take $t\to 0$, then we see that
\begin{align}
    \partial^2_t \nabla_\xi X(0,0;0)= c\,I= c\,\partial_t \nabla_\xi X(0,0;0).
\end{align}
This proves that $X$ satisfies Bourgain's condition at 0.

\section{Proof of Theorem \ref{thm: main}, part \ref{thm: main part 4}: Totally geodesic and Bourgain imply B.R.-type}\label{251003secc6}

In this section we will show that if a $\Phi$-Kakeya problem satisfies both the totally geodesic condition and Bourgain's condition, then it is of the Bochner-Riesz-type. \\

Let us state a quantitative version of our result. Let $V: \R^{n-1}\to \R$ be an analytic function near the origin. If for some $n\le p< \infty$ it holds that 
\begin{equation}
\Norm{
\sup_{\omega\in \B^{n-1}_{\varepsilon}}
\int_{t\simeq 1} |
f(\eta+ t\omega, t- V(\eta))
| dt
}_{L^p(\B^{n-1}_{\varepsilon})}\lesim_{V, p, \epsilon_0}
\delta^{-\epsilon_0} \|f\|_p,
\end{equation}
for every $\epsilon_0$ and $\delta\in (0, 1)$, and every function $f$ that is locally constant at the scale $\delta$, then it also holds that 
\begin{equation}
\|\mathcal{K}^{(\Phi)}_{\delta}f\|_{L^p(\B^{n-1}_{\varepsilon})} \lesim \delta^{-\epsilon_0} \|f\|_p
\end{equation}
for the same $p$, where $\mathcal{K}^{(\Phi)}_{\delta}$ is the $\Phi$-Kakeya maximal operator, with $\Phi$ satisfying  the Katz-Wolff condition. \\

By Theorem \ref{thm: main} Part \ref{thm: main part 1}, if $\Phi$ satisfies the totally geodesic condition, then we can write  
\begin{equation}
\Phi(w, t; \xi)= (X(w, t; \xi), t),
\end{equation}
with 
\begin{equation}
X(w, t; \xi)= (
w_1+ t V_1(w; \xi), \dots,  w_{n-1}+ t V_{n-1}(w; \xi)
).
\end{equation}
H\"ormander's condition means that 
\begin{equation}
\left.\nabla_w X(w, t ; \xi)\right|_{(w, t)=0 ; \xi=0}
\end{equation}
has rank $(n-1)$, 
and 
\begin{equation}
\left[\begin{array}{cc}
\nabla_w X(w, t ; \xi) & \nabla_{\xi} X(w, t ; \xi) \\
\partial_t \nabla_w X(w, t ; \xi) & \partial_t \nabla_{\xi} X(w, t ; \xi)
\end{array}\right]
\end{equation}
at $w=\xi=0, t=0$ has rank $2(n-1)$. Bourgain's condition implies that there exists a scalar function $c(w, t; \xi)$ such that 
\begin{equation}
\partial_t \nabla_w X(w, t; \xi)= c(w, t; \xi) \nabla_{w}X(w, t; \xi).
\end{equation}
Note that this means 
\begin{equation}\label{250521e3_5}
[\partial_{w_i} V_j(w; \xi)]_{1\le i, j\le n-1}= c(w, t; \xi) (I_{n-1}+ t [\partial_{w_i} V_j(w; \xi)]_{1\le i, j\le n-1}).
\end{equation}
\begin{claim}\label{250521claim3_1}
We have that 
\begin{equation}
\partial_{w_i} V_j(w; \xi)=0
\end{equation}
for all $w, \xi$, and all $i\neq j$. 
\end{claim}
\begin{proof}[Proof of Claim \ref{250521claim3_1}] 
Let $i\neq j$. It follows from \eqref{250521e3_5} that for all $w,t$ and $\xi$,
\begin{equation}\label{251125e6-10}
    \partial_{w_i} V_j(w; \xi)= c(w, t; \xi)t\ \partial_{w_i} V_j(w; \xi).
\end{equation}
If we set $t=0$ in \eqref{251125e6-10}, then we will see that Claim \ref{250521claim3_1} holds.
\end{proof}

At this point, \eqref{250521e3_5} can be written as 
\begin{align}\label{250521e3_8}
&
\begin{bmatrix}
\partial_{w_1} V_1(w; \xi) &  & \\
& \dots & \\
&  & \partial_{w_{n-1}} V_{n-1}(w; \xi)
\end{bmatrix}
\\
&
=c(w, t; \xi) 
\begin{bmatrix}
1+
t\partial_{w_1} V_1(w; \xi) & & \\
& \dots & \\
& & 1+t\partial_{w_{n-1}} V_{n-1}(w; \xi)
\end{bmatrix},
\end{align}
where all the non-diagonal entries are zero. 
This implies 
\begin{equation}\label{251003ee6_10kk}
\partial_{w_1} V_1(w; \xi)=\dots= \partial_{w_{n-1}} V_{n-1}(w; \xi),
\end{equation}
for all $w$ and $\xi$. 
Claim \ref{250521claim3_1} and \eqref{251003ee6_10kk} together 
 further imply that 
\begin{equation}
V_i(w; \xi)= V_{i}(\xi)+ w_i V(\xi)
\end{equation}
for every $i=1, \dots, n-1$ and for some scalar functions $V_i(\xi), V(\xi)$. 
So far we have obtained that  
\begin{equation}
X(w, t; \xi)= 
w+ t (V_1(\xi), \dots, V_{n-1}(\xi))+ 
tw V(\xi). 
\end{equation}
We apply a nonlinear change of variables in $\xi$ and obtain 
\begin{equation}
X(w, t; \xi)= 
w+ t\xi+ twV(\xi),
\end{equation}
for some potentially different $V(\xi)$. \\

Recall the definition of Kakeya maximal functions
	\begin{align}
		\mathcal{K}_\delta f(\xi)= \sup_\omega \frac{1}{\delta^{n-1}}\int_{T_{\xi}^{\delta}(\omega)} |f(x,t)| \,\text{d}x\,\text{d}t.
	\end{align}
	Here the supremum is taken over $\omega\in \mathbb{B}^{(n-1)}_\varepsilon$. We without loss of generality assume that  $f$ is locally constant at the scale $\delta$. Under this assumption, we have that 
	\begin{align}
		\mathcal{K}_\delta f(\xi)\simeq \sup_\omega \int_0^\varepsilon |f(\omega+t\xi+t\omega V(\xi),t)| \,\text{d}t.
	\end{align}
	Split the integral into
	\begin{align}\label{251003e6_15kk}
		\int_0^\delta+ \sum_{\delta\le 2^{-j}\varepsilon\le \varepsilon}\int_{t\simeq 2^{-j}\varepsilon}
	\end{align}
	The first part can be easily controlled:  
	\begin{align}\label{eee1}
		\left\|   \sup_\omega \int_0^\delta |f(\omega+t\xi+t\omega V(\xi),t)| \,\text{d}t   \right\|_{L^p(\mathbb{B}_\varepsilon^{n-1})}  \lesssim \|f\|_p,
	\end{align}
	whenever $p\ge n$. 
	This is because via locally constant property we have
	\begin{align*}
		& \sup_\omega \int_0^\delta |f(\omega+t\xi+t\omega V(\xi),t)| \,\text{d}t
		\simeq \delta \sup_\omega |f(\omega, 0)|  \\
		&\lesssim \delta^{1-\frac{n}{p}} \sup_\omega  \left( \int_{\mathbb{B}^n_\delta(\omega,0)} |f(x,t)|^p \,\text{d}x\,\text{d}t\right)^{1/p} \le \delta^{1-\frac{n}{p}} \|f\|_p.
	\end{align*}
This finishes the proof of \eqref{eee1}. \\

	Next we deal with the remaining terms in \eqref{251003e6_15kk}.  By rescaling, we obtain 
	\begin{align}
		& \sup_\omega \int_{t\simeq 2^{-j}\varepsilon} |f(\omega+ t\xi+ t\omega V(\xi),t)| \,\text{d}t \\
		&\leq 2^{-j}\sup_\omega \int_{t\simeq 1} |f_j(2^j\omega+ t\xi+ t\omega V(\xi), t)| \,\text{d}t  \\
		&= 2^{-j}\sup_{|\omega|\le 2^j\varepsilon} \int_{t\simeq 1} |f_j(\omega+ t\xi+ t\omega 2^{-j} V(\xi), t)| \,\text{d}t, \label{251003e6_19kk}
	\end{align}
	where
\begin{equation}
f_j(\cdot):= f(2^{-j}\cdot).
\end{equation}
Note that the function $f_j$ is locally constant at the scale $2^j \delta$. \\

	 Let $\{\mathbb{B}^{n-1}_\varepsilon(\omega_k)\}_k$ be a family of finitely overlapping balls in $\mathbb{R}^{n-1}$, with the union of these balls covering $\mathbb{B}^{n-1}_{2^j\varepsilon}$; we cover this ball because in \eqref{251003e6_19kk} the supremum is taken over all $|\omega|\le 2^j\varepsilon$.  Let 
    \begin{align}
        f_{j,k}(x,t)= f_j(x,t) \chi_{\mathbb{B}^{n-1}_{C\varepsilon}(\omega_k)}(x),
    \end{align}
    where $C$ is a sufficiently large constant and $\chi_{\mathbb{B}^{n-1}_{C\varepsilon}(\omega_k)}$ is an $L^{\infty}$-normalized bump function adapted to the ball $\mathbb{B}^{n-1}_{C\varepsilon}(\omega_k)$. Under the above notation, we obtain that 
    \begin{align}
        \sup_{|\omega|\le 2^j\varepsilon} \int_{t\simeq 1} |f_j(\omega+ t\xi+ t\omega 2^{-j} V(\xi), t)| \,\text{d}t 
    \end{align}
    can be dominated by
	\begin{align}\label{251003e6_23kk}
		 \left( \sum_k \left( \sup_{\omega\in \mathbb{B}^{n-1}_\varepsilon(\omega_k)} \int_{t\simeq 1} |f_{j,k}(\omega+ t\xi+ t\omega 2^{-j} V(\xi), t)| \,\text{d}t \right)^p  \right)^{\frac1p},
	\end{align}
	where we use the fact that if $\omega\in \mathbb{B}^{n-1}_\varepsilon(\omega_k)$, then
    \begin{align}
        \omega+ t\xi+ t\omega 2^{-j} V(\xi)\in \mathbb{B}^{n-1}_{C\varepsilon}(\omega_k),
    \end{align}
provided that $C$ is chosen to be sufficiently large. 
	For each $k$, we apply the change of variables
    \begin{align}
        \omega\mapsto \omega+ \omega_k,
    \end{align}
 and obtain 
 \begin{equation}\label{251003e6_26kk}
 \sup_{\omega\in \mathbb{B}^{n-1}_\varepsilon} \int_{t\simeq 1} |f_{j,k}(\omega+\omega_k+ t\xi+ t(\omega+ \omega_k) 2^{-j} V(\xi), t)| \,\text{d}t
 \end{equation}
Next we apply the change of variable   
   \begin{equation}
   \xi+ 2^{-j}\omega_k V(\xi)\mapsto \eta
   \end{equation}
and write \eqref{251003e6_26kk} as 
     \begin{equation}\label{251003e6_28kk}
 \sup_{\omega\in \mathbb{B}^{n-1}_\varepsilon} \int_{t\simeq 1} |\widetilde{f}_{j,k}(\omega+ t\eta
 + t\omega \widetilde{V}(\eta), t)| \,\text{d}t,
 \end{equation}
 for some new function $\widetilde{V}$, and 
 \begin{equation}
 \widetilde{f}_{j, k}(\cdot, \cdot):= f_{j, k}(\cdot+ \omega_k, \cdot).
 \end{equation}
We apply the change of variable 
\begin{equation}
t\mapsto 1/t,
\end{equation}
denote 
\begin{equation}
g_{j, k}(x, t)=\widetilde{f}_{j, k}(x/t, 1/t),
\end{equation}
and write \eqref{251003e6_28kk} as 
\begin{align}
& 
\sup_{\omega\in \B^{n-1}_{\varepsilon}}\int_{t\simeq 1} |
g_{j, k}(\eta+ \omega\widetilde{V}(\eta)+ t\omega, t)
|d t\\
& \lesim \sup_{\omega\in \B^{n-1}_{\varepsilon}}\int_{t\simeq 1} |
g_{j, k}(\eta+ t\omega, t-\widetilde{V}(\eta))
|d t.
\end{align}
Note that the function $g_{j,k}$ is locally constant at the scale $2^j\delta$. If we know already that for some $n\le p<\infty$ we have
	\begin{align}\label{eee3}
		\left\| \sup_{\omega\in \mathbb{B}^{n-1}_\varepsilon} \int_{t\simeq 1} |g_{j,k}(\eta+ t\omega, t -\widetilde{V}(\eta))| \,\text{d}t\right\|_{L^p(\mathbb{B}^{n-1}_\varepsilon)}\lesssim (2^j\delta)^{-\epsilon_0}\|g_{j,k}\|_p,
	\end{align}
for every $\epsilon_0>0$, 
	then we can control that
	\begin{align}\label{eee4}
		& \left\| \sup_\omega \int_{t\simeq 2^{-j}\varepsilon} |f(\omega+ t\xi+ t\omega V(\xi),t)| \,\text{d}t\right\|_{L^p(\mathbb{B}^{n-1}_\varepsilon)} \\
		&\lesssim (2^j\delta)^{-\epsilon_0}2^{-j}\left( \sum_k  \|f_{j,k}\|_p^p  \right)^{\frac1p} \notag\lesssim (2^j\delta)^{-\epsilon_0}2^{-j}\|f_j\|_p \notag\lesssim (2^j\delta)^{-\epsilon_0}2^{-j(1-\frac{n}{p})}\|f\|_p.
	\end{align}
	Combining \eqref{eee1} and \eqref{eee4}, we obtain that
	\begin{align}
		\|\mathcal{K}_\delta f\|_{L^p(\mathbb{B}^{n-1}_\varepsilon)}\lesssim \delta^{-\epsilon_0}\|f\|_p,
	\end{align}
	holds for $p\ge n$, provided that \eqref{eee3} also holds for this $p$.

\section{Proof of Theorem \ref{thm: main} part \ref{thm: main part 5}: Totally geodesic and failing Bourgain's condition means extra-worst compression}\label{251003sec7}

Consider the $\Phi$-Kakeya problem, and we assume that $\Phi$ satisfies the totally geodesic condition but fails Bourgain's condition. 
By Theorem \ref{thm: main} Part \ref{thm: main part 1}, we can write 
\begin{equation}
\Phi(w, t; \xi)=(X(w, t; \xi), t),
\end{equation}
with 
\begin{align}
    X(w,t;\xi) = w + tV(w,\xi). 
\end{align}
By Theorem \ref{250727theorem1_13}, the failure of Bourgain's condition implies $\partial_t \nabla_w X(w,t;\xi)$ is not a multiple of $\nabla_w X(w,t;\xi)$. This means that $\partial_w V(w,\xi)$ is not a multiple of the identity matrix $I$.  

We will show that there exists an open set of hyper-planes passing through the origin, each of which gives rise to a Kakeya set. By replacing $X(w,t; \xi)$ with a small translation of $X(w, t; \xi)$, we get the same conclusion for hyper-planes through nearby points. 
Parameterize the normal direction of a hyper-plane by $(\vec a, c)$. We seek $w = w(\xi)$ satisfying 
\begin{align}\label{eq: pf of T.G. + fail Bourg -> worst}
    (w + tV(w,\xi), t) \cdot  (\vec a, c)= 0.
\end{align}
On extracting the coefficients in $t$,
\begin{align}
    w \cdot \vec{a} = 0 \\
     V(w,\xi) \cdot \vec{a} + c=0
\end{align}
Define the auxiliary function 
\begin{align}
    F(\vec a,c, w,\xi) = (V(w,\xi) \cdot \vec a + c, w \cdot \vec a).
\end{align}
We compute 
\begin{align}
    \nabla_w F(\vec a, 0,0,0) = (\nabla_w V(0)\cdot \vec a,\vec a).
\end{align}
We claim that for some $\vec a_0$, it holds that\footnote{Here the assumption  $n=3$ occurs. } 
\begin{equation}
\det \nabla_w F(\vec a_0, 0,0,0) \neq 0.
\end{equation}
 If this failed for every $\vec a$ then $\partial_w V(0)$ is a multiple of the identity, which we already showed is not the case. Thus we can apply the implicit function theorem to find $w_{\vec a,c}(\xi)$ which solves
\begin{align}
    F(\vec a,c, w_{\vec a, c}(\xi), \xi) = F(\vec a_0, 0,0,0),
\end{align}
for $(\vec a, c, \xi)$ in a small open set around $(\vec a_0, 0,0)$.
This means that
\begin{equation}
w_{\vec a,c}(\xi) \cdot \vec a = 0,
\end{equation}
 and
 \begin{equation}
 V(w_{\vec a,c}(\xi),\xi) \cdot \vec{a} + c = V(0) \cdot \vec a_0.
 \end{equation}
  Replacing $c$ by $c + V(0)\cdot \vec a_0$, we see that \eqref{eq: pf of T.G. + fail Bourg -> worst} is satisfied for $\xi$ in an open set, and for an open set of normal directions $n = (\vec a, c)$. This finishes the proof of the extra-worst compression phenomenon.

\section{Proof of Theorem \ref{thm: main} Part \ref{thm: main part 2}: Katz-Wolff implies totally geodesic.}\label{251003sec8}

We use $\Gamma_{\xi}(\bfx)$ to denote a curve in the $\Phi$-Kakeya problem passing through $\bfx\in \R^n$, with direction $\xi$. Take two intersecting curves 
\begin{equation}
\ell_0:= \Gamma_{\xi_0}(\bar{\bfx}), \ \ \ell_1:= \Gamma_{\xi_1}(\bar{\bfx}).
\end{equation}
We generate a surface $S$ by sweeping out a curve in $\xi_1$ along $\ell_0$:
\begin{equation}
S:= \bigcup_{\bfx\in \ell_0} \Gamma_{\xi_1}(\bfx).
\end{equation}
Clearly $S$ contains $\ell_0$ and $\ell_1$. That $S$ is totally geodesic (see Definition \ref{250801definition1_16}) follows immediately from the following lemma.

\begin{claim}\label{claim: in KW implies flat}
    If a curve $\ell_2$ from the $\Phi$-Kakeya problem intersects $\ell_0$ and another point in $S$, then $\ell_2 \subset S$. 
\end{claim}

\begin{figure}
    \centering
    \includegraphics[width=0.75\linewidth]{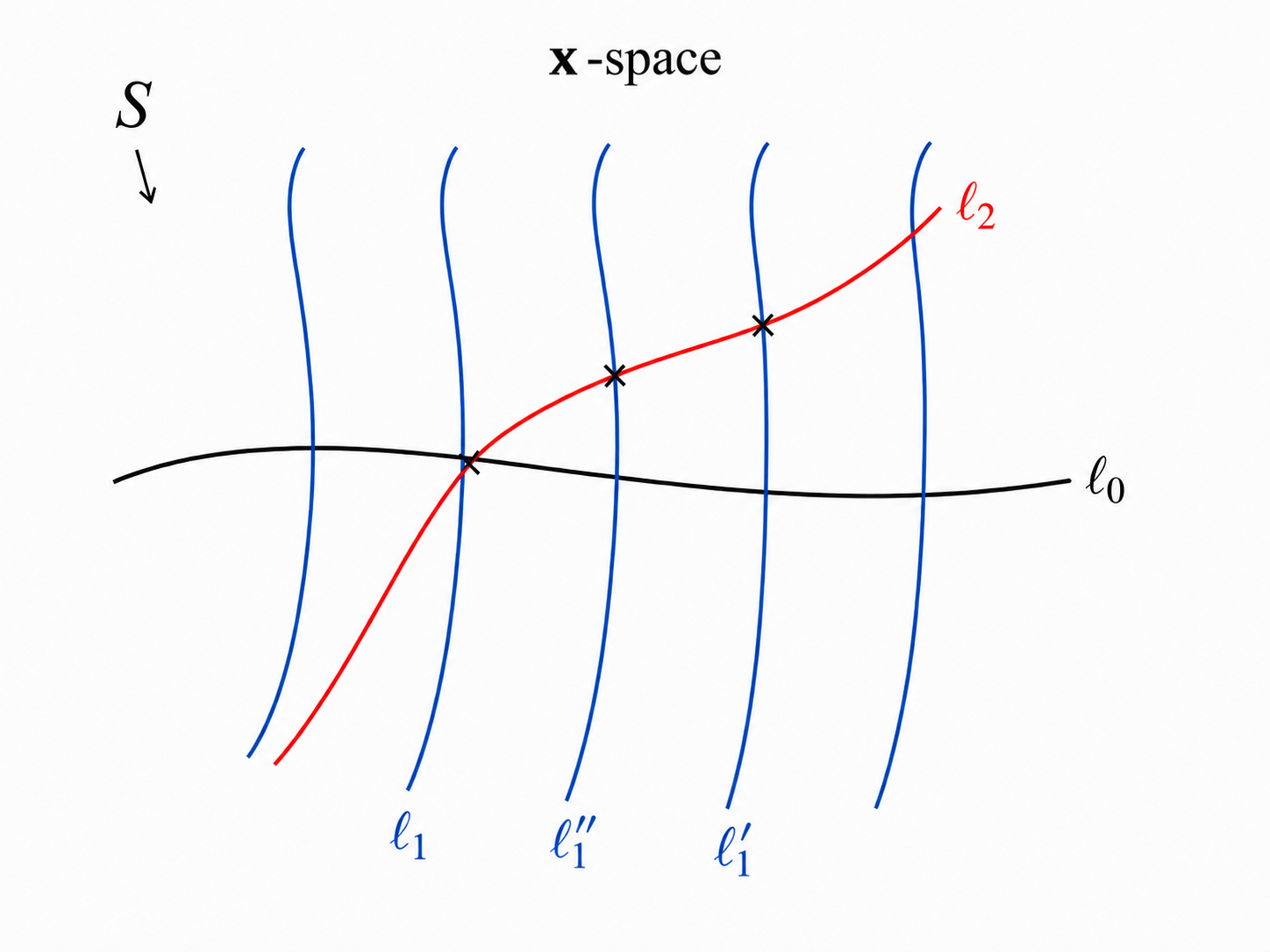}
    \caption{Picture of K.W. implies T.G.}
    \label{fig: KW implies TG}
\end{figure}

\begin{proof}[Proof of Claim \ref{claim: in KW implies flat}]
    If $\ell_2$ hits another point, then it hits one of the ``translates'' $\ell_1'$ of $\ell_1$. We will use the Katz-Wolff assumption to show every point of $\ell_2$ is hit by a translate $\ell_1''$ of $\ell_1$.
    
    We apply the Katz-Wolff condition to the two curves $\ell_0$ and $\ell_2$: If we use $\Xi_{0, 2}$ to denote the collection of all the direction parameters of curves intersecting both $\ell_0$ and $\ell_2$, that is, 
    \begin{equation}
    \Xi_{0, 2}:=\{\xi: \Gamma_{\xi}(\bfx) \text{ intersects both } \ell_0  \text{ and } \ell_2 \text{ for some } \bfx\},
    \end{equation}
 then the Katz-Wolff condition says that $\Xi_{0, 2}$ is one-dimensional. Moreover, it is elementary to see that if we fix $\bfx_0\in \ell_0, \ \bfx_0\notin\ell_2$, then the collection of the direction parameters 
 \begin{equation}
 \Xi_{0, 2}(\bfx_0):= \{\xi: \Gamma_{\xi}(\bfx_0) \text{ intersects } \ell_2\}
 \end{equation}
 coincides with $\Xi_{0, 2}$. \\
 
 Recall that $\ell'_1$ intersects both $\ell_0$ and $\ell_2$, and therefore 
 \begin{equation}
 \xi_1\in \Xi_{0, 2}.
 \end{equation}
 This further implies that 
 \begin{equation}
  \xi_1\in \Xi_{0, 2}(\bfx_0),
 \end{equation}
 whenever $\bfx_0\in \ell_0, \ \bfx_0\notin \ell_2$. In other words, we know that $\Gamma_{\xi_1}(\bfx_0)$ intersects $\ell_2$. This finishes the proof that $\ell_2\subset S$. 
\end{proof}

\section{Proof of Theorem \ref{thm: main} :  Part 1: Totally geodesic implies projectively flat}\label{251003sec9}

We will prove that a spray space $(M, \mathcal U, \bG)$ satisfying Definition \ref{def: T.G. geometric} is locally projectively flat. By Theorem \ref{thm: douglas weyl vanishing}, matters reduce to proving that $\mathbf D = 0$ and $\mathbf W = 0$. We prove each of these separately. 

\subsection{Showing that \texorpdfstring{$\mathbf D=0$}{}}

Fix $(x_0,y_0) \in \mathcal U$. For convenience, pass to a smaller open neighborhood $M' \times C \subset \mathcal U$ of $(x_0,y_0)$, where $M' \subset M$ is open and $C$ is an open cone. The next lemma lets us straighten the geodesics through a point. This is similar to the geodesic coordinates in Riemannian geometry, except for the technicality that $C$ is not necessarily $\R^3$.

\begin{lemma}[Straightening Lemma]\label{lem: straightening}
    Consider a spray $G : M \times C \to \R^n$ and fix $(x_0,y_0) \in M \times C$. There exist open neighborhoods $U \ni x_0$ and $V \ni y_0$, with $0 \notin V$, a point $p_0 \notin U$, and a diffeomorphism $F : V \to U$ such that 
\begin{enumerate}
        \item $F(y_0)=x_0$ and $dF_{y_0}(y_0) \in \R_{>0}y_0$,
        \item $F$ maps radial line segments in $V$ to $G$-geodesic segments. More precisely, for each $y$ in a sufficiently small open cone containing $y_0$, the curve 
        \begin{align}
            t \mapsto F(ty)
        \end{align}
        is a geodesic whenever $ty \in V$, and it extends to a geodesic passing through $p_0$ at $t=0$. 
    \end{enumerate}
    Thus $F^{-1}$ straightens the geodesics through $p_0$ into radial line segments. 
\end{lemma}
\begin{proof}[Proof of Lemma \ref{lem: straightening}]
    By basic ODE theory, the initial value problem 
    \begin{align}
        \ddot \gamma(t) + 2 G(\gamma(t),\dot \gamma(t)) = 0, \ \gamma(0) = x, \dot \gamma(0)=y
    \end{align}
    has a unique solution and depends smoothly on $(t,x,y)$ for $(x,y)$ near $(x_0,y_0)$ and $t$ sufficiently small. Choose $0 < \tau \ll 1$ and define 
    \begin{align}
        p_0:=\gamma(-\tau; x_0,y_0),\eta_0:=\dot \gamma(-\tau;x_0,y_0).
    \end{align}
    By uniqueness, $\gamma(s;p_0,\eta_0) = \gamma(s-\tau;x_0,y_0)$, and thus $\gamma(\tau;p_0,\eta_0)=x_0$, $\dot \gamma(\tau;p_0,\eta_0)=y_0$. Since $\eta_0=y_0+O(\tau)$, we may choose an invertible linear map $A_\tau$ such that $A_\tau y_0 = \eta_0$, $A_\tau = I + O(\tau)$. Define 
    \begin{align}
        F(y):=\gamma(\tau;p_0,A_\tau y)
     \end{align}
     for $y$ in a neighborhood of $y_0$. We then have $F(y_0) =\gamma(\tau;p_0,\eta_0)=x_0$.
     Taylor expansion of the geodesic equation in $y$ near $y_0$ gives 
     \begin{align}
         F(y) = p_0 + \tau A_\tau y +O(\tau^2).
     \end{align}
     Hence 
     \begin{align}
         dF_{y_0} = \tau A_\tau + O(\tau^2) = \tau I + O(\tau^2).
     \end{align}
     Thus $dF_{y_0}$ is invertible for $\tau$ sufficiently small. The inverse function theorem then gives neighborhoods $V \ni y_0$ and $U \ni x_0$ such that $F : V \to U$ is a diffeomorphism. By the 2-homogeneity of the spray, 
     \begin{align}
         \gamma(s;x,\lambda v) = \gamma(\lambda s;x,v), \lambda > 0.
     \end{align}
     Thus whenever $ty\in V$,
     \begin{align}
         F(ty) = \gamma(t\tau; p_0, A_\tau y).
     \end{align}
     Hence $t \mapsto F(ty)$ is a parameterized geodesic, which passes through $p_0$ at $t=0$. Finally 
     \begin{align}
         dF_{y_0}(y_0) &= \frac{d}{dt}\mid_{t=1} \gamma(t\tau; p_0, \eta_0) \\
         &=\tau \dot \gamma(\tau; p_0,\eta_0) = \tau y_0 \in \R_{>0} y_0.
     \end{align}
     Moreover, $p_0 = x_0 - \tau y_0 + O(\tau^2)$, so $p_0 \neq x_0$ for small $\tau$. We can shrink $U$ and $V$ so that $p_0 \notin U$ and $0 \notin V$. This proves the lemma.
    \end{proof}

    We pass to the coordinates provided by Lemma \ref{lem: straightening}. After shrinking the coordinate neighborhoods, $V = B(y_0,\epsilon)$, $0 \notin B(y_0,\epsilon)$. Choose an open cone $C_0$ containing $y_0$ such that $B(y_0,\epsilon) \times C_0$ lies in the domain of the transformed spray 
    \begin{align}
        \mathbf{G}' :=(F^{-1})_*\mathbf{G}.
    \end{align}
    We have $d(F^{-1})_{x_0}(y_0) \in \R_{>0} y_0 \subset C_0$ since $C_0$ is conic. By  Lemma \ref{lem: straightening}, every radial segment $t \mapsto ty$ contained in $B(y_0,\epsilon)$, with $y \in C_0$, is a geodesic of $\mathbf{G}'$. Thus 
    \begin{equation}
    G'(ty,y) = 0.
    \end{equation}
     Since the geodesic surface with a given tangent plane is unique, all the planes passing through $0$ and intersecting $B(y_0, \epsilon)$ are geodesic surfaces. Any geodesic $\gamma$ is tangent to the plane through $0$ with normal $\gamma(0) \times \dot \gamma(0)$, and is thus contained in that plane. So 
     \begin{equation}
     \det (\gamma, \dot \gamma, \ddot \gamma) = 0.
     \end{equation}
      Taking $x = \gamma(0) \in B(y_0,\epsilon)$ and $y = \dot \gamma(0) \in C_0$, we get 
\begin{equation}
\det(x,y,G'(x,y)) = 0.
\end{equation}
    When $x$ and $y$ are linearly independent, there are scalars $a(x,y)$ and $b(x,y)$ such that 
    \begin{align}
        G'(x,y) = a(x,y)x + b(x,y)y. 
    \end{align}
    It is straightforward to see that $a$ and $b$ are uniquely determined smooth functions of $(x,y)$ when $x,y$ are independent, $a$ is 2-homogeneous in $y$, and $b$ is $1$-homogeneous in $y$. Below we study the spray 
    \begin{equation}
    G_1(x,y) = a(x,y)x
    \end{equation}
     defined on
\begin{equation}
\{(x,y) : x \in M, y \in C_0 \setminus \mathrm{Span}(x)\}.
\end{equation}
    Since $G_1$ is projectively equivalent to $G'$ where they are both defined, they have the same (unparameterized) geodesics. Thus $G_1$ also satisfies the totally geodesic condition (Definition \ref{def: T.G. geometric}). We will now show that $a(x,y)$ is quadratic in $y$.

    Fix $x \in M$ and a plane $n^\perp$ intersecting $C_0$ with $x \notin n^\perp$. Then there is a totally geodesic surface $S$ with $T_x S = n^\perp$. We can write 
    \begin{equation}
    S = \{x : f(x) = 0\}
    \end{equation}
     with $|\nabla f| = 1$ (the signed distance function to $S$ works, for example). As a consequence, 
     \begin{equation}
     \nabla^2 f \cdot\nabla f = 0.
     \end{equation}
      Any geodesic $\gamma$ with $\gamma(0) = x$ and $\dot \gamma(0) = y \in n^\perp \cap C_0$ is completely contained in $S$: $f(\gamma) = 0$. Differentiating this equation twice,
    \begin{align}
        \nabla f \cdot \dot \gamma &= 0 \\
        \nabla f \cdot \ddot \gamma + \nabla^2 f[\dot \gamma, \dot \gamma] &= 0.
    \end{align}
    Thus 
    \begin{align} \label{eq: a and f expr}
        a(x,y) = \frac{ \nabla^2 f[y,y]}{2\langle x,n\rangle}.
    \end{align}
    Fix $v \in n^\perp$ and consider the function expression $g(t) = a(x, y+tv)$. Then \eqref{eq: a and f expr} applies to give 
    \begin{align}
        g(t) &= \frac{ \nabla^2 f[y+tv,y+tv]}{2\langle x,n\rangle}.
    \end{align}
    This expression is quadratic in $t$, so differentiating three times,
    \begin{align}
        \nabla_y^3a(x,y)[v,v,v] = g'''(0) = 0.
    \end{align}
    We have proven that for any $n \notin x^\perp$ and $y,v \in n^\perp$, we have
    \begin{equation}
    \nabla_y^3 a(x,y)[v,v,v] = 0.
    \end{equation}
     This easily shows that 
     \begin{equation}
     \nabla_y^3 a(x,y) = 0
     \end{equation}
      for $y \in C_0 \setminus \mathrm{Span}(x)$. Indeed, fix such a $y$ and for each $v$ independent of $x$ and $y$, choose $n = y \times v$. Then $n \notin x^\perp$, so we find 
      \begin{equation}
      \nabla_y^3 a(x,y)[v,v,v]=0.
      \end{equation}
       This is now a polynomial in $v$ which vanishes on an open set, so each coefficient must vanish. Therefore 
       \begin{equation}
       \nabla_y^3 a(x,y) = 0.
       \end{equation}
        Thus the Douglas curvature $\mathbf D'$ of $\bG'$ vanishes on 
        \begin{equation}
        \{(x,y) : x \in B(y_0,\epsilon), y\in C_0 \setminus \mathrm{Span(x)}\}.
        \end{equation}
        This set is dense in $B(y_0,\epsilon) \times C_0$, so continuity gives 
        \begin{align}
            {\mathbf D}' \equiv 0 \text{ on } B(y_0, \epsilon) \times C_0.
        \end{align}
        In particular, ${\mathbf D}'_{(y_0, d(F^{-1})_{x_0}(y_0))}=0$, and $(y_0, d(F^{-1})_{x_0}(y_0))$ is the image of $(x_0,y_0)$ under the tangent lift of $F^{-1}$. Since the Douglas tensor is invariant under coordinate changes, it follows that $\mathbf{D}_{(x_0,y_0)} = 0$.

    \subsection{Showing that \texorpdfstring{$\mathbf W = 0$}{}}

    Since $\mathbf D = 0$, $\mathbf G$ is pointwise projectively related to an affine spray $\tilde \bG$. Let $\nabla$ be the affine connection associated to $\tilde \bG$. Since $\mathbf W$ is a projective invariant, it is enough to prove that 
    \begin{equation}
    W^\nabla \equiv 0.
    \end{equation}
     To do this, we will take advantage of more of the totally geodesic surfaces. We need two basic results to do this. 
    The first is about the behavior of the Weyl curvature tensor on a totally geodesic surface.
    \begin{lemma}[Totally geodesic surfaces and Weyl curvature]\label{eq: TG surface and weyl curvature}
        Let $S \subset M$ be a totally geodesic surface (with admissible directions in $\mathcal U$) for a connection $\nabla$ on $M$. Let $\Pi = T_xS$. Then 
        \begin{equation}
        W_x^\nabla(X,Y)Y \in \Pi
        \end{equation}
         for $X,Y \in \Pi$.
    \end{lemma}

    \begin{proof}[Proof of Lemma \ref{eq: TG surface and weyl curvature}]
        Choose coordinates $(x^1,x^2,x^3)$ on $M$ so that locally $S = \{x^3 =0\}$. Let 
        \begin{equation}
        v = (v^1,v^2,0) \in T_u S \cap \mathcal U_u.
        \end{equation}
         Let $\gamma$ be the geodesic with 
         \begin{equation}
         (\gamma(0),\dot \gamma(0)) = (x, v).
         \end{equation}
          Since $S$ is totally geodesic with admissible directions in $\mathcal U$, $\gamma(t) \in S$ for small $t$. In particular, $\gamma^3(t) \equiv 0$. Plugging this into the geodesic equation at $t = 0$, 
        \begin{align}
            \Gamma_{jk}^3(x) v^j v^k = 0.
        \end{align}
        This holds for all $x = (x^1,x^2,0)\in S$ and on an open set $(v^1,v^2)$, so 
        \begin{equation}
        \Gamma_{jk}^3(x^1,x^2,0) \equiv 0
        \end{equation}
         for $j,k \in \{1,2\}$. By a quick calculation in local coordinates, this implies that for all $X,Y \in \Gamma(TS)$, $\nabla_X Y \in \Gamma(TS)$.
        Since this holds for all $X,Y \in \Gamma(TS)$, we have by the formula \eqref{eq: affine riemann tensor formula} that 
        \begin{equation}
        R^\nabla(X,Y)Y \in \Gamma(T S)
        \end{equation}
         for $X,Y \in \Gamma(TS)$. By the formula \eqref{eq: affine weyl tensor formula}, we get 
         \begin{equation}
         W^\nabla(X,Y)Y \in \Gamma(T S).
         \end{equation}
          Finally viewing $X,Y \in T_xS$, we have 
          \begin{equation}
          W_x^\nabla(X,Y)Y \in T_x S.
          \end{equation}
          This finishes the proof of the lemma. 
    \end{proof}

    We also need the following linear algebra lemma.
    \begin{lemma}[Linear algebra lemma]\label{lem: lin alg lemma}
    Let $T$ be a $(1,3)$-tensor on 3-dimensional vector space $V$ satisfying the following properties: 
    \begin{enumerate}
        \item (Skew symmetry) $T(X,Y,Z) = -T(Y,X,Z)$ 
        \item (Bianchi identity) $T(X,Y,Z) + T(Y,Z,X)+T(Z,X,Y) = 0$ 
        \item (Trace-free) $\mathrm{trace}(X \mapsto T(X,Y,Z)) = 0$.
        \item (Plane property) $T(X,Y,Y) \in \mathrm{Span}(X,Y)$ for all $X,Y \in V$.
    \end{enumerate}
    Then $T \equiv 0$. 
    \end{lemma}

    \begin{proof}[Proof of Lemma \ref{lem: lin alg lemma}]

    For a basis $X,Y,Z \in V$, write $X^*, Y^*, Z^* \in V^*$ for the dual basis. \\
    
    \noindent \textbf{Step 1: $T(X,Y,Y) \| Y$ and $T(Y,X,X) \| X$}\\
    
    The second claim follows from the first by swapping $X$ and $Y$. We focus on the first claim. 
    By the plane property, write 
    \begin{equation}
    T(X,Y,Y) = aX + bY.
    \end{equation}
    The goal is to show $a = 0$. By the trace-free condition, 
    \begin{align} \label{eq: 1 alg lemma}
        0 &= \mathrm{trace}(U \mapsto T(U,Y,Y)) \\
        &=X^*T(X,Y,Y) + Y^*T(Y,Y,Y) + Z^*T(Z,Y,Y) \\
        &= a + 0 + Z^*T(Z,Y,Y),
    \end{align}
    where we used skew-symmetry for the second summand. Apply the plane property to the independent vectors $X + \epsilon Z, Y$ to get 
    \begin{align}\label{eq: 2 alg lemma}
        T(X+ \epsilon Z,Y,Y) = \lambda(\epsilon) (X + \epsilon Z) + \mu(\epsilon) Y. 
    \end{align}
    Taking $\epsilon = 0$ and applying $X^*$ in \eqref{eq: 2 alg lemma} gives 
    \begin{align}
        a = \lambda(0). 
    \end{align}
    Taking a derivative in $\epsilon$ at $0$ and applying $Z^*$ in \eqref{eq: 2 alg lemma} gives 
    \begin{align}
        Z^*T(Z,Y,Y) = \lambda(0) = a.
    \end{align}
    Plugging this into \eqref{eq: 1 alg lemma}, we find $a = 0$.\\

\noindent \textbf{Step 2: $T(Y,X,X), T(X,Y,Y) = 0$}\\

Write $T(Y,X,X) = aX$ and $T(X,Y,Y) = bY$. We will show $a,b=0$. Define the coefficients 
\begin{align}
    r &= Z^* T(X,Y,Z) \\
    s &= Z^*T(Y,Z,X) \\
    t &= Z^*T(Z,X,Y).
\end{align}
By the plane-property, 
\begin{align}\label{eq: 3 alg lemma}
    T(X, Y+\epsilon Z, Y +\epsilon Z) = \lambda(\epsilon) (Y + \epsilon Z) + \mu(\epsilon) X,
    \end{align}
    which implies $r-t=b$. 
    Applying the plane property again (of course for a possibly different choice of $\lambda, \mu$), 
    \begin{align}
        T(Y, X + \epsilon Z, X + \epsilon Z) = \lambda(\epsilon) (X+ \epsilon Z) + \mu(\epsilon)Y.
    \end{align}
    This implies $s-r = a$. The Bianchi identity gives $r + s + t =0$. We will also need the following two applications of the trace-free condition: 
    \begin{align}
        0 &= \mathrm{trace}(U \mapsto T(U,X,Y)) \\
        &= -b + t \\
        0 &= \mathrm{trace}(U \mapsto T(U,Y,X)) \\
        &= -a - s.
    \end{align}
    In total, all the equations are: 
    \begin{align}
        r - t &= b \\
        s-r &= a \\
        r + s + t &= 0 \\
        b &= t \\
        a &= -s.
    \end{align}
    These equations easily imply $a = b = 0$. \\

    \noindent \textbf{Step 3: $T \equiv 0$}\\

    We now use $T(X,Y,Y) = 0$ for all $X,Y \in V$ to conclude that $T \equiv 0$. By using
    \begin{equation}
    T(X, Y+ Z, Y+Z) =0,
    \end{equation}
     expanding, and using 
     \begin{equation}
     T(X,Y,Y) = T(X,Z,Z) = 0,
     \end{equation}
      we get  
    \begin{align}
        T(X,Y,Z) &= - T(X,Z,Y). 
    \end{align}
    Thus $T$ is skew-symmetric in the second two entries. Combined with the skew-symmetry in the first two entries and the Bianchi identity,
    \begin{align}
        0 &= T(X,Y,Z) + T(Y,Z,X) + T(Z,X,Y) \\
        &= 3T(X,Y,Z),
    \end{align}
    and hence $T \equiv 0$. 
\end{proof}

    We are now ready to prove $W^\nabla \equiv 0$. Fix $x \in M$. We want to apply Lemma \ref{lem: lin alg lemma} to 
    \begin{equation}
    (X,Y,Z) \mapsto W_x^\nabla(X,Y)Z.
    \end{equation}
     We already know that $W_x^\nabla$ satisfies items 1, 2, 3 from Lemma \ref{lem: lin alg lemma} (see Section \ref{subsec: appendix affine}). We show that item 4, the plane property, holds. Observe that it suffices to prove 
     \begin{equation}
     W_x^\nabla(X,Y)Y \in \mathrm{span}(X,Y)
     \end{equation}
      for all $X,Y$ in an open subset of $\R^3$. Indeed the condition
      \begin{equation}
      W_x^\nabla(X,Y) Y \in \mathrm{Span}(X,Y)
      \end{equation}
       is equivalent to 
       \begin{equation}
       P(X,Y) = 0
       \end{equation}
        for the polynomial map 
       \begin{equation}
       P(X,Y) = X \wedge Y \wedge W_x^\nabla(X,Y)Y,
       \end{equation}
         and a polynomial that vanishes on an open set must vanish everywhere. Fix $X,Y \in \mathcal U_x$. By the totally geodesic condition, Definition \ref{def: T.G. geometric}, there is a totally geodesic surface $S$ with 
         \begin{equation}
         T_x S = \mathrm{Span}(X,Y).
         \end{equation}
          By Lemma \ref{eq: TG surface and weyl curvature}, 
          \begin{equation}
          W_x^\nabla(X,Y)Y \in \mathrm{Span}(X,Y).
          \end{equation}
           Since $\mathcal U_x$ is open, we conclude 
           \begin{equation}
           W_x^\nabla(X,Y)Y \in \mathrm{Span}(X,Y)
           \end{equation}
            for all $X,Y \in \R^3$. Thus we can apply Lemma \ref{lem: lin alg lemma} to conclude $W^\nabla_x \equiv 0$, finishing the proof.

\section{Proof of Theorem \ref{260615theorem1_30}}

\subsection{A polynomial Wolff axiom}
For a map 
\begin{equation}
\Phi: \R^2\times \R\times \R^2\to \R^3
\end{equation}
that is analytic and semi-algebraic, we first write it as 
\begin{equation}
\Phi(w, t; \xi)= (X(w, t; \xi), t).
\end{equation}
Let $\delta_{\Phi}>0$ be a small constant that will be chosen later, depending only on $\Phi$.
Without loss of generality, we assume that $X$ is in a normal form. 
Let $\epsilon_{\Phi}>0$ be a sufficiently small constant depending only on $\Phi$. Let $0< \kappa< \epsilon_{\Phi}$. Take a $\kappa$-net of the direction space $\B^{(2)}_{\epsilon_{\Phi}}$, and write it as $\{\xi_{\iota}\}_{\iota}$; take a $\kappa$-net of the physical space $\B^{(2)}_{\epsilon_{\Phi}}$, and write it as $\{w_{\iota'}\}_{\iota'}$. We use 
\begin{equation}\label{260619e10_3}
\{
(X(w_{\iota'}, t; \xi_{\iota})+
\kappa\cdot \B^{(2)}_{\epsilon_{\Phi}}
, t): |t|\le \epsilon_{\Phi}
\}
\end{equation}
to denote a $\kappa$-tube pointing in the direction $\xi_{\iota}$, and 
\begin{equation}
\{(X(w_{\iota'}, t; \xi_{\iota}), t): |t|\le \epsilon_{\Phi}\}
\end{equation}
is called the central curve of the tube.  We use $\T_{\xi_{\iota}}$ to denote the collection of all $\kappa$-tubes of the form \eqref{260619e10_3} pointing in the direction $\xi_{\iota}$, and 
\begin{equation}
\T:=\cup_{\iota} \T_{\xi_{\iota}}.
\end{equation}
For a constant $C\ge 1$ and for $T\in \T$, we use $C T$ to denote 
\begin{equation}
\{
(X(w_{\iota'}, t; \xi_{\iota})+
C \kappa\cdot \B^{(2)}_{\epsilon_{\Phi}}
, t): |t|\le \epsilon_{\Phi}
\}.
\end{equation}
When applying the polynomial partitioning in \cite{guth2016restriction}, we will pick a non-zero polynomial $P$ of a certain degree $D$ defined on $\R^3$; without loss of generality we also assume that $P$ is a product of non-singular polynomials.  Let $Z(P)$ be the zero set of the polynomial $P$. Define 
\begin{equation}
W:= \mathcal{N}_{\kappa} (Z(P)),
\end{equation}
the $\kappa$-neighborhood of $Z(P)$.

\begin{definition}
For a tube $T\in \T$, we say that $T$ is tangent to $W$ if 
\begin{equation}
T\cap W\neq \emptyset, 
\end{equation}
and for every non-singular point $z\in Z(P)$ lying in $10 T$, it holds that 
\begin{equation}
\mathrm{Angle}(v_{z'}(T), T_z Z(P))\le \kappa^{1-\delta_{\Phi}},
\end{equation}
where $T_z Z(P)$ denotes the tangent plane,  $z'$ denotes the point on the central curve of $T$ that is closest to $z$, and $v_{z'}(T)$ denotes the tangent direction of the central curve of $T$ at $z'$. We use $\T_{\mathrm{tang}}$ to denote the collection of all tubes tangent to $W$. 
\end{definition}

The main ingredient we need to prove Theorem \ref{260615theorem1_30} is as follows. 

\begin{proposition}\label{260622prop10_2}
Under the above notation, and under the assumption that $\Phi$ fails the worst compression condition, we have 
\begin{equation}\label{260629e10_10}
\#
\{
\xi_{\iota}: \T_{\mathrm{tang}}\cap \T_{\xi_{\iota}}\neq \emptyset 
\}
\lesim_{D} 
\kappa^{-(2-\delta_{\Phi})},
\end{equation}
where $\delta_{\Phi}>0$ is a small constant depending only on $\Phi$. 
\end{proposition}

\subsection{Definition of tangency order}

For a given point $\bfx=(x_1, x_2, t)\in \B^{(3)}_{\epsilon_{\Phi}}$, and for $\xi\in \B^{(2)}_{\epsilon_{\Phi}}$, we can find $w=w_{\bfx, \xi}\in \B^{(2)}_{2\epsilon_{\Phi}}$ such that 
\begin{equation}
\bfx= \Phi(w, t; \xi).
\end{equation}
We use $\gamma_{\bfx, \xi}$ to denote the curve 
\begin{equation}
\{\Phi(
w, t; \xi
):
|t|\le \epsilon_{\Phi}
\}.
\end{equation}
If $\bfx$ is taken to be $(w, 0)$, then we simplify write $\gamma_{w, \xi}$ instead of $\gamma_{\bfx, \xi}$, and by our convention that $\Phi$ is in a normal form, the curve $\gamma_{w, \xi}$ passes through $(w, 0)$. 
For an interval $I\subset [-2\epsilon_{\Phi}, 2\epsilon_{\Phi}]$, we denote 
\begin{equation}
\gamma_{\bfx, \xi}(I):=
\{\Phi(
w, t; \xi
):
t\in I
\};
\end{equation}
if $I$ is just one point, say $s$, then we write 
\begin{equation}
\gamma_{\bfx, \xi}(s)= \Phi(w, s; \xi).
\end{equation}
For an affine line segment $L\subset \B^2_{\epsilon_{\Phi}}$ of length $2\epsilon_{\Phi}$, and a positive real $\epsilon_{\Phi}>\rho>0$, 
we define a surface patch 
\begin{equation}
S(\bfx, L, \rho):= 
\bigcup_{
\xi\in L
}
\gamma_{\bfx, \xi}([t+ \rho, t+2\rho]),
\end{equation}
where $\bfx=(x_1, x_2, t)$. \\

For $\bfx=(x_1, x_2, t)\in \B^{(3)}_{\epsilon_{\Phi}}$, an affine line segment $L$ as above, and a positive integer $k$,  we say that the tangency order at $(\bfx, L)$ is $\le k$, written as 
\begin{equation}
\mathrm{Ord}(\bfx, L)\le k,
\end{equation}
if there exists a small constant $c_{\bfx, L}>0$ such that  the following statement holds for all $0< \rho< c_{\bfx, L}$: Consider the set $\Omega=\Omega_{\bfx, L, \rho}$ consisting of all the $\xi\in \B^{(2)}_{\epsilon_{\Phi}}$ for which there exists $w$ such that 
\begin{equation}\label{260621e10_17}
\anorm{
\set{
s\in (t+\rho, t+ 2\rho):
\mathrm{dist}(
\gamma_{w, \xi}(s), S(\bfx, L, \rho)
)\le \rho^k
}
}\ge \rho^2,
\end{equation}
it holds that 
\begin{equation}\label{260621e10_18}
|\Omega\cap B_{\rho}|\le \rho^{2+\frac{1}{k}},
\end{equation}
for every ball $B_{\rho}\subset \B^{(2)}_{\epsilon_{\Phi}}$ of radius $\rho$. \\

It is easy to see that if $\mathrm{Ord}(\bfx, L)\le k$, then 
\begin{equation}
\mathrm{Ord}(\bfx, L)\le k', \ \forall k'\ge k.
\end{equation}
The term $\rho^2$ on the right hand side of \eqref{260621e10_17} can be replaced by $\rho^{1+c}$ for $c\in (0, 1)$ as well. Intuitively, when $\mathrm{Ord}(\bfx, L)$ is high, it means that we have heavy Kakeya compression in the spirit of Bourgain \cite{Bou91} around the surface consisting of the curves $\gamma_{\bfx, \xi}, \xi\in L$.

\subsection{A dichotomy theorem for tangency orders}

We make the convention that 
\begin{equation}
\mathrm{Ord}(\bfx, L)=\infty
\end{equation}
if $\mathrm{Ord}(\bfx, L)\le k$ fails for every $k$. We continue to adopt the notation $\bfx=(x_1, x_2, t)$ unless otherwise stated.

\begin{lemma}\label{volsliceslem}
    Suppose $H \subset (0, 2) \times \R^m\times \R^n$ is a nonempty bounded and relatively closed semi-algebraic set of complexity $O(1)$. Label the points in $H$ using triples $(\rho, \bfa, \bfb)$ and define the slices $H_{\rho} = \{(\bfa, \bfb): (\rho, \bfa, \bfb) \in H\}$ and $H_{\rho, \bfa} = \{\bfb: (\rho, \bfa, \bfb) \in H\}$.

    Suppose that for every $\rho$, $H_{\rho}$ is compact. Suppose further that each $H_{\rho, \bfa}$ has dimension $<n$. Then there exists $C, M, \alpha, \beta>0$ (depending only on $H$) such that\footnote{
    One may be attempting to replace $H$ by $H_{\rho, \bfa}$ in \eqref{eq: distance lower bound}, and this will make the lemma easier to prove. However, this distance function is not as strong as the one in \eqref{eq: distance lower bound}, and cannot guarantee the openness of $\cup_{\bfa} Bad_{\rho, \bfa}$ (this openness will be used in the compactness argument below \eqref{260626e10_25}). 
    }
    \begin{equation}\label{eq: distance lower bound}
        \text{dist} ((\rho, \bfa, \bfb), H)\geq C \rho^{\alpha}
    \end{equation}
    for all $\bfb$ except those in a set $Bad_{\rho, \bfa}$ of Lebesgue measure $\leq M \rho^{\beta}$, uniformly for every $(\rho, \bfa)$ with $0<\rho \leq 1$. Moreover,
    \begin{equation}
        \bigcup_{\bfa} \pnorm{
        \{\bfa\}\times Bad_{\rho, \bfa}
        }
    \end{equation}
    is open. 
\end{lemma}

Lemma \ref{volsliceslem} will be used for a compactness argument and we postpone its proof to Subsection \ref{Proofofvolslices}. For now let us take it for granted and apply it to prove the following theorem:

\begin{theorem}[Dichotomy theorem for tangency orders]\label{Dichotomythm}
Under the assumption that the map $\Phi$ is analytic and semi-algebraic, 
    exactly one of the two alternatives holds: 
    \begin{enumerate}
        \item[(i)]  $\mathrm{Ord}(\bfx, L)$  is uniformly bounded for all $\bfx\in \B^{(3)}_{\epsilon_{\Phi}}$ and all $L$, and the smallness threshold $c_{\bfx, L}$ in the definition of tangency orders can be taken to be uniform; 
        \item[(ii)] 
        There is some $(\bfx_0, L_0, \rho_0)$ with $\bfx_0 = (x_{10}, x_{20}, t_0)$ and an open set of directions in $\B^{(2)}_{\epsilon_{\Phi}}$, such that for each direction $\xi$ in this set there is a curve $\gamma_{w, \xi}$ for some $w$ satisfying that  $\gamma_{w, \xi}([t_0+\rho_0, t_0+2\rho_0])$ lies entirely in the surface $S_{\bfx_0, L_0, \rho_0}$.
    \end{enumerate}
\end{theorem}

\begin{proof}[Proof of Theorem \ref{Dichotomythm}]
    For every $\rho\in (0, \epsilon_{\Phi})$, every $\bfx\in \B^{(3)}_{\epsilon_{\Phi}}$, every line segment $L$ as in the definition of tangency orders and every ``coarse-scale direction'' $\bfy\in [-2, 2]^2$, define a set
    \begin{multline}
        G_{\rho}(\bfx, L, \bfy) = \{\bfz \in [-2, 2]^2: \exists  w \text{ such that }\\
         \gamma_{w, \bfy+ \rho \bfz}([t+\rho, t+2\rho])\subset S(\bfx, L, \rho)\}.
    \end{multline}
Each $G_{\rho}(\bfx, L, \bfy)$ is closed by definition (with a standard limiting argument). By the effective quantifier elimination (the Tarski-Seidenberg Theorem, see for instance \cite{basu2006algorithms}), we see  each $G_{\rho}(\bfx, L, \bfy)$ is a semialgebraic set of complexity $O(1)$ (uniformly bounded). 

We have two alternatives: Either (Case 1)  the dimension of every $G_{\rho}(\bfx, L, \bfy)$ at every $(\bfz, \rho)$  is $\leq 1$ (i.e. not full dimension), where we will see alternative (i) occurs, or (Case 2) some $G_{\rho}(\bfx, L, \bfy)$ has dimension $2$ (i.e. full dimension), where we will see alternative (ii) occurs. 

Suppose first that we are in Case $1$. We would like to use Lemma \ref{volsliceslem}. 
Define
\begin{equation}
    G_{\rho} = \{(\bfx, L, \bfy, \bfz): \bfz \in G_{\rho} (\bfx, L, \bfy)\}
\end{equation}
for $\rho\in (0, \epsilon_{\Phi})$. These are compact semialgebraic sets of complexity $O(1)$ by definition.  We can invoke Lemma \ref{volsliceslem} (relatively closedness follows immediately from continuity) for $\bfa = (\bfx, L, \bfy)$ and $\bfb=\bfz$ and find $C, M, \alpha, \beta$ satisfying the conclusion of that Lemma.

Fix once and for all a number $\sigma\in (0, 1)$; for example, one may take $\sigma=1/10.$ For a surface patch $S=S(\bfx, L, \rho)$, an interval $I$, and a curve $\gamma$, denote 
\begin{equation}\label{260626e10_25}
    d_S(\gamma|_I):=\sup_{s\in I} \dist(\gamma(s), S(\bfx, L, \rho)).  
\end{equation}
For $0<\rho< \epsilon_{\Phi}$, consider \begin{equation}
    r(\rho) = \min_{(\bfx, L, \bfy, \bfz): \bfz \notin Bad_{\rho, \bfx, L, \bfy}} \min_{w} \min_{I\subset [t+\rho, t+2\rho]: \ell(I)=\rho^{2+\sigma}} d_S(\gamma_{w, \bfy+ \rho\bfz}|_I),
\end{equation}
where $S=S(\bfx, L, \rho)$, and $I$ is an interval of length $\rho^{2+\sigma}$. Note that this is the minimum of a positive\footnote{The positivity follows from the definition of $G_{\rho}$: If $\bfz\notin Bad_{\rho, \bfx, L, \bfy}$, then $\bfz \notin G_{\rho}(\bfx, L, \bfy)$; here analyticity is crucial. } continuous function on a compact set, and therefore $r(\rho)>0, \forall \rho>0$. Moreover, by definition and effective quantifier elimination (see \cite{basu2006algorithms}), $r(\rho)$ is a semialgebraic function of complexity $O(1)$. Since $r(\rho) > 0$ for $\rho > 0$, the function 
\begin{align}
    h(s) := \frac{1}{r(1/s)}, s > \epsilon_\Phi^{-1}
\end{align}
is semialgebraic. By a standard estimate for the growth of semialgebraic functions of one variable \cite[Proposition 2.11]{Coste2002SemialgebraicGeometry}, there exists $M \geq \epsilon_\Phi^{-1}$, an integer $\nu >0$, and $K \geq 1$ such that $h(s) \leq K^{-1}s^\nu$ for $s > M$. That is, $r(\rho) \geq K\rho^\nu$ for $\rho < 1/M$. By shrinking $\epsilon_\Phi$ (note that $M$ only depended on $\Phi$), we can assume 
\begin{equation}
    r(\rho)\geq K\rho^{\nu}
\end{equation}
for $\rho \in (0,\epsilon_\Phi)$. 
So far we have proven that uniformly for every $(\bfx, L, \rho), 0<\rho< \epsilon_{\Phi}$, as long as $\bfz \notin Bad_{\rho, \bfx, L, \bfy}$, every segment of length $\rho^{2+\sigma}$ in every curve $\gamma$ in direction $\bfy+ \rho \bfz$ has  distance $\geq K\rho^{\nu}$ against $S(\bfx, L, \rho)$, where distance is given by \eqref{260626e10_25}. Moreover, 
the set $Bad_{\rho, \bfx, L, \bfy}$ has measure $\leq M\rho^{\beta}$ in $\bfz$, and this will give rise to the condition \eqref{260621e10_18} (recall we parametrize directions as $\bfy+ \rho \bfz$ and the extra $\rho^2$ in \eqref{260621e10_18} comes from Jacobian). 
As a consequence we know that no such segment of length $\rho^{2+\sigma}$ can lie entirely in the $\frac{K}{2}\rho^{\nu}$-neighborhood of $S(\bfx, L, \rho)$. The latter has complexity $O(1)$, so if we divide $\gamma$ into disjoint segments of length $\rho^{2+\sigma}$, only $O(1)$ many of those can intersect this neighborhood (since each segment must exit this neighborhood). This means that all $s$ satisfying 
\begin{equation}
    \text{dist} (\gamma(s), S(\bfx, L, \rho))\leq \frac{K}{2}\rho^{\nu}
\end{equation}
is always contained in $O(1)$ many intervals of length $\rho^{2+\sigma}$. Comparing this with the definition of tangency orders, we see $\mathrm{Ord} (\bfx, L)$ is uniformly bounded for every $(\bfx, L)$. Hence alternative (i) occurs.

Finally, let us suppose we are in Case 2. This just means that for some $\rho_0$ and some $(\bfx_0, L_0)$ with $\bfx_0 = (x_{10}, y_{10}, t_0)$, there is a full ball of directions such that for every direction in that ball, there is a trajectory $\gamma$ such that $\gamma([t_0+\rho_0, t_0+2\rho_0])$ lies entirely in the surface $S(\bfx_0, L_0, \rho_0)$. This by definition means alternative (ii) occurs.

\end{proof}

\subsection{Proof of Proposition \ref{260622prop10_2}}

We use Theorem \ref{Dichotomythm}. First, we point out that alternative (ii) implies the worst compression condition.  To see this, we just need to take two different directions $\xi', \xi''$ from $L_0$, and consider the curves passing through $\bfx_0$ and pointing in the directions $\xi', \xi''$, respectively. The rank condition required by Definition \ref{260615defi1_19} (worst compression condition) holds because of the openness of directions in alternative (ii).\\

Thus, to prove Proposition \ref{260622prop10_2}, it suffices to assume that we are in alternative (i) of Theorem \ref{Dichotomythm}. We assume this for the rest of the proof, and assume 
\begin{equation}
    \mathrm{Ord}(\bfx, L)\leq k_0,\  \forall \bfx, \ \forall L.
\end{equation}
We will also suppress the dependence of all constants on $k_0$, or on $c_{\bfx, L}$, or the degree $D$ of the polynomial $P$. Note that $c_{\bfx, L}$ is uniform for all $\bfx$ and $L$.

For the sake of convenience, the proof below is presented as proof by contradiction. But we remark that everything is explicit enough that it is not hard to write down a straightforward proof following the one here. We assume 
\begin{equation}
\#
\{
\xi_{\iota}: \T_{\mathrm{tang}}\cap \T_{\xi_{\iota}}\neq \emptyset 
\}
\gtrsim 
\kappa^{-(2-\delta_{\Phi})}
\end{equation}
and will derive a contradiction for some fixed $\delta_{\Phi}>0$ and all small enough $\kappa$. 

Our proof will rely on the assumption in the definition of $\T_{\mathrm{tang}}$ that $T \in \T_{\mathrm{tang}}$ is tangent to $W$ with an angle $O(\kappa^{1-\delta_{\Phi}})$. Because we only have this tangency, it is convenient to work with scale \[\tilde{\kappa} = \kappa^{1-\delta_{\Phi}}\] and work with $\tilde{\kappa}$-tubes.  We take a $\tilde{\kappa}$-net of ($\tilde{\kappa}$-scale) directions and write the collection as $\tilde{\Xi} = \{\tilde{\xi}_j\}_{j}$. By the counter-assumption that
\begin{equation}
    \#
\{
\xi_{\iota}: \T_{\mathrm{tang}}\cap \T_{\xi_{\iota}}\neq \emptyset 
\}
\gtrsim 
\kappa^{-(2-\delta_{\Phi})},
\end{equation}
we can find a collection $\tilde{\Xi}_{\text{tang}}$ of ($\tilde{\kappa}$-scale) directions $\tilde{\xi}_j$ of cardinality \[|\tilde{\Xi}_{\text{tang}}|\gtrsim \kappa^{-(2-3\delta_{\Phi})},\] and can choose a $\tilde{\kappa}$-tube (that we will call $\tilde{T}_{\tilde{\xi}_j}$ henceforth) in each direction $\tilde{\xi}_j \in \tilde{\Xi}_{\text{tang}}$ that contains a $\kappa$-scale tube in  $\T_{\text{tang}}\cap \T_{\xi_{\iota}}$ with 
\begin{equation}
    |\tilde{\xi}_j - \xi_{\iota}| = O(\tilde{\kappa}).
\end{equation}
We also take a collection $\B_{\tilde{W}}$ of $\tilde{\kappa}$-balls that form a bounded-overlapping cover of $\tilde{W} = \mathcal{N}_{\tilde{\kappa}} (Z(P))$. If some $\tilde{T}_{\tilde{\xi}_j}$ and $B \in \B_{\tilde{W}}$ satisfy $\tilde{T}_{\tilde{\xi}_j}\cap B \neq \emptyset$, we say $\tilde{T}_{\tilde{\xi}_j}$ \emph{passes through} $B$. For a ball $B \in \B_{\tilde{W}}$, we define $t(B)$ to be the $t$-coordinate of its center. We discretize the $t$-coordinates, and align the balls $\{B\}$ so that there are only $O(\tilde{\kappa}^{-1})$ many different $t(B)$'s (define $Y$ to be this set of $t(B)$).  We use $\B_{\tilde{W}, t_0}$ to denote the set of all balls $B\in \B_{\tilde{W}}$ with $t$-coordinate $t_0$.

By Wongkew's theorem in \cite{wongkew1993volumes}, 
\begin{equation}\label{upperboundoftildeB}
    |\B_{\tilde{W}}|\lesssim \tilde{\kappa}^{-2}.
\end{equation}
Hence for $\gtrsim \tilde{\kappa}^{-1}$ different $t$-coordinate choices in $t \in Y$, there are $\lesssim \tilde{\kappa}^{-1}$ many $B \in \B_{\tilde{W}}$ with $t(B) = t$. We call these \emph{good} $t$. We will work with a special scale $\rho = \tilde{\kappa}^{c}$, where $0<c< \frac{1}{k_0}$ (this ensures $\tilde{\kappa} < \rho^{k_0}$) is a constant depending on $k_0$ to be specified later. Its role will become relevant when we use the finite order assumptions \eqref{260621e10_17} and \eqref{260621e10_18}. In terms of the dependence of parameters, the reader should imagine that we are first handed $k_0$, then we choose $c$ to be small depending on $k_0$, then we choose $\delta_{\Phi}$ to be even much smaller depending on $k_0$ and $c$. With this $\rho$, we now find and fix a subset $Y_{\text{good}}\subset Y$ contained in a time window of length $10\rho$, such that there are $\gtrsim \frac{\rho}{\tilde{\kappa}}$ many good $t \in Y_{\text{good}}$. Our discussion below will only be applied to $B$ in this time window $t \in Y_{\text{good}}$.

Each $\tilde{T}_{\tilde{\xi}_j}$ passes through a ball for every $t$, and in particular for every good $t$, in $Y_{\text{good}}$. By Cauchy-Schwarz and \eqref{upperboundoftildeB}, we see the number of triples $(\tilde{T}_{\tilde{\xi}_{j_1}}, \tilde{T}_{\tilde{\xi}_{j_2}}, B)$ with good $t(B) \in Y_{\text{good}}$  and satisfying that both $\tilde{T}_{\tilde{\xi}_{j_1}}$ and $\tilde{T}_{\tilde{\xi}_{j_2}}$ pass through $B$ is
\begin{equation}
    \gtrsim \frac{\rho}{\tilde{\kappa}}\cdot \tilde{\kappa}^{-1}\cdot(\tilde{\kappa}\cdot \kappa^{-(2-3\delta_{\Phi})})^2\sim \rho \kappa^{-(4-6\delta_{\Phi})}.
\end{equation}
By double counting, we can then find 
 \begin{equation}
     \bfx = (x_1, x_2, t(\bfx)) \in \mathcal{N}_{\tilde{\kappa}}(Z(P)),
 \end{equation}
 where $t(\bfx)$ denotes the $t$-coordinate of $\bfx$, 
 such that among the above triples, we can find at least $\rho\tilde{\kappa}\cdot \kappa^{-(4-6\delta_{\Phi})}$ many, such that all $\tilde{T}_{\tilde{\xi}_{j_1}}$ in those selected triples contain $\bfx$ and that 
 \begin{equation}
     t(B) \in [t(\bfx)+\rho, t(\bfx)+2\rho]
 \end{equation}
 for all $B$ in the selected triples.

Now we look only at these selected triples with $\tilde{T}_{\tilde{\xi}_{j_1}}$ containing $\bfx$. We can use the tangency assumption. Pick any point $\bfx' \in Z(P)$ that is $O(\tilde{\kappa})$-close to $\bfx$, and set $\Pi$ to be the tangent plane to $\bfx'$ at $Z(P)$. Then the central curve of each $\tilde{T}_{\tilde{\xi}_{j_1}}$ through $\bfx$ (up to an $O(\tilde{\kappa})$-translation) must be almost tangential to $\Pi$ (up to error $O(\tilde{\kappa})$). Thus each  $\tilde{T}_{\tilde{\xi}_{j_1}}$ for the above selected tuples must be $\tilde{\kappa}$-close to the surface $S(\bfx, L, \rho)$ for some line $L$ of directions (corresponding to the tangential condition to $\Pi$) in the range $t \in [t(\bfx)+\rho, t(\bfx)+2\rho]$. By the above lower bound of selected triples, we deduce that $S(\bfx, L, \rho)$ is $\tilde{\kappa}$-close to 
\begin{equation}
    \gtrsim \rho\tilde{\kappa}^2\cdot \kappa^{-(4-6\delta_{\Phi})}
\end{equation}
many $\tilde{T}_{\tilde{\xi}_{j_2}}$, each at 
\begin{equation}
    \gtrsim \rho\tilde{\kappa}^3\cdot \kappa^{-(4-6\delta_{\Phi})}
\end{equation}
many different $B$'s (all in $\B_{\tilde{W}}$).

Recall $\tilde{\kappa} = \kappa^{1-\delta_{\Phi}}$. Thus when $\delta_{\Phi}$ is very small (depending on $k_0$ and $c$), the above numbers $\rho\tilde{\kappa}^2\cdot \kappa^{-(4-6\delta_{\Phi})}$ and $\rho\tilde{\kappa}^3\cdot \kappa^{-(4-6\delta_{\Phi})}$ become close to $\rho\kappa^{-2}$ (if $\rho=1$, this means we have almost all directions $\tilde{\xi}_j$) and $\frac{\rho}{\kappa}$ (this means we almost have all balls on the tube $\tilde{T}_{\tilde{\xi}_{j_2}}$ in time $[t(\bfx)+\rho, t(\bfx)+2\rho]$ lying on $\mathcal{N}_{\tilde{\kappa}} (S(\bfx, L, \rho))$), respectively.

Now recall that we are in alternative (i) and have $\mathrm{Ord}(\bfx, L)\le k_0$, with precise meaning in \eqref{260621e10_17} and \eqref{260621e10_18}, and we use it to get a contradiction. We remember $c$ is very small depending on $k_0$, and then $\delta_{\Phi}$ is chosen to be very small depending on $k_0$ and the choice of $c$. For this choice (when $\delta_{\Phi}$ is very small), every direction of every $\tilde{T}_{\tilde{\xi}_{j_2}}$ will fall in $\Omega$ described in \eqref{260621e10_17}. Thus the number of directions $\tilde{\xi}_{j_2}$ in every $\rho$-ball
 is $\lesssim \rho^{2+\frac{1}{k_0}}\tilde{\kappa}^{-2}$.

In our last step of analysis, we will show that not all possible $\rho$-balls (in the space of directions) can contain the direction $\tilde{\xi}_{j_2}$. In fact, note that the tube $\tilde{T}_{\tilde{\xi}_{j_2}}$ intersects $\mathcal{N}_{\tilde{\kappa}} (S(\bfx, L, \rho))$ at 
\begin{equation}
    \gtrsim \rho\tilde{\kappa}^3\cdot \kappa^{-(4-6\delta_{\Phi})} \approx \rho \kappa^{-1}
\end{equation}
many different $B$'s (we use $\approx$ to hide factors of the form $\kappa^{O(\delta_{\Phi})}$). This implies the tangent direction $\tilde{\xi}_{j_2}$ (up to some diffeomorphism depending on $\Phi, \bfx$) is $\lessapprox \frac{\kappa}{\rho}$-close to some tangent plane of some point on $S(\bfx, L, \rho)$. But $S(\bfx, L, \rho)$ is at scale $\rho$ and its tangent plane is constant up to scale $\rho$. We take $c< \frac{1}{2}$ so that $\frac{\kappa}{\rho} < \rho$ and this ensures that all $\tilde{\xi}_{j_2}$ are contained in $\lessapprox \rho^{-1}$ many $\rho$-balls (dictated by $L$). The total number of $\tilde{\xi}_{j_2}$ is then bounded by 
\begin{equation}
    O(\rho^{1+\frac{1}{k_0}}\kappa^{-2-O(\delta_{\Phi})}).
\end{equation}
We choose $\delta_{\Phi}$ very small and this will be smaller than the actual total number $\rho\tilde{\kappa}^2\cdot \kappa^{-(4-6\delta_{\Phi})}$, causing a contradiction.
 
Finally, we comment that the above analysis is robust: The argument above can be adapted to an analysis in a ball of radius $b>\rho$ instead of the unit ball, yielding a similar conclusion with $b^{O(1)}$-loss.

\subsection{Proof of Lemma \ref{volsliceslem}}\label{Proofofvolslices}

We will need the semi-algebraic \L ojasiewicz inequality in the proof, which we state below. 

\begin{lemma}[\L ojasiewicz inequality \protect{(\cite[Corollary 2.6.7]{realalggeombook})}]\label{lem: loja}
    Let $A$ be a closed and bounded semialgebraic set and $f,g : A \to \R$ continuous semialgebraic functions such that $f^{-1}(0) \subset g^{-1}(0)$. Then there exist an integer $N > 0$ and a constant $c \in \R$, such that $|g|^N \leq c|f|$ on $A$. 
\end{lemma}

Along with the semi-algebraic \L ojasiewicz inequality Lemma \ref{lem: loja}, we will need the following volume estimate for semialgebraic fibers. Write $N_t(E) = \{x \in \R^n : \dist(x,E) < t \}$. 


\begin{lemma}[Volume estimate for semialgebraic fibers]\label{lem: semialg fiber est}
    Let $P$ be a semialgebraic parameter set and let $\Gamma \subset P \times \R^n$ be semialgebraic. Write $\Gamma_p := \{x : (p,x) \in \Gamma\}$. Suppose that all the fibers $\Gamma_p$ are contained in a fixed compact box $Q \subset \R^n$ and that $\dim \Gamma_p \leq n-1$ for all $p \in P$. Then there is a constant $C_\Gamma >0$ such that 
    \begin{align}
        |N_t(\Gamma_p)| \leq C_\Gamma t, \text{ for all $p \in P$, $0 < t \leq 1$}.
    \end{align}
\end{lemma}

Lemma \ref{lem: semialg fiber est} is a straightforward consequence of the following standard result, which we state in the special case of semi-algebraic sets. 

\begin{lemma}[\protect{\cite[Theorem 5.9]{tameGeometry}}]\label{lem: tame geometry lemma}
    Let $A$ be a semi-algebraic set of dimension $\ell < n$. Then for any ball $B_r^n \subset \R^n$ and for any $\eta > 0$, 
    \begin{align}
        |N_\eta(A \cap B_r^n)| \leq C_A (r^\ell \eta^{n-\ell} + \eta^n),
    \end{align}
    where $C_A$ is a constant that depends only on the complexity of $A$ \footnote{In \cite{tameGeometry}, $C_A$ depends on the \emph{diagram} of the semi-algebraic set $A$, but the diagram is easily bounded in terms of the complexity.}
\end{lemma}

\begin{proof}[Lemma \ref{lem: tame geometry lemma} implies Lemma \ref{lem: semialg fiber est}]
    The fibers $\Gamma_p$ are all semi-algebraic of complexity bounded in terms of $\Gamma$, so we can choose a constant $C$ in Lemma \ref{lem: tame geometry lemma} uniform in $p \in P$. 
    Choose $R$ large enough that $Q \subset B_R$. For a nonempty fiber, put $\ell_p = \dim \Gamma_p$. Applying Lemma \ref{lem: tame geometry lemma} to $\Gamma_p$, we get 
    \begin{align}
        |N_t(\Gamma_p)| &\leq C(R^{\ell_p}t^{n-\ell_p} + t^n) \\
        &\leq C(\max(1,R^{n-1}) + 1)t
    \end{align}
    for $p \in P$ and $0 < t \leq 1$. Taking $C_{\Gamma} = C(\max(1,R^{n-1}) + 1)$ finishes the proof. 
\end{proof}

We now have the tools to prove Lemma \ref{volsliceslem}. We give a brief sketch of the idea first. For fixed $(\rho,a,\eta)$, let $T_{\rho,a,\eta} = \{b : \dist((\rho,a,b),H) < \eta\}$, The hypothesis $\dim H_{\rho,a} < n$ says that $H_{\rho,a}$ contains no $n$-dimensional ball. We use Lemma \ref{lem: loja} to make this quantitative: $T_{\rho,a,\eta}$ cannot contain a ball of radius larger than $\sim\rho^{-1} \eta^{1/N}$. It follows that $T_{\rho,a,\eta}$ lies within a $O(\rho^{-1} \eta^{1/N})$ neighborhood of its boundary. Then we use Lemma \ref{lem: semialg fiber est} to bound the volume of the neighborhood of that boundary. 

\begin{proof}[Proof of Lemma \ref{volsliceslem}]
    Since $H$ is bounded, there are compact balls $A \subset \R^m$ and $B \subset \R^n$ of radii at least $100$ so that $H \subset (0,2) \times (1/10)A \times (1/10)B$. 
    Define $d(\rho,a,b) := \dist((\rho,a,b),H)$. 
    Then $d$ is continuous and semi-algebraic. 
    To see why $d$ is semi-algebraic, notice that the graph of $d$ can be described by a first-order formula expressing that some point of $\overline H$ occurs at distance $d$ and that no point of $\overline H$ is closer; Tarski-Seidenberg then applies.
    For $0 < \rho \leq 1$, $a \in A$, and $0 < \eta \leq 1$, set 
    \begin{align}\label{eq: sets Trhoaeta}
        T_{\rho,a,\eta} := \{b \in \R^n : d(\rho,a,b) < \eta\}.
    \end{align}
    These sets form a semi-algebraic family of open subsets of $\R^n$, and $T_{\rho,a,\eta} \subset B$.

    \medskip

    \noindent \textbf{Step 1: no large balls in $T_{\rho,a,\eta}$}

    We claim that there are $N \geq 1$ and $C_* > 0$ such that 
    \begin{align}
        \overline{B(c,r)} \subset T_{\rho,a,\eta} \implies r \leq C_* \rho^{-1} \eta^{1/N}.\label{eq: no large balls}
    \end{align}
    Here $B(c, r)$ is a ball of radius $r$ centered at $c$. 
    The next definition encodes that an entire ball is close to $H$. Choose $R > \mathrm{diam}(B)$ and, on the compact set $K := [0, 1] \times A \times B \times [0,R]$, define 
    \begin{align}
        \Delta(\rho,a,c,r) := \max_{|u| \leq 1} d(\rho,a,c+ru) = \max_{b' \in \overline{B(c,r)}} d(\rho,a,b').
    \end{align}
    We have in particular, 
    \begin{align}
        \Delta(\rho,a,c,r) = 0 &\iff d(\rho,a,b') = 0 \text{ for every } b' \in \overline{B(c,r)},\\
        \Delta(\rho,a,c,r) < \eta &\iff \overline{B(c,r)} \subset T_{\rho,a,\eta}. \label{eq: Delta < eta inclusion}
    \end{align}
    The maximum exists  by compactness. That $\Delta$ is continuous is a straightforward consequence of the uniform continuity of $(\rho,a,c,r,u) \mapsto d(\rho,a,c+ru)$ over $K \times \{u : |u| \leq 1\}$. The function $\Delta$ is semi-algebraic because its graph is described by the first-order conditions that the claimed value is attained for some $|u| \leq 1$ and is an upper bound for all $|u| \leq 1$ (and Tarski-Seidenberg). 

    We next identify the zero-set of $\Delta$. If $\rho, r>0$ and $\Delta(\rho,a,c,r)=0$, then $d(\rho,a,b') =0$ for every $b' \in \overline{B(c,r)}$. Hence $(\rho,a,b') \in \overline H$ for every such $b'$. Since $0 < \rho < 2$ and $H$ is relatively closed on $(0,2) \times \R^m \times \R^n$, we have $(\rho,a,b') \in H$. Therefore $\overline{B(c,r)} \subset H_{\rho,a}$. This is impossible because $r > 0$ while $\dim H_{\rho,a} < n$. We have proved 
    \begin{align}
        \Delta(\rho,a,c,r) = 0 \implies \rho r = 0 \text{ for $(\rho,a,c,r) \in K$}.
    \end{align}
    We may now apply Lemma \ref{lem: loja} on $K$ with $f = \Delta$ and $g(\rho,a,c,r) = \rho r$. We obtain $N \geq 1$ and $C_L > 0$ such that 
    \begin{align}
        (\rho r)^N \leq C_L \Delta(\rho,a,c,r) \text{ for $(\rho,a,c,r) \in K$}.\label{eq: Loj application}
    \end{align}
    Now suppose $\overline{B(c,r)} \subset T_{\rho,a,\eta}$. Then  $c \in B$ and $r \leq \mathrm{diam}(B) < R$. By \eqref{eq: Delta < eta inclusion}, $\Delta(\rho,a,c,r) < \eta$. Hence \eqref{eq: Loj application} gives 
    \begin{align}
        r < C_{L}^{1/N} \rho^{-1} \eta^{1/N}.
    \end{align}
    This proves \eqref{eq: no large balls} with $C_* := C_L^{1/N}$.

    \medskip

    \noindent \textbf{Step 2: $T_{\rho,a,\eta}$ is close to its boundary}

    Fix $(\rho,a,\eta)$ and abbreviate $T = T_{\rho,a,\eta}$.
    For $b \in T$, let $\delta(b) := \dist(b, \R^n \setminus T)$. Because $T$ is open and bounded, 
    \begin{align}
        \delta(b) = \dist(b,\partial T) > 0. 
    \end{align}
    For every $0 < s < \delta(b)$, $\overline{B(b,s)} \subset T$. Applying the estimate from Step 1 and then letting $s \uparrow \delta(b)$ gives 
    \begin{align}
        \dist(b,\partial T) \leq C_* \rho^{-1} \eta^{1/N}.
    \end{align}
    Thus taking $\tau :=C_* \rho^{-1} \eta^{1/N}$, we have 
    \begin{align}
        T \subset N_{2\tau}(\partial T). \label{eq: near boundary}
    \end{align}

    \noindent \textbf{Step 3: neighborhood of boundary has small volume}
    Define the semi-algebraic set 
    \begin{align}
        \Gamma :=\{(\rho,a,\eta,b) : 0 < \rho \leq 1, a \in A, 0 < \eta \leq 1, b \in \partial T_{\rho,a,\eta}\}.
    \end{align}
    The fibers are $\Gamma_{\rho,a,\eta} = \partial T_{\rho,a,\eta}$. Each fiber has dimension at most $n-1$, since the boundary of an open set contains no interior points. Lemma \ref{lem: semialg fiber est} therefore gives $C_0 > 0$ such that 
    \begin{align}
        |N_t(\partial T_{\rho,a,\eta})| \leq C_0 t \text{ for $0 < \rho \leq 1, a \in A, 0 < \eta \leq 1, 0 < t \leq 1$}. \label{eq: boundary estimate}
    \end{align}
    If $2\tau \leq 1$, then \eqref{eq: near boundary} and \eqref{eq: boundary estimate} give $|T_{\rho,a,\eta}| \leq 2C_0 \tau$. If $2\tau > 1$, then $T_{\rho,a,\eta} \subset B$ gives $|T_{\rho,a,\eta}| \leq |B| \leq 2|B| \tau$. Hence there is $C_1$ such that 
    \begin{align}
        |T_{\rho,a,\eta}| \leq C_1 \rho^{-1} \eta^{1/N} \text{ for $0 < \rho \leq 1, a \in A, 0 < \eta \leq 1$}. \label{eq: T vol estimate}
    \end{align}
    Finally choose $C=1$, $\alpha = 2N$, $\beta = 1$, and put 
    $\mathrm{Bad}_{\rho,a} :=T_{\rho,a,\rho^{2N}}$. By \eqref{eq: T vol estimate}, 
    \begin{align}
        |\mathrm{Bad}_{\rho,a}| \leq C_1 \rho^{-1} (\rho^{2N})^{1/N} = C_1 \rho.
    \end{align}
    Thus we can take $M = C_1$. The distance lower bound \eqref{eq: distance lower bound} holds on the complement of $\mathrm{Bad}_{\rho,a}$. Finally, $(a,b) \mapsto d(\rho,a,b)$ is continuous for fixed $\rho$, so
    \begin{align}
        \bigcup_a\pnorm{
        \{a\}\times \mathrm{Bad}_{\rho,a}
        }  &= \{(a,b) \in A \times B : d(\rho,a,b) < \rho^{2N}\} \\
        &= \{(a,b) \in \R^m \times \R^n : d(\rho,a,b) < \rho^{2N}\} 
    \end{align}
    is relatively open and contains $H_\rho$. 
\end{proof}

\subsection{Proof of Theorem \ref{260615theorem1_30}}

Now we are ready to prove 
Theorem \ref{260615theorem1_30}, by applying Proposition \ref{260622prop10_2} and its small ``perturbation" (see the end of the proof of Proposition \ref{260622prop10_2}). The argument is almost standard, and we will give a sketch by following the framework of Hickman and Rogers \cite{hickman2019improved}, Hickman, Rogers and Zhang \cite{HRZ22} and Guth, Hickman and Iliopoulou \cite{GHI19}, which further generalize the framework of Guth \cite{guth2016restriction}. 
It seems to us that the best way to sketch the proof of Theorem \ref{260615theorem1_30} is to sketch the proof of Corollary \ref{260617corollary1_31}, by following  the numerology in \cite{DGGZ24}. Once this is done, it will be very clear how the proof of Theorem \ref{260615theorem1_30} goes.

Let us only point out the modifications to the numerology in \cite{DGGZ24}.
We will replace $\delta_{\Phi}$ in Proposition \ref{260622prop10_2} by $\delta_{\phi}$, a small positive constant depending on the phase function $\phi$.  
Theorem 4.6 of \cite{DGGZ24} is about the polynomial Wolff axiom at the largest scale, and it is now replaced by Proposition \ref{260622prop10_2}. The difference is that in the setting of \cite[Theorem 4.6]{DGGZ24}, the left hand side of \eqref{260629e10_10} is essentially bounded by $\kappa^{-1}$, which is the best possible scenario one can expect,  while in the current setting we only know the bound $\kappa^{-(2-\delta_{\phi})}$ for some $\delta_{\phi}>0$ depending on $\phi$. \footnote{Note that $\kappa^{-2}$ is the trivial bound.} As a consequence, Corollary 4.8 in \cite{DGGZ24} needs to be replaced by
\begin{equation}\label{260629e10_62}
    \left\|f_{\iota, S_2}^*\right\|_2^2 \lesssim_{\phi} r_2^{-\delta'_{\phi}}\|f\|_{\infty}^2,
\end{equation}
where $\delta'_{\phi}>0$ is another small constant that depends only on $\phi$. Note that \eqref{260629e10_62} with $\delta'_{\phi}=0$ is trivial. 
 It is exactly the gain $r_2^{-\delta'_{\phi}}$ over the trivial bound  that allows us to obtain the gain $\kappa_{\phi}$  in Corollary \ref{260617corollary1_31} over the universal exponent $10/3$. 

Let us be slightly more precise. 
In the previous paragraph, we only talked about the polynomial Wolff axiom at the largest scale. The framework of Hickman and Rogers \cite{hickman2019improved} and Hickman, Rogers and Zhang \cite{HRZ22} requires that we obtain a non-trivial polynomial Wolff axiom at other scales as well. Imagine that we are discussing polynomial Wolff axioms at the scale $1/N_2$, where $N_2\ge 1$ occurred in \cite[page 992]{DGGZ24}. At this new scale, Proposition 4.9 in \cite{DGGZ24} needs to be replaced by 
\begin{equation}\label{260629e10_63}
    \left\|f_{\iota, S_2}^*\right\|_2^2 \lesssim_{\phi} \min \left\{\left(\frac{N}{r_2}\right)^{\frac{1}{\delta'_{\phi}}} r_2^{-\delta'_{\phi}},\left(r_2^{-1 / 2}+\frac{r_2}{N}\right)\right\}\left(r_2\right)^{\delta_0}\|f\|_{\infty}^2,
\end{equation}
where $\delta_0$ in \cite[Proposition 4.9]{DGGZ24} is an extremely small constant that will eventually be sent to zero. On the right hand side of \eqref{260629e10_63}, there are two terms under the minimum. The first term there is analogous to \eqref{260629e10_62}, and  comes from modifying the proof of Proposition \ref{260622prop10_2} (see the end of the proof of Proposition \ref{260622prop10_2}). The second term under the minimum in \eqref{260629e10_63}  has nothing to do with the geometry of $\phi$ and the proof in \cite[Proposition 4.9]{DGGZ24} works for all $\phi$. 

With \eqref{260629e10_62} and \eqref{260629e10_63}, we can now repeat the entire calculation in \cite{DGGZ24}, and obtain that $\kappa_{\phi}$ in Corollary \ref{260617corollary1_31} is given by 
\begin{equation}
    \frac{
    2(\delta'_{\phi})^2
    }{
    9+ 9 \delta'_{\phi}+ 6 (\delta'_{\phi})^2
    }>0.
\end{equation}
This finishes the sketch of the proof of Corollary \ref{260617corollary1_31}.

\appendix

\section{Geometry of paths}\label{250711appendix_a}

We recall some standard material on sprays and their invariants. Our conventions mostly follow \cite{Shen01}, though in our setup we work on an open fiberwise conic subset of $TM \setminus 0$ instead of all of $TM \setminus 0$. Since the definitions and results we use work locally on $TM \setminus 0$, this does not cause any complications. For the more specific material on affine sprays, we also follow \cite[Chapter 1.3]{affDiffGeoKatsumi}

\subsection{Riemann, Weyl, and Douglas curvatures}

 The Riemann tensor $\mathbf R$ on a spray space $(M, \mathcal U, \bG)$ is an element of 
 \begin{equation}
 \Gamma(\mathcal U, \mathrm{End}(\pi^* TM))
 \end{equation}
  (some other books might call this the Jacobi endomorphism). What that means is for each $(x,y) \in \mathcal U$, $R_{(x,y)}$ is a linear map on $T_x M$, and these maps vary smoothly in $(x,y)$. 
This is an intrinsic object which can be realized in local coordinates as follows. To do this, we first define the \emph{Christoffel symbols}
\begin{align}
    \Gamma_{jk}^i = \frac{\partial^2 G^i}{\partial y^j \partial y^k}
\end{align}
on $\mathcal U$. Then 
\begin{align}
    \mathbf{R}_{(x,y)}v = R_k^i(x,y)v^k \frac{\partial}{\partial x^i} |_x,   
\end{align}
where $R_k^i(x,y)$ are the local functions on $\mathcal U$ given by 
\begin{align}
    R_k^i = (\frac{\partial \Gamma_{jl}^i}{\partial x^k} - \frac{\partial \Gamma^i_{jk}}{\partial x^l} + \Gamma_{ks}^i \Gamma_{jl}^s - \Gamma_{jk}^s \Gamma_{ls}^i) y^j y^l.
\end{align}
It satisfies $\mathbf{R}_{(x,y)}(y) =0$. \\

The Weyl tensor 
\begin{equation}
\mathbf{W} \in \Gamma(\mathcal U, \mathrm{End}(\pi^* TM))
\end{equation}
 is another intrinsic object which has a local coordinate realization as follows. Define 
\begin{align}
    A_k^i = R_k^i - \frac{1}{n-1} R_m^m \delta_k^i
\end{align}
and 
\begin{align}
    W_k^i(x,y) = A_k^i - \frac{1}{n+1} \frac{\partial A_k^m}{\partial y^m} y^i.
\end{align}
Then 
\begin{equation}
\mathbf W_{(x,y)}(u) = W_k^i(y) u^k \frac{\partial}{\partial x^i}|_x.
\end{equation}
 We have $\mathbf W_{(x,y)}(y) = 0, \mathrm{trace}(\mathbf W_{(x,y)}) = 0$. 
Define the \emph{local projective spray} associated with $\bG$ by
\begin{equation}
\Pi = y^i \frac{\partial}{\partial x^i} - 2\Pi^i(y) \frac{\partial}{\partial y^i},
\end{equation}
 where
\begin{align}
    \Pi^i = G^i - \frac{1}{n+1} \frac{\partial G^m}{\partial y^m}y^i.
\end{align}
The Douglas curvature 
\begin{equation}
\mathbf D \in \Gamma(\mathcal U, \mathrm{Hom}((\pi^* TM)^{\otimes 3}, \pi^* TM))
\end{equation}
 is given in local coordinates by 
 \begin{equation}
 \mathbf D_{(x,y)}(u,v,w) = D^i_{jkl}u^j v^k w^l \frac{\partial}{\partial x^i}|_x
 \end{equation}
  where 
\begin{align}
    D_{jkl}^i = \frac{\partial^3 \Pi^i}{\partial y^j \partial y^k \partial y^l}.
\end{align}
One has $\mathbf D = 0$ when $\mathbf G$ is projectively related to an affine spray, and we may further assume the affine spray has symmetric Ricci tensor by passing to the associated local projective spray. As we will see below an affine spray is associated to an affine connection, and affine connections with symmetric Ricci tensor are said to be \emph{locally equiaffine} in \cite[Chapter 1.3]{affDiffGeoKatsumi}.
Below we discuss these affine sprays in more detail. 

\subsection{Special case of affine sprays}\label{subsec: appendix affine}

The \emph{Berwald tensor} 
\begin{equation}
\mathbf B \in \Gamma(\mathcal U, \mathrm{Hom}((\pi^* TM)^{\otimes 3}, \pi^* TM))
\end{equation}
 is given in local coordinates by 
 \begin{equation}
 \mathbf B_{(x,y)}(u,v,w) = B_{jkl}^i u^j v^k w^l \frac{\partial}{\partial x^i}|_x
 \end{equation}
  where 
\begin{align}
    B^i_{jkl} = \frac{\partial^3 G^i}{\partial y^j \partial y^k \partial y^l}.
\end{align}
A spray $\bG$ is affine if $\mathbf B = 0$. In this case, the Christoffel symbols 
\begin{equation}
\Gamma_{jk}^i(x,y) = \frac{\partial^2 G^i}{\partial y^j \partial y^k}
\end{equation}
 are independent of $y$ and 
\begin{align}
    G^i(y) = \frac{1}{2} \Gamma^i_{jk}(x) y^j y^k. 
\end{align}
A priori this holds only on $\mathcal U$, but it of course naturally extends the spray to $(M, TM)$. 
These Christoffel symbols then uniquely define a torsion-free affine connection 
\begin{equation}
\nabla : \Gamma(TM) \times \Gamma(TM) \to \Gamma(TM)
\end{equation}
 with the same geodesics as the spray $\bG$. In local coordinates, 
\begin{align}
    \nabla_X Y = X^j (\frac{\partial Y^i}{\partial x^j} + \Gamma_{jk}^i Y^k) \frac{\partial}{\partial x^i}
\end{align}
In local coordinates, the geodesic equation $\nabla_{\dot \gamma} \dot \gamma = 0$ is 
\begin{align}
    \ddot \gamma^i + \Gamma_{jk}^i \dot \gamma^j \dot \gamma^k = 0.
\end{align}
One can now view the Riemann curvature of $\nabla$ as the $(1,3)$-tensor
\begin{align}\label{eq: affine riemann tensor formula}
    R^\nabla(X,Y)Z = ([\nabla_X, \nabla_Y] - \nabla_{[X,Y]})Z. 
\end{align}
If $\mathbf R$ is the Riemann curvature of $\bG$, we have 
\begin{align}
    \mathbf R_{(x,y)}(X) = R_x^\nabla(y,X)y.  
\end{align}
The Riemann curvature $R^\nabla$ satisfies the following properties: 
\begin{itemize}
    \item (Skew-symmetry) $R^\nabla(Y,X) = -R^\nabla(X,Y)$ 
    \item (First Bianchi identity) $R^\nabla(X,Y)Z + R^\nabla(Z,X)Y + R^\nabla(Y,Z)X = 0$
\end{itemize}
Similarly the Weyl tensor of $\nabla$ can be viewed as the $(1,3)$-tensor 
\begin{align}\label{eq: affine weyl tensor formula}
    W^\nabla(X,Y)Z = R^\nabla(X,Y)Z - P(X,Z)Y + P(Y,Z)X, 
\end{align}
where\footnote{This is the definition of Weyl tensor in \cite[Page 17]{affDiffGeoKatsumi}, which assumes $\nabla$ is equiaffine. There is a similar but longer formula for the Weyl tensor when $\nabla$ is affine but not equiaffine, but this is not necessary.}
\begin{equation}
P(X,Y) = \frac{1}{n-1}\mathrm{trace}(Z \mapsto R^\nabla(Z,X)Y).
\end{equation}
That is, $W^\nabla$ is the projective trace-free part of $R^\nabla$. This satisfies skew symmetry and the first Bianchi identity, and in addition the trace-free condition
\begin{equation}
\mathrm{trace}(Z \mapsto W^\nabla(Z,X)Y) = 0.
\end{equation}
 Similarly to the Riemann tensor, we have \begin{equation}
 \mathbf W_{(x,y)}(X) = W_x^\nabla(y,X)y.
 \end{equation}
This finishes our discussion on affine sprays.

	\section{Schr\"odinger equation with potentials}\label{260615append_c}

	Let $x\in\R^n$ and $t\in\R$. We consider the Schr\"odinger equation
	\begin{align}\label{Schro}
		\begin{cases}
			i \partial_t u = -\frac12 \Delta u +V(x)u,   \\
			u(x,0)= u_0(x),  
		\end{cases}
	\end{align}
	with a potential $V(x)$ satisfying the following growth conditions at infinity:
	\begin{align}
            |V(x)|\le C(1+|x|^2),\qquad \\
		|\partial_x^\alpha V(x)|\le C_\alpha,\quad 2\le |\alpha|\le C_n.
	\end{align}
	It is well known that
	\begin{equation}
	L: u\to -\frac12\Delta u+ V(x)u,
	\end{equation}
	 defined on $C_0^\infty(\R^n)$, is essentially self-adjoint in $L^2(\R^n)$, and that the equation has a unique solution given by 
    \begin{align}
        u(x,t)= e^{-itH}u_0(x),     
    \end{align}
    where $H$ is the unique self-adjoint extension of $L$. Fujiwara \cite[(14)]{F80} proved that the distribution kernel $E(x,t;y)$ of $e^{-itH}$ has the following structure, at least for $t\in (0,\delta)$ with a sufficiently small $\delta$. \\
    
    Let $(x(k,s;y),p(k,s;y))$ be the solution to the Hamiltonian system corresponding to \eqref{Schro},
	\begin{align}\label{Newton}
		\begin{cases}
			\dot{x}(t)= p(t),\quad \dot{p}(t)= -\nabla_x V(x(t)), & t\in (0,\delta),  \\
			x(0)= y,\quad p(0)= k.
		\end{cases}
	\end{align}
	The mapping $k\mapsto x(k,t;y)$ is a global diffeomorphism for every $y\in\R^n$ and every $t$. Therefore, for given $x$ and $t$, we can uniquely select $k= k(x,t;y)$ such that the solution of \eqref{Newton} satisfies $x(k,t;y)= x$. With this solution, we define
	\begin{align}
		S(x;t,y)= \int_0^t \bigg[\frac12 \dot{x}(s)^2- V(x(s))\bigg] \text{d}s.
 	\end{align}
	The distribution kernel can then be written in the form
	\begin{align}
		E(x,t;y)= \frac{1}{(2\pi it)^{\frac{n}{2}}} e^{iS(x;t,y)} a(t,x,y).
	\end{align}

    \medskip
	
	Our aim is to determine the potentials $V(x)$ for which the phase function $S(x,t;y)$ satisfies Bourgain's condition. In \cite[Proposition 1.3]{F80}, it was shown that
	\begin{align}
		\nabla_x S(x,t;y)= p(k(x,t;y),t;y),\quad
		\nabla_y S(x,t;y)= -k(x,t;y).
	\end{align}
	Let us analyze the characteristic curves of $S$. Consider
	\begin{align}\label{251211e00}
		\nabla_y S(x,t;y)= -w.
	\end{align}
	This implies $k(x,t;y)=w$. Since $k(x,t;y)$ refers to the $k$ forcing $x(k,t;y)= x$, we have $x(w,t;y)= x$. Substituting this back into \eqref{251211e00} yields
	\begin{align}
		\nabla_y S(x(w,t;y),t;y)= -w.
	\end{align}
    Thus, the characteristic curves are parametrized as $(x(w,t;y),t)$.
	
	\begin{theorem}\label{SchroBour}
        Assume that $n\geq 2$ and the normalization $V(0) = \nabla V(0) = 0$. 
		Let $x(k,t;y)$ solve the equation \eqref{Newton}. Then $x(k,t;y)$ satisfies Bourgain's condition in Definition \ref{BourCond} if and only if 
		\begin{align}
			V(x)= c\, |x|^2,
		\end{align}
        for some constant $c$.
	\end{theorem}

    \begin{proof}
        Integrating the differential equation \eqref{Newton} with respect to $t$, we obtain an integral equation:
	\begin{align}\label{NewtonInt}
		\begin{cases}
			x(t)= y+ \int_0^t p(s)\,\text{d}s, \\
			p(t)= k- \int_0^t \nabla_x V(x(s)) \,\text{d}s.
		\end{cases}
	\end{align}
	   To check whether $x(k,t;y)$ satisfies Bourgain's condition,  we need to examine the rank of the matrix:
	\begin{align}
		\begin{bmatrix}
			\nabla_k x & \nabla_y x  & 0  \\
			\partial_t\nabla_k x & \partial_t\nabla_y x & \nabla_k x  \\
			\partial_t^2\nabla_k x & \partial_t^2\nabla_y x  
			& c \nabla_k x+ 2\ \partial_t\nabla_k x 
		\end{bmatrix}.
	\end{align}
	Note that $\partial_t^2 x(k,t;y)= -\nabla_x V(x(k,t;y))$. Taking $\nabla_k$ and $\nabla_y$ implies 
	\begin{align}
		\partial^2_t\nabla_k x(k,t;y)= -\nabla^2_x V(x(k,t;y))\nabla_k x(k,t;y), \label{251210eC-14}\\	
		\partial^2_t\nabla_y x(k,t;y)= -\nabla^2_x V(x(k,t;y))\nabla_y x(k,t;y),
	\end{align}
	respectively. After row transformations, the matrix becomes
	\begin{align}
		\begin{bmatrix}
			\nabla_k x & \nabla_y x  & 0  \\
			\partial_t\nabla_k x & \partial_t\nabla_y x & \nabla_k x  \\
			0 & 0  
			& c \nabla_k x+ 2\ \partial_t\nabla_k x 
		\end{bmatrix}.
	\end{align}
	From \eqref{NewtonInt}, a direct calculation shows the following asymptotic expansion for small $t$:
	\begin{align}
		\nabla_k x=O(t),\ \nabla_y x= I_n+O(t),\ \partial_t\nabla_k x= I_n+ O(t),\ \partial_t\nabla_y x= O(t).
	\end{align}
	Thus, for small $t$, $x(k,t;y)$ satisfies H\"ormander's condition \eqref{le1.12}, i.\,e., 
	\begin{align}
		\begin{bmatrix}
			\nabla_k x & \nabla_y x   \\
			\partial_t\nabla_k x & \partial_t\nabla_y x
		\end{bmatrix}
	\end{align}
	is non-degenerate. It then follows from Item \ref{thm1_13item3} in Theorem \ref{250727theorem1_13} that $x$ satisfies Bourgain's condition if and only if there exists a scalar
	$c=c(k,t;y)$ such that
	\begin{align}\label{251210eC-19}
		\partial_t\nabla_k x = c\nabla_k x.
	\end{align}
    If the potential $V$ has the form $c|x|^2$, then $x(k,t;y)$ can be explicitly solved, and a direct calculation confirms that \eqref{251210eC-19} is satisfied. Conversely, taking $\partial_t$ on both sides of \eqref{251210eC-19} yields
	\begin{align}
		\partial_t^2 \nabla_k x= (\partial_t c+c^2)  \nabla_k x.
	\end{align}
	Comparing this result with \eqref{251210eC-14}, we conclude that $\nabla^2_x V(x(k,t;y))$ is a multiple of the identity matrix $I_n$. Since $k$, $y$ and $t$  can be chosen arbitrarily,  we see that $V$ must be of form $c\, |x|^2$. This finishes the proof.
    
    \end{proof}

	\section{An appendix by Terence Tao: Theorem \ref{B.C.Tao}}\label{Tao's note}

\subsection{Preliminaries in differential geometry}

All manifolds in this section are assumed to be smooth, real, finite-dimensional, and without boundary.  All functions, forms, vector fields, etc., will also be assumed to be smooth.  If $M$ is a smooth manifold, $C^\infty(M)$ will denote the smooth scalar functions $f: M \to \R$ on $M$, which is a commutative real algebra; $TM$ and $T^*M$ denote the tangent and cotangent bundles, and $\Gamma(TM)$ denotes the space of vector fields of $M$.  Algebraically, one can view $\Gamma(TM)$ as the space of \emph{derivations} $X \colon C^\infty(M) \to C^\infty(M)$, that is to say linear maps that obey the Leibniz rule
$$ X(fg) = f X(g) + g X(f)$$
for all $f,g \in C^\infty(M)$.  In particular $X1=0$ and hence $Xc=0$ for any constant $c \in \R$.  The space $\Gamma(TM)$ is both a $C^\infty(M)$-module and a Lie algebra with the usual commutator bracket $[X,Y] := XY - YX$.\\

We will be working on manifolds in a local fashion.  What this means in practice is that every smooth manifold $M$ that we work with implicitly comes with an origin $0_M$, and we reserve the right to restrict the manifold $M$ to a sufficiently small open neighbourhood of $0_M$ whenever we wish; we refer to this operation as ``working locally''.  Thus, for instance, working locally, every $d$-dimensional $M$ can be identified with an open simply connected subset of $\R^d$, and in particular can be given a set of coordinate functions $x^1,\dots,x^d \in C^\infty(M)$ and partial derivatives $\partial_{x^1},\dots,\partial_{x^d} \in \Gamma(TM)$ obeying the relations
$$ [\partial_{x^i}, \partial_{x^j}] = 0; \quad \partial_{x^i} x^j = \delta_i^j$$
for $i,j=1,\dots,d$, where $\delta$ denotes the Kronecker delta.  It is a standard fact that (working locally) $\Gamma(TM)$ is a free module with $\partial_{x_1},\dots,\partial_{x_d}$ as a basis, or in other words every $X \in \Gamma(TM)$ can be uniquely represented in the form
$$ X = X^i \partial_{x^i}$$
for some $X^i \in C^\infty(M)$ for $i=1,\dots,d$, where we adopt the usual summation conventions.  Indeed one can take $X^i = X(x^i)$ for $i=1,\dots,d$.  At any point $x_0 \in M$ with coordinates $x^i_0$, and any $f \in C^\infty(M)$, one has the Taylor expansion
$$ f(x) = f(x_0) + (x^i - x^i_0) \partial_{x^i} f(x_0) + (x^i - x^i_0) (x^j - x^j_0) f_{ij}(x)$$
for some $f_{ij} \in C^\infty(M)$, so from several applications of the Leibniz rule one concludes that
$$ Xf(x_0) = X^i(x_0) \partial_{x^i} f(x_0)$$
as claimed. In coordinates, the Lie bracket $[X,Y]$ of two vector fields $X = X^i \partial_{x^i}$ and $Y = Y^i \partial_{x^i}$ can be computed as
$$ [X,Y] = (X^j \partial_{x^j} Y^i - Y^j \partial_{x^j} X^i) \partial_{x^i}.$$

\smallskip

A \emph{fiber bundle} is a pair of manifolds $E, M$ of dimensions $d+m$, $d$ respectively for some $d,m \geq 0$, linked by a smooth surjective map $\pi: E \to M$ which is locally trivial.  What this means is that, working locally, one can view $M$ as an open subset of $\R^d$, and $E$ as an open subset of $\R^{d+m}$, such that $\pi$ is the coordinate map $\pi(x,y) := x$; we refer to such a coordinatization as a \emph{(local) trivialization} of the bundle.
By abuse of notation, we will sometimes refer to the projection map $\pi: E \to M$ as the bundle, with $M$ referred to as the \emph{base space} and $E$ the \emph{total space}; the $m$-dimensional spaces $\pi^{-1}(\{x\})$ are known as \emph{fibers}.  From the inverse function theorem, we see that a smooth map $\pi: E \to M$ locally gives rise to a fiber bundle if and only if it is a submersion, that is to say its derivative (which in coordinates is an $d \times (d+m)$ matrix) is of full rank $d$. A \emph{bundle map} $T: E \to E'$ between two fiber bundles
$\pi: E \to M$, $\pi': E' \to M$ over the same base $M$ is a smooth map such that $\pi' \circ T = \pi$, that is to say it maps fibers to fibers.\\

A fiber bundle induces an embedding $C^\infty(M) \hookleftarrow C^\infty(E)$ that maps $f$ to $f \circ \pi$; by abuse of notation we will view this embedding as an inclusion, thus we will view functions in $C^\infty(M)$ as also being functions on $C^\infty(E)$.  A vector field $Z \in \Gamma(TE)$ is said to be \emph{vertical} if it annihilates $C^\infty(M)$; the space of such vector fields is a $C^\infty(M)$-submodule of $\Gamma(TE)$ and will be denoted $\Gamma_v(TE)$.  If we introduce local coordinates $x^i$, $i=1,\dots,d$ on $M$ and $x^i, y^\alpha$, $i=1,\dots,d$, $\alpha=1,\dots,m$ on $E$ corresponding to a local trivialization, then an arbitrary vector field in $TE$ takes the form
$$ Z = X^i \partial_{x^i} + Y^\alpha \partial_{y^\alpha}$$
for some $X^i, Y^\alpha \in C^\infty(E)$ (where Roman indices range over $1,\dots,d$ and Greek indices range over $1,\dots,m$), and $Z$ is vertical if and only if the $X_i$ all vanish.\\

An \emph{Ehresmann connection}, or \emph{connection} for short, on this fiber bundle is a $C^\infty(M)$-linear map $\nabla \colon \Gamma(TM) \to \Gamma(TE)$ that maps each vector field $X$ on $M$ to a vector field $\nabla_X$ on $E$ which obeys the compatibility condition
$$ \nabla_X f = X f $$
for all $f \in C^\infty(M)$.  In local coordinates, the connection takes the form
$$ \nabla_X = X^i \partial_{x^i} + X^j \Gamma^\alpha_j \partial_{y^\alpha}$$
for some \emph{Christoffel symbols} $\Gamma^\alpha_j \in C^\infty(TE)$, which are given by the formula
$$ \Gamma^\alpha_j = \nabla_{\partial_{x^j}} y^\alpha.$$
If $\nabla,\nabla'$ are two connections, then $\nabla'-\nabla$ is a $C^\infty(M)$-linear map from $\Gamma(TM)$ to $\Gamma_v(TE)$.\\

A connection $\nabla$ is \emph{flat} if we have the identity
\begin{equation}\label{flat}
    [\nabla_X, \nabla_Y] = \nabla_{[X,Y]}
\end{equation}
for all $X,Y \in \Gamma (TM)$.  For instance, every local trivialization gives rise to a flat connection $\nabla^0_X$ given by setting all the Christoffel symbols $\Gamma^\alpha_j$ to zero:
$$ \nabla^0_X = X^i \partial_{x^i}.$$
Conversely, every flat connection (locally) arises from a local trivialization in this fashion; this is a well-known theorem of Frobenius (the flat connection condition is equivalent to the vector fields $\nabla_X$ being integrable).  In terms of Christoffel symbols, the flatness condition can be expressed as
$$ \partial_{x^i} \Gamma^\alpha_j - \partial_{x^j} \Gamma^\alpha_i + \Gamma^\beta_i \partial_{y^\beta} \Gamma^\alpha_j - \Gamma^\beta_j \partial_{y^\beta} \Gamma^\alpha_i = 0.$$

\smallskip

Two important bundles over a manifold $M$ are the \emph{tangent bundle} $TM$ and the \emph{cotangent bundle}.  In local coordinates $x^i$ for $M$, a point in the tangent bundle $TM$ takes the form
$$ ( (x^i)_{i=1}^d, v^i \partial_{x^i} ) $$
for some real numbers $x^i, v^i$, and a point in the cotangent bundle $T^* M$ takes the form
$$ ( (x^i)_{i=1}^d, p_i dx^i ) $$
for some real numbers $x^i, p_i$.  These bundles project to $M$ in the obvious fashion. A vector field $X = X^i \partial_{x^i}$ can be viewed as a map $X : M \to TM$ from a point $x$ to a tangent vector $X(x) = X^i(x) \partial_{x^i}$; this is a bundle map over $M$, viewing $M$ as a trivial bundle over itself.  The vector field $X$ can also be viewed as an element $l_X$ of $C^\infty(T^* M)$ by the formula
$$ l_X( (x, p_i dx^i) ) = X^i(x) p_i.$$
Note that the map $l: X \mapsto l_X$ is linear over $C^\infty(M)$.\\

Every scalar function $f \in C^\infty(M)$ gives rise to a \emph{differential} $df: M \to T^* M$, defined in local coordinates by
$$ df(x) = (x, \partial_{x^i} f(x) dx^i),$$
or via duality as
\begin{equation}\label{lxdf}
    l_X(df) = Xf
\end{equation}
for all $X \in \Gamma(TM)$. This is a bundle map over $M$.

\subsection{Bourgain's condition}

Now let $M$ and $\Sigma$ be $d$ and $d-1$-dimensional manifolds for some $d \geq 3$, and let $\phi \in C^\infty(M \times \Sigma)$ be a ``phase function'' $(x,\xi) \mapsto \phi(x,\xi)$.  For each $x \in M$, we can view $\xi \mapsto \phi(x,\xi)$ as an element of $C^\infty(\Sigma)$; the differential $\xi \mapsto d_\xi \phi(x,\xi)$ of this map is then a bundle map from $\Sigma$ to $T^* \Sigma$ over $\Sigma$.  We can thus view $d_\xi \phi$ as a bundle map (over $\Sigma$) from $M \times \Sigma$ to $T^* \Sigma$.  In local coordinates,
$$ d_\xi \phi(x,\xi) = (\xi, \partial_{\xi^i} \phi(x,\xi) d\xi^i)$$
(where now Roman indices range in $1,\dots,d-1$).
We say that $\phi$ obeys axiom (H1) if $d_\xi \phi: M \times \Sigma \to T^* \Sigma$ is a submersion near the origin.  In coordinates, this amounts to requiring that the $d \times (d-1)$ matrix $\partial_{x^\alpha} \partial_{\xi^i} \phi(x,\xi)$ has full rank $d-1$ at the origin (and hence in a sufficiently small neighbourhood).  We now assume (H1) henceforth; thus $M \times \Sigma$ is now locally a fiber bundle over $T^* \Sigma$, and thus can be locally trivialized as $T^*\Sigma \times \R$.  In particular, we now have an identification of $M \times \Sigma$ with $T^* \Sigma \times \R$ as fiber bundles over $\Sigma$.  The trivial connection $\nabla^0$ of $M \times \Sigma$ (viewed as a $\Sigma$-bundle) then gives rise to a flat connection $\nabla$ of $T^* \Sigma \times \R$ (also viewed as a $\Sigma$-bundle).  The function $\phi$ can now also be viewed as an element of $C^\infty(T^* \Sigma \times \R)$.  Meanwhile, every vector field $X$ in $\Gamma(T\Sigma)$ gives rise to a function $l_X$ in $C^\infty(T^*\Sigma)$, which by abuse of notation we can also view as an element of $C^\infty(T^* \Sigma \times \R)$.  From \eqref{lxdf} and unfolding all the definitions we see that
\begin{equation}\label{nx}
    \nabla_X \phi = l_X
\end{equation}
for all $X \in \Gamma(T\Sigma)$.  From the flatness \eqref{flat} of $\nabla$, we then have the \emph{torsion-free condition}
\begin{equation}\label{lxy}
     \nabla_X l_Y - \nabla_Y l_X = l_{[X,Y]}
\end{equation}
for all $X,Y \in \Gamma(T\Sigma)$.  In local coordinates, this is just Clairaut's theorem
$$ \partial_{\xi^i} \partial_{\xi^j} \phi = \partial_{\xi^j} \partial_{\xi^i} \phi.$$

\smallskip

Conversely, suppose one has a flat connection $\nabla$ on $T^* \Sigma \times \R$ (viewed as a $\Sigma$-bundle) obeying \eqref{lxy} for all $X,Y \in \Gamma(T\Sigma)$.  Then we claim that locally there is a $d$-dimensional manifold $M$ and a function $\phi: C^\infty(M \times \Sigma)$ that generates $\nabla$ by the above procedure.  Indeed, since $\nabla$ is flat, it can be locally identified with the trivial connection $\nabla^0$ in a trivial bundle $\R^d \times \Sigma$.  If we write the identification of $\R^d \times \Sigma$ with $T^* \Sigma \times \R$ in local coordinates as
$$(x,\xi) \mapsto ((\xi, p_i(x,\xi) d\xi^i), c(x,\xi))$$
then by the chain rule the connection $\nabla$ is given by
\begin{align*}
    \nabla_X F((\xi, p_i(x,\xi) d\xi^i), c(x,\xi)) &= X^j(\xi) \partial_{\xi^j} F((\xi, p_i(x,\xi) d\xi^i), c(x,\xi)) \\
&\quad + (X^j(\xi) \partial_{\xi^j} p_i(x,\xi)) \partial_{p_i} F((\xi, p_i(x,\xi) d\xi^i), c(x,\xi)) \\
&\quad + (X^j(\xi) \partial_{\xi^j} c(x,\xi)) \partial_c F((\xi, p_i(x,\xi) d\xi^i), c(x,\xi))
\end{align*}
for any $F \in C^\infty(T^* \Sigma \times \R)$.
The condition \eqref{lxy} then becomes the curl-free condition
$$ \partial_{\xi^j} p_i(x,\xi) = \partial_{\xi^i} p_j(x,\xi).$$
Since we can locally work in a simply connected domain, we can integrate this condition and write
$$ p_i(x,\xi) = \partial_{\xi^i} \phi(x,\xi)$$
for some $\phi \in C^\infty(M \times \Sigma)$.  We thus see that the projection of $\R^d \times \Sigma \equiv T^* \Sigma \times \R$ to $T^*\Sigma$ is given by $d_\xi \phi$, which thus has full rank $d-1$, so that (H1) holds.   The identification $\R^d \times \Sigma \equiv T^* \Sigma \times \R$ is now a trivialization of the bundle $d_\xi \phi: M \times \Sigma \to T^* \Sigma$, and the claim follows.\\

In view of this equivalence, we now expect many conditions on the phase $\phi$ to be expressible in terms of conditions on flat connections $\nabla$ on $T^* \Sigma \times \R$ that obey the torsion-free condition \eqref{lxy}; in particular, we expect that the manifold $M$ no longer plays an explicit role. Note though that we have the freedom to change the trivialization of $T^* \Sigma \times \R$ (as a bundle over $T^*\Sigma$).  In coordinates $((\xi,p),c)$, this amounts to making a change of variables
$$ ((\xi,p),c) \mapsto ((\xi,p),c'((\xi,p),c))$$
where $c \mapsto c'((\xi,p),c)$ is a local diffeomorphism for $(\xi,p) \in T^*\Sigma$ sufficiently close to the origin.\\

Once one fixes a trivialization $T^*\Sigma \times \R$ over $T^* \Sigma$, we get a vertical vector field $\partial_c$ corresponding to the coordinate function $c: ((\xi,p),c) \mapsto c$.  For any $X \in \Gamma(T\Sigma)$, the function $l_X$ does not depend on $c$, hence
$$ \partial_c l_X = 0$$
and hence by \eqref{nx}
$$ \partial_c \nabla_X \phi = 0.$$
As a consequence of this, \eqref{flat}, and the Leibniz rule, the expression
$$\partial_c \nabla_X \nabla_Y \phi$$
is a symmetric $C^\infty(\Sigma)$-bilinear form in $X,Y$.  We say that condition (H2) holds if this form is non-degenerate, that is to say for every $((\xi,p),c) \in T^*\Sigma \times \R$ near the origin, and any $X \in \Gamma(T\Sigma)$ that does not vanish at $\xi$, there exists $Y \in \Gamma(T\Sigma)$ such that $\partial_c \nabla_X \nabla_Y \phi$ does not vanish at $((\xi,p),c)$.  In local coordinates, this is equivalent to the $d-1 \times d-1$ matrix
$$\partial_{\xi_i} \partial_{\xi_j} (G_0(x,\xi)^\alpha \partial_{x^\alpha} \phi(x,\xi))$$
having non-zero determinant, where $G_0(x,\xi)$ is the wedge product of the vectors $\partial_{\xi_i}\nabla_x \phi(x,\xi)$ for $i=1,\dots,d-1$.  Thanks to \eqref{nx}, we can express the condition (H2) purely in terms of the connection $\nabla$ as the requirement that the bilinear form
$$ X, Y \mapsto \partial_c \nabla_X l_Y = \partial_c \nabla_Y l_X$$
is non-degenerate.\\

We now say that the \emph{Bourgain condition} holds if the matrix
$$\partial_{\xi_i} \partial_{\xi_j} (G_0(x,\xi)^\alpha \partial_{x^\alpha})^2 \phi(x,\xi)$$
is a scalar multiple of
$$\partial_{\xi_i} \partial_{\xi_j} (G_0(x,\xi)^\alpha \partial_{x^\alpha}) \phi(x,\xi)$$
for all $x,\xi$ close to the origin.  By the above discussion, this is (locally) equivalent to having a relation of the form
$$ \partial_c^2 \nabla_X l_Y = a \partial_c \nabla_X l_Y$$
for all $X,Y \in \Gamma(T\Sigma)$ and some $a \in C^\infty(T^*\Sigma \times \R)$.  By the fundamental theorem of calculus, we may write $a = \partial_c u$ for some $u \in C^\infty(T^*\Sigma \times \R)$, and then we can write the Bourgain condition in integrating factor form as
$$ \partial_c (e^{-u} \partial_c \nabla_X l_Y) = 0,$$
thus $e^{-u} \partial_c \nabla_X l_Y$ is independent of the $c$ variable.  By integrating $e^{-u}$, we may reparameterize the trivialization $T^*\Sigma \times \R$ by replacing $c$ with a reparameterized coordinate $c'$ for which $\partial_{c'} = e^{-u} \partial_c$.  Thus, without loss of generality, we may assume $u=0$, and normalize the Bourgain condition as
$$ \partial_c^2 \nabla_X l_Y = 0.$$
In particular, the form $\partial_c \nabla_X l_Y$ is independent of $c$, and can thus be written as $B(X,Y)$ for some symmetric non-degenerate $C^\infty(\Sigma)$-bilinear form $B: \Gamma(T\Sigma) \times \Gamma(T\Sigma) \to C^\infty(T^*\Sigma)$:
\begin{equation}\label{bxy}
B(X,Y) = \partial_c \nabla_X l_Y = \partial_c \nabla_Y l_X.
\end{equation}
With this normalization, we still have the freedom to make \emph{affine} reparameterizations of the trivialization $T^*\Sigma \times \R$ by replacing $c$ with $ac + b$ for some $a,b \in C^\infty(T^*\Sigma)$ with $a$ non-vanishing, which has the effect of replacing $B$ by $a^{-1} B$.\\

We write the flat connection $\nabla$ in coordinates $((\xi,p),c)$ as
$$ \nabla_X = X^i(\xi) \partial_{\xi^i} + X^i(\xi) A_{ij}((\xi,p),c) \partial^{p_j} + X^i(\xi) \omega_i((\xi,p),c) \partial_c$$
for some $A_{ij}, \omega_i \in C^\infty(T^*\Sigma \times \R)$, which can be determined by the formulae
$$ X^i A_{ij} = \nabla_X p_j$$
and
$$ X^i \omega_i = \nabla_X c.$$
Then
$$ \partial_c \nabla_X l_Y((\xi,p),c) = X^i(\xi) Y^j(\xi) \partial_c A_{ij}((\xi,p),c).$$
Meanwhile, $B$ can be expressed in coordinates as
$$ B(X,Y)(\xi,p) = X^i(\xi) Y^j(\xi) B_{ij}(\xi,p)$$
where the $B_{ij} \in C^\infty(T^*\Sigma)$ form a symmetric non-degenerate matrix.  From \eqref{bxy} we thus have
$$ \partial_c A_{ij}((\xi,p),c) =B_{ij}(\xi,p)$$
and thus we may write
$$ A_{ij}((\xi,p),c) = A_{ij}(\xi,p) + c B_{ij}(\xi,p)$$
for some $A_{ij} \in C^\infty(T^*\Sigma)$.  The torsion-free condition \eqref{lxy} then becomes a symmetry condition on $A$:
$$ A_{ij} = A_{ji}.$$
If we now introduce the reduced connection $\nabla'$ on $T^* \Sigma$ by the formula
$$ \nabla'_X = X^i(\xi) \partial_{\xi^i} + X^i(\xi) A_{ij}(\xi,p) \partial^{p_j}$$
(which by abuse of notation we identify with a connection on $T^* \Sigma \times \R$ in the obvious fashion), then we have the torsion-free condition
\begin{equation}\label{torsionfree}
\nabla'_X l_Y = \nabla'_Y l_X
\end{equation}
for all $X,Y \in \Gamma(T\Sigma)$, as well as the relation
\begin{equation}\label{nablax}
    \nabla_X = \nabla'_X + c E_X + \omega(X) \partial_c
\end{equation}
where $E_X \in \Gamma_v(T(T^* \Sigma))$ denotes the vertical vector field
$$ E_X := X^i B_{ij} \partial^{p_j}$$
and $\omega(X)$ denotes the scalar
$$ \omega(X) := X^i \omega_i.$$
Both of these expressions depend in a $C^\infty(\Sigma)$-linear fashion on $X$. Also from the symmetry of $B$ we have
\begin{equation}\label{E-torsionfree}
E_X l_Y = E_Y l_X.
\end{equation}

\smallskip

We assume $X=\partial_{\xi_k}$ and $Y=\partial_{\xi_i}$ hereafter and then we can use the linearity for general $X$ and $Y$. The flatness condition \eqref{flat} now expands to 
\begin{align*}
    &[\nabla'_X,\nabla'_Y] - \nabla'_{[X,Y]} + c ([\nabla'_X,E_Y] - [\nabla'_Y, E_X]) + c^2 [E_X,E_Y] \\
    &\quad + (\omega(X) E_Y - \omega(Y) E_X) + (\nabla_X \omega(Y) - \nabla_Y \omega(X)) \partial_c = 0.
\end{align*}
Extracting the $\partial_c$ coefficient we thus see that
\begin{equation}\label{nabla-torsion}
    \nabla_X \omega(Y) = \nabla_Y \omega(X)
\end{equation}
and
\begin{multline}\label{nab-com}
    [\nabla'_X,\nabla'_Y] - \nabla'_{[X,Y]} + c ([\nabla'_X,E_Y] - [\nabla'_Y, E_X])\\
    + c^2 [E_X,E_Y] + (\omega(X) E_Y - \omega(Y) E_X) = 0.
\end{multline}
The vector fields $[\nabla'_X,\nabla'_Y] - \nabla'_{[X,Y]}$ and $[\nabla'_X,E_Y] - [\nabla'_Y, E_X]$ are $c$-invariant.  Taking two Lie derivatives in $\partial_c$ (or equivalently, two Lie brackets with $\partial_c$) we conclude that
\begin{equation}\label{2xy}
 2 [E_X,E_Y] + {\mathcal L}_{\partial_c}^2 (\omega(X) E_Y - \omega(Y) E_X) = 0.
\end{equation}

\smallskip

From \eqref{nab-com} we see that the expression $\omega(X) E_Y - \omega(Y) E_X$ is quadratic in $c$ for any fixed $X,Y$ (and also fixing the element $(\xi,p)$ of $T^* \Sigma$).  If $d \geq 3$, then for any $X$ we can always choose $E_Y$ to be independent of $E_X$ at a given point $(\xi,p)$, since the non-degeneracy of $B$ means that the $E_Y$ span the vertical tangent space.  Thus we see that each $\omega(X)$ is separately quadratic in $c$.  That is to say, we have an expansion of the form
$$ \omega(X) = \omega_0(X) + c \omega_1(X) + c^2 \omega_2(X)$$
for some $\omega_0(X), \omega_1(X), \omega_2(X) \in C^\infty(T^* \Sigma)$.
We can eliminate $\omega_2$ and then $\omega_1$ as follows.  From \eqref{nabla-torsion}, \eqref{nablax} we have
\begin{equation}\label{naba-2}
\begin{split}
    &    (\nabla'_X + c E_X + (\omega_0(X) + c \omega_1(X) + c^2 \omega_2(X)) \partial_c) (\omega_0(Y) + c \omega_1(Y) + c^2 \omega_2(Y))\\
&\quad     =
(\nabla'_Y + c E_Y + (\omega_0(Y) + c \omega_1(Y) + c^2 \omega(Y)) \partial_c) (\omega_0(X) + c \omega_1(X) + c^2 \omega_2(X));
\end{split}
\end{equation}
extracting the $c^3$ coefficient and canceling terms we conclude that
$$ E_X \omega_2(Y) = E_Y \omega_2(X).$$
Similarly, from extracting the $c^2$ coefficient from \eqref{nab-com} we have
\begin{equation}\label{xye-2}
    [E_X,E_Y] = \omega_2(X) E_Y - \omega_2(Y) E_X.
\end{equation}
If we write in coordinates
$$ \omega_2(X) = X^i B_{ij}  \omega_2^j$$
for some $\omega_2^j \in C^\infty(T^* \Sigma)$, we thus have
$$ B_{kl} \partial^{p_l} (B_{ij}  \omega_2^j) = B_{il} \partial^{p_l} (B_{kj}  \omega_2^j)$$
while from \eqref{xye-2} we have
$$ B_{kl} \partial^{p_l} B_{ij} - B_{il} \partial^{p_l} B_{kj} = B_{km} \omega_2^m B_{ij} - B_{im} \omega_2^m B_{mj} $$
so from the Leibniz rule and some canceling and relabeling
$$ B_{kl} B_{ij}  (\partial^{p_l} \omega_2^j - \partial^{p_j} \omega_2^l)$$
and thus by the nondegeneracy of $B$
$$ \partial^{p_l}  \omega_2^j = \partial^{p_j}  \omega_2^l$$
and hence we may locally integrate
$$ \omega_2^i = \partial^{p_i} u$$
for some $u \in C^\infty(T^* \Sigma)$.
If we change variables $c \mapsto e^u c$, then $B$ maps to $e^{-u} B$ and hence $E_X$ maps to $e^{-u} E_X$, and from the product rule we then see that \eqref{xye-2} simplifies to
\begin{equation}\label{exey}
    [E_X,E_Y] = 0
\end{equation}
and $\omega_2$ now vanishes. Thus we may normalize so that \eqref{exey} holds and $\omega_2=0$.  By doing so, we no longer have the freedom to perform arbitrary affine reparameterizations $c \mapsto ac+b$ of $c$, but we can still perform translational reparameterizations $c \mapsto c+b$ for $b \in C^\infty(T^* \Sigma)$; these do not affect $B$ or $E$, but replace $\omega(X)$ with $\omega(X) - \nabla_X b = \omega(X) - \nabla'_X b - c E_X b$.\\

We can now eliminate $\omega_1$.  The equation \eqref{naba-2} has now simplified to
\begin{align}\label{naba}
    (\nabla'_X + c E_X + 
    & (\omega_0(X) + c \omega_1(X)) \partial_c) (\omega_0(Y) + c \omega_1(Y))  \\ 
    &=(\nabla'_Y + c E_Y + (\omega_0(Y) + c \omega_1(Y)) \partial_c) (\omega_0(X) + c \omega_1(X)),\notag
\end{align}
and thus on extracting the $c^2$ coefficient
$$ E_X \omega_1(Y) = E_Y \omega_1(X).$$
As before this can be integrated to give
$$ \omega_1(X) = E_X b$$
for some $b \in C^\infty(T^* \Sigma)$.  Performing the reparameterization $c \mapsto c+b$ as defined above, we may thus normalize $\omega_1$ to be zero.  Thus $\omega = \omega_0$ is now independent of $c$, and \eqref{naba} simplifies further to
$$ (\nabla'_X + c E_X) \omega(Y) =
(\nabla'_Y + c E_Y) \omega(X).$$
Extracting the $c$ coefficient we obtain
$$ E_X \omega(Y) = E_Y \omega(X)$$
so once again we may integrate to obtain
$$ \omega(X) = E_X J$$
for some $J \in C^\infty(T^*\Sigma)$.  The condition \eqref{naba} now simplifies further to
\begin{equation}\label{naba-0}
\nabla'_X E_Y J = \nabla'_Y E_X J.
\end{equation}
From this and \eqref{exey}, the condition \eqref{nab-com} has simplified to
$$ [\nabla'_X,\nabla'_Y] - \nabla'_{[X,Y]} + c ([\nabla'_X,E_Y] - [\nabla'_Y, E_X]) + ((E_X J) E_Y - (E_Y J) E_X) = 0$$
and so on comparing coefficients in $c$ we obtain the curvature condition
\begin{equation}\label{curv}
    [\nabla'_X,\nabla'_Y] - \nabla'_{[X,Y]} + ((E_X J) E_Y - (E_Y J) E_X) = 0
\end{equation}
and the symmetry condition
\begin{equation}\label{symm}
    [\nabla'_X,E_Y] = [\nabla'_Y, E_X].
\end{equation}

\smallskip

Conversely, suppose that we have a connection $\nabla'$ on $T^* \Sigma$, a function $J \in C^\infty(T^* \Sigma)$, and a $C^\infty(\Sigma)$-linear map $X \mapsto E_X$ from $\Gamma(T\Sigma)$ to $\Gamma_v(T(T^*\Sigma))$ which is non-degenerate and obeys the conditions \eqref{curv}, \eqref{symm}, \eqref{exey}, \eqref{torsionfree}, \eqref{E-torsionfree}, \eqref{naba-0}.  Then, by reversing the above arguments, we see that the connection $\nabla$ defined by
\begin{align}\label{260604leD-19}
    \nabla_X := \nabla'_X + c E_X + (E_X J) \partial_c,
\end{align} 
is a flat connection obeying \eqref{lxy} as well as (H2) and the Bourgain condition with $B$ given by the formula
$$ B(X,Y) = E_X l_Y = E_Y l_X.$$

\smallskip

The fact \eqref{exey} that the $E_X$ commute with each other allows one to find a coordinate system for the bundle $T^*\Sigma$ in which the $E_X$ are coordinate vector fields.  Indeed, by exponentiating, we get (local) bundle diffeomorphisms $\exp(E_X): T^*\Sigma \to T^*\Sigma$ for each $X \in \Gamma(T\Sigma)$, with the value of $\exp(E_X)(\xi,p)$ depending only on $X$ through its value $X(\xi)$ at $\xi$.  If we arbitrarily pick a smooth section $\omega_0 \in \Gamma(T^*\Sigma)$ (i.e., a $1$-form), we can then produce (locally) a bundle diffeomorphism $\psi : T \Sigma \to T^* \Sigma$ by the formula
$$ \psi( (\xi, X(\xi)) ) = \exp(E_X)(\xi,\omega_0(\xi))$$
for any $\xi \in \Sigma$ and vector field $X \in \Gamma(T\Sigma)$.  Note that every vector field $X \in \Gamma(T\Sigma)$ canonically induces a vertical vector field $V_X \in \Gamma_v(T(T\Sigma))$, which can be uniquely defined by the formula
$$ V_X \omega = \omega(X)$$
for all $1$-forms $\omega \in \Gamma(T^*\Sigma)$, where we identify such one-forms with linear functionals $(\xi,v) \mapsto \omega(\xi)(v)$ on $T\Sigma$; using local coordinates $(x,v)$ for $T\Sigma$ we have
$$ V_X = X^i(x) \partial_{v^i}.$$
In particular, the $V_X$ commute with each other:
\begin{equation}\label{vx-comm}
[V_X, V_Y] = 0.
\end{equation}
By chasing through all the definitions, we see that $V_X$ is the pullback of $E_X$ by $\psi$: $\psi^* E_X = V_X$.  In particular, \eqref{E-torsionfree} is equivalent to
\begin{equation}\label{vxvy}
    V_X (l_Y \circ \psi) = V_Y (l_X \circ \psi).
\end{equation}
Using local coordinates $(\xi,v)$ for the tangent bundle $T\Sigma$, this relation can be written as
$$ \partial_{v^i} \psi_j = \partial_{v^j} \psi_i.$$
We can then locally integrate this as
\begin{align}\label{260604leD-21}
    \psi_j = \partial_{v^j} F^*
\end{align} 
or in coordinate free notation
\begin{equation}\label{lxf}
     l_X \circ \psi = V_X F^*
\end{equation}
for all $X \in \Gamma(T\Sigma)$ and some $F^* \in C^\infty(T\Sigma)$.  As $\psi$ is a diffeomorphism, $F^*$ must then have non-degenerate vertical Hessian, in that the symmetric bilinear form $(X,Y) \mapsto V_Y V_X F^*$ is non-degenerate.  (One could similarly integrate the inverse of $\psi$ to obtain the Legendre transform $F$ of $F^*$, as in previous sections, but we will not need to do so here.)\\

Using $\psi$, we can pull back the connection $\nabla'$ on $T^* \Sigma$ to a connection $\nabla^*$ on $T\Sigma$, and similarly pull back the function $J \in C^\infty(T^* \Sigma)$ to a function $J^* \in C^\infty(T\Sigma)$.  The relations \eqref{curv}, \eqref{symm}, \eqref{naba-0} then pull back to
\begin{equation}\label{curv-pull}
    [\nabla^*_X,\nabla^*_Y] - \nabla^*_{[X,Y]} + ((V_X J^*) V_Y - (V_Y J^*) V_X) = 0,
\end{equation}
\begin{equation}\label{symm-pull}
    [\nabla^*_X,V_Y] = [\nabla^*_Y, V_X],
\end{equation}
and
\begin{equation}\label{naba-pull}
    \nabla^*_X V_Y J^* = \nabla^*_Y V_X J^*.
\end{equation}
respectively.  The relation \eqref{torsionfree} similarly pulls back to
$$
\nabla^*_X (l_Y \circ \psi) = \nabla^*_Y (l_X \circ \psi)$$
and hence by \eqref{lxf}
\begin{equation}\label{torsionfree-pull}
\nabla^*_X V_Y F^* = \nabla^*_Y V_X F^*.
\end{equation}
Thus, intriguingly, $F^*$ obeys the same equation \eqref{naba-pull} as $J^*$!\\

Conversely, suppose one has a connection $\nabla^*$ on $T \Sigma$, and functions $J^*, F^* \in C^\infty(T\Sigma)$ obeying \eqref{curv-pull}, \eqref{symm-pull}, \eqref{naba-pull}, \eqref{torsionfree-pull} with $F^*$ (locally) having non-degenerate vertical Hessian.  The non-degeneracy then allows us to construct a local bundle diffeomorphism $\psi: T\Sigma \to T^*\Sigma$ using \eqref{lxf}, and pushing forward by that diffeomorphism (locally) gives a connection $\nabla'$ on $T^*\Sigma$ and a function $J \in C^\infty(T^*\Sigma)$.  The relations \eqref{curv-pull}, \eqref{symm-pull}, \eqref{naba-pull}, \eqref{torsionfree-pull} then push forward to \eqref{curv}, \eqref{symm}, \eqref{naba-0}, \eqref{torsionfree} respectively, and \eqref{vx-comm} similarly pushes forward to \eqref{exey}.  From \eqref{lxf}, \eqref{vx-comm} we have \eqref{vxvy}, which pushes forward to \eqref{E-torsionfree}.  Thus, by previous discussion we have constructed a solution to the Bourgain condition.\\

To summarize, we have described solutions to the Bourgain condition in terms of just three objects $\nabla^*, J^*, F^*$ defined on the tangent bundle $T\Sigma$, obeying four relations \eqref{curv-pull}, \eqref{symm-pull}, \eqref{naba-pull}, \eqref{torsionfree-pull} and a non-degeneracy hypothesis. \\

\noindent {\bf Acknowledgment.}
S. G. is partly supported by the Nankai Zhide Foundation, NSFC Grant No. 12426204, and the New Cornerstone Science Foundation.
The authors would like to thank Terry Tao for helpful discussions throughout the project, and in particular for suggesting the main idea of the proof of Theorem \ref{thm: main} Part \ref{thm: main part 2}. The proofs of Lemma \ref{lem: lin alg lemma} and Lemma \ref{volsliceslem} were developed with the assistance of ChatGPT Pro 5.5, and were  independently checked by the authors, who take full
responsibility for their correctness. ChatGPT Pro 5.5 was also used for minor editing and proofreading throughout the paper.

\bibliographystyle{alpha}
\bibliography{bibliography}

\end{document}